\documentclass[11pt,a4paper]{amsart}

\usepackage{amsmath,amsfonts,amssymb,amsthm}
\usepackage{txfonts}
\usepackage{latexsym,bm,graphicx}
\usepackage{mathrsfs}
\usepackage{color}

\title[Lipschitz regularity of harmonic map heat flows]{Lipschitz regularity of harmonic map heat flows into $CAT(0)$ spaces}
\author{Hui-Chun Zhang}
\address{Department of Mathematics\\  Sun Yat-sen University\\ Guangzhou 510275\\ \newline E-mail address: zhanghc3@mail.sysu.edu.cn}
\author{Xi-Ping Zhu}
\address{Department of Mathematics\\  Sun Yat-sen University\\ Guangzhou 510275\\ \newline E-mail address: stszxp@mail.sysu.edu.cn}

 \newtheorem{theorem}{Theorem}[section]

\newtheorem{proposition}[theorem]{Proposition}
\newtheorem{lemma}[theorem]{Lemma}
\newtheorem{corollary}[theorem]{Corollary}

\theoremstyle{definition}
\theoremstyle{remark}

\newtheorem{defn}[theorem]{Definition}
\newtheorem{remark}[theorem]{Remark}

\numberwithin{equation}{section}

\newcommand{\ls}{\leqslant}
\newcommand{\gs}{\geqslant}

\newcommand{\ip}[2]{\left<{#1},{#2}\right>}

\newcommand{\dm}{{\rm d}\mu}

\begin{document}


\begin{abstract}

In 1964, Eells and Sampson proved the celebrated long-time existence and convergence for the harmonic map heat flow into non-positively curved Riemannian manifolds. In 1992, Gromov and Schoen initiated the study of harmonic maps into $CAT(0) $ metric spaces. It naturally motivates the study of the harmonic map heat flow into singular metric spaces.

In the 1990s, Mayer and Jost independently studied convex functionals on $CAT(0)$ spaces and extended Crandall-Liggett's theory of gradient flows from Banach spaces to $CAT(0) $ spaces to obtain the weak solutions for the harmonic map heat flow
into $CAT(0)$ spaces.    The weak solutions enjoy the favorable long-time existence, uniqueness and well-established long-time behaviors.  It is a long-standing open question to ask if the weak solutions possess the Lipschitz regularity. Very recently, by using elliptic approximation
method, Lin, Segatti, Sire, and Wang proved the weak solutions are Lipschitz in space and $1\over 2$-H\"older continuous in time, for a wide class of $CAT(0)$ spaces. 

In the present paper, we give a complete answer to the question. We show that every weak solution of the harmonic map heat flow into
$CAT(0)$ spaces is Lipschitz continuous in both space and time. We also establish an Eells-Sampson-type Bochner inequality.
\end{abstract}

\maketitle
  \tableofcontents
  \setcounter{tocdepth}{1}

\section{Introduction}

The harmonic map heat flow deforms a given map $u_0: M\to N\subset \mathbb R^\ell$ between two compact Riemannian manifolds via the equation
\begin{equation}\label{equ-1.1}
\begin{cases}
	\partial_t u&=\Delta u+A(u)(\nabla u,\nabla u),\quad (x,t)\in M\times(0,+\infty),\\
	u|_{t=0}&=u_0,\quad x\in M,
		\end{cases}
\end{equation}
where $\Delta $ is the Laplace-Beltrami operator on $M$ and  $A$ is the second fundamental form of the embedding $N\subset \mathbb R^\ell.$
In 1964,  Al'ber \cite{Alber64,Alber67} introduced the flow,  Eells and Sampson \cite{ES64}  proved the following famous theorem.
\begin{theorem}
	[Eells-Sampson \cite{ES64}]\label{thm-ES64}
Let $M, N$ be two compact Riemannian manifolds without boundaries. Suppose the sectional curvature of $N$ is non-positive. Then for any $u_0\in C^\infty(M,N)$, the flow (\ref{equ-1.1}) admits a unique, smooth solution $u\in C^\infty(M\times[0,+\infty),N)$. Moreover, there exists a subsequence of $\{u(\cdot,t_j)\}$ which converges to a harmonic map, as $t_j\to+\infty$.
\end{theorem}
This smoothness of $u(x,t)$ in Theorem \ref{thm-ES64} is important, because it makes $t \mapsto u(\cdot, t)$  a continuous deformation that preserves the homotopy class of the initial value $u_0$.

When $M$ has non-empty smooth boundary $\partial M$, Hamilton \cite{Ham75} studied the  initial-boudnary problem
\begin{equation}\label{equ-1.2}
\begin{cases}
	\partial_t u&=\Delta u+A(u)(\nabla u,\nabla u),\quad (x,t)\in M\times(0,+\infty),\\
	u|_{t=0}&=u_0, \quad x\in M,\\
	u(x,t)&=\psi(x),\quad x\in \partial M,\ \  t>0.
	\end{cases}
\end{equation}
Hamilton proved that if $N$ has non-positive sectional curvature, then for  given $u_0,\psi\in C^\infty(\overline M,N)$ with $u_0|_{\partial M}=\psi$, there exists a unique  smooth solution in $C^\infty(\overline M\times[0,+\infty),N).$

In 1992, Gromov and Schoen \cite{GS92} developed a theory of harmonic maps into $CAT(0)$ complexes, by using the calculus of variations,  to establish the $p$-adic superrigidity of lattices in groups of rank one. After this, many interesting results for harmonic maps into or between singular metric spaces have been obtained (see, e.g., \cite{KS93,Che95,KS97,Jost94,Jost95,Jost96,Jost97,Lin97,EF01,Wang01,Stu01,Stu05, Mese02, DM08,DM10,HZ17,ZZ18,ZZZ19, Fre19, BFHMSZ20,FZ20,Guo21,MS22+,Gig23,ZZ24} and references therein).

Remark that variational methods have their limitations; for instance, the convergence in $W^{1,2}(M, N)$   does not preserve some topological properties (examples can be found in \cite[page 110]{LinW08}).
 This naturally motivates the study of the harmonic map heat flow into singular metric spaces.

\subsection{Harmonic map heat flows into $CAT(0)$ spaces} Since a  general $CAT(0)$ space $Y$    (i.e., a globally non-positively curved metric space in the sense of Alexandrov) might not be locally compact and can not be embedded into an Euclidean space, one has to use an intrinsic approach to study the harmonic maps and their heat flows into $CAT(0)$ spaces.

The theory of Sobolev spaces for maps into metric spaces has been well-developed \cite{KS93,Jost94,Jost98,KS03,HKST01,Ohta04,GT21b}.
Let  $u$ be a map from a bounded open domain $\Omega\subset M$ to an arbitrary metric space  $(Y,d_Y)$. It is called a $L^2(\Omega,Y)$ map if its range is  separable and if for some $P\in Y$, the function $d_Y(P,u(\cdot))$ is in $L^2(\Omega)$.
Given   a map $u\in L^2(\Omega,Y),$
 for each $\epsilon>0$, the \emph{approximating energy} $E^u_{\epsilon}$ is defined as a functional on $C_c(\Omega)$, the space of continuous functions compactly supported on $\Omega$:
$$E^u_{\epsilon}(\phi):=\int_\Omega\phi(x) e^u_{\epsilon}(x)\dm(x),\quad e^u_{\epsilon}(x):=\frac{n(n+2)}{\omega_{n-1}}\int_{B_\epsilon(x)\cap\Omega}\!\!\frac{d^2_X\big(u(x),u(y)\big)}{\epsilon^{n+2}}\dm(y)$$
for all $\phi\in C_c(\Omega)$, where $\omega_{n-1}$ is the volume of $(n-1)$-dimensional sphere $\mathbb S^{n-1}$ with the standard metric. In \cite{KS93},  Korevaar-Schoen proved that
\begin{equation}\label{equ-1.3}
	\lim_{\epsilon\to0^+} E^u_{\epsilon}(\phi)=E[u](\phi)
\end{equation}
for some positive functional $E[u](\phi)$ on $C_c(\Omega)$.   An $L^2(\Omega,Y)$-map $u$ is called in $W^{1,2}(\Omega,Y)$ if $$E[u]:=\sup_{\phi\in C_c(\Omega), \ 0\ls \phi\ls 1}E[u](\phi)<+\infty.$$

 The first theory of harmonic map heat flow into $CAT(0)$ space was established by Mayer \cite{May98} and Jost \cite{Jost98} via the theory of gradient flow in $CAT(0)$ spaces. The theory of gradient flows in metric spaces was initiated by De Giorgi, A. Marino, and M. Tosques \cite{DegMT80}, and by De Giorgi \cite{DeG93}. In the setting of $CAT(0)$ spaces, Mayer \cite{May98} and Jost \cite{Jost98}, independently, extended Crandall-Liggett's theory of gradient flows from Banach spaces to $CAT(0)$ spaces, for semi-convex and lower semi-continuous functionals. For this paper, we consider only a convex and lower semi-continuous functional $G:L\to [0,+\infty]$ on a $CAT(0)$ space $(L,D)$. Denote by $D(G):=\{u\in L: G(u)<+\infty\}$. 
 Jost \cite{Jost95} and Mayer \cite{May98} studied the resolvent  $J_h$ of $G$, that is,  for any given $h>0$ and $u_0\in \overline{D(G)}$, $J_h(u_0)$ is the (unique) minimizer of
\begin{equation*}
	u\mapsto G(u)+\frac{1}{2h}D^2(u,u_0),  \quad {\rm for}\ u\in L.
\end{equation*}
It was proved \cite{May98,Jost98} that  the limit
\begin{equation*}
  u(t)=F_t(u_0):=\lim_{m\to+\infty}J^m_{t/m}(u_0)
\end{equation*}
exists for every $t\in (0,+\infty)$, and satisfies the following properties:
\begin{itemize}
\item $\lim_{t\to0}u(t)=u_0$ and the semi-group property $F_{t+s}(u_0)=F_t\circ F_s(u_0)$ for all $t, s>0$;
\item  $u(t)\in Lip_{\rm loc}((0,+\infty),L)\cap C^{1/2}([0,+\infty),L)$;
\item the Evolution Variational Inequality (EVI) (see \cite[Lemma 2.37]{May98}, or Lemma \ref{lem-2.4}(vi)),
\begin{equation}
	\label{equ-1.4}
\frac{1}{2}\frac{\rm d}{\rm dt}D^2(u(t),v)+G(u(t))\ls G(v)\quad {\rm in}\ \ \mathscr D'(0,+\infty),\quad \forall \ v\in L.
\end{equation}
\end{itemize}
The curve $u(t)$ is called the gradient flow of $G$ starting at $u_0$.

Assume that $(Y,d_Y)$ is a $CAT(0)$ space.  It has been showed \cite{KS93} that $(L^2(\Omega,Y),D)$ is also a $CAT(0)$ space, where the metric $D$ is defined by 
$$D(u,v):=\left(\int_\Omega d^2_Y(u(x),v(x))\dm(x)\right)^{1/2},\quad \forall u,v\in L^2(\Omega,Y).$$ 
Given any $\psi\in W^{1,2}(\Omega,Y)$,  
$$W^{1,2}_\psi(\Omega,Y):=\Big\{u\in W^{1,2}(\Omega,Y)|\ d_Y\big(u(x),\psi(x)\big)\in W^{1,2}_0(\Omega),\ \ E[u]\ls E[\psi]\Big\}$$ 
is a subset of the space of $W^{1,2}(\Omega,Y)$-maps with the same boundary values as $\psi.$  It is clear that $W^{1,2}_\psi(\Omega,Y)$ is a convex subset of $L^2(\Omega,Y)$. The lower semi-continuous of the energy $E[u]$ in distance $D$ implies that $W^{1,2}_\psi(\Omega,Y)$ is closed in $L^2(\Omega,Y)$. Hence,   the space $(W^{1,2}_\psi(\Omega,Y),D)$ is a $CAT(0)$ space too (see also   \cite[Lemma 3.3]{May98}). Remark that in the case when $\Omega=M$ is a compact Riemannian manifold without boundary, we take $W^{1,2}_\psi(\Omega,Y)=W^{1,2}(\Omega,Y)$.
Let $u_0\in W^{1,2}_\psi(\Omega,Y)$ be arbitrarily given. Because the functional $G(u):=\frac{E[u]}{2}$ is convex and semi-continuous on the metric space $(L^2(\Omega,Y),D)$,
by applying the above theory of gradient flows to the functional $E[u]/2$ on $(W^{1,2}_\psi(\Omega,Y),D)$, Mayer \cite[Theorem 3.4]{May98} obtained a gradient flow $u(x,t): \Omega\times (0,+\infty)\to Y$ started at $u_0$, called a \emph{semi-group weak solution} to the \emph{harmonic map heat flow} on $\Omega$ with initial value $u_0$ and boundary data $\psi$.  About the same time, Jost \cite{Jost98} also obtained the long-time existence of semi-group weak solutions.
Furthermore, Mayer \cite{May98}  established the uniqueness of $u(x,t)$ in $W^{1,2}(\Omega,Y)$, and the regularity
\begin{equation}
\label{equ-1.5}
t\mapsto u(x,t) \in  Lip_{\rm loc}\big((0,+\infty),L^{2}(\Omega,X)\big)\cap C^{1/2}\big([0,+\infty),L^{2}(\Omega,X)\big).
\end{equation}
Moreover, the long-time asymptotic limit
$$u_\infty:=\lim_{t\to+\infty}F_t(u_0)$$
exists, and is the unique harmonic map with boundary data $\psi$ (i.e., the minimizer of $E[u]$ in $W^{1,2}_\psi(\Omega,Y)$).

Motivated by studying the gradient flows for functionals on the $L^2$-Wasserstein space, Ambrosio, Gigli and Sarav\'e \cite{AGS08} introduced a very general theory of gradient flows via the EVI condition. More recently, Ohta and P\'alfia \cite{OP17} and Gigli and Nobili \cite{GN21} adopted the EVI condition to research the gradient flows for semi-convex functionals on $CAT(\kappa)$ space for some $\kappa\gs0$. For our purpose, the general uniqueness result for gradient flows satisfying the EVI condition (see \cite[Theorem 4.0.4]{AGS08} or \cite[Theorem 3.3]{GN21}) implies that, \emph{for any convex and lower semi-continuous functional $G:[0,+\infty]\to L$ on a $CAT(0)$ space $(L,D)$ and any $u_0\in \overline{D(G)}$, there exists a unique locally absolutely continuous curve $ (0,+\infty)\ni t\mapsto u(t)\to L$, called an EVI-gradient flow, such that (i) $\lim_{t\to0}u(t)=u_0$, (ii) it satisfies the EVI \eqref{equ-1.4}.}  According to this uniqueness, we know that the gradient flows given in \cite{May98,Jost98} are also the EVI-gradient flows in \cite{AGS08,GN21}. 
Recall that a curve $u(t):(0,+\infty)\to L$ is called  locally absolutly continuous  if for each $a\in (0,+\infty)$ there exist  a neighborhood $(a-\delta,a+\delta)$ and a function $f\in L^1(a-\delta,a+\delta)$ such that $D(u(s),u(t))\ls\int_s^tf(r){\rm d}r$ for all $s,t\in (a-\delta,a+\delta)$.

Very recently, Lin, Segatti, Sire, and Wang \cite{LSSW25+} provided an elliptic approximation method to construct weak solutions to the harmonic map heat flow into $CAT(0)$ spaces. Let $(M,g)$ be an $n$-dimensional complete Riemannian manifold without boundary. They considered the minimizers $u_\epsilon(x,t):M\times\mathbb R_+\to Y$  of
$$E_\epsilon(v):=\frac 1 2\int_0^\infty\frac{e^{-t/\epsilon}}{\epsilon}\int_M\left(\epsilon|\partial_tv|^2+|\nabla v|^2\right){\rm d}t\dm(x).$$
Given $u_0\in L^2_{\rm loc}(M,Y)$ with  $E(u_0)<+\infty$, they \cite{LSSW25+} proved the following results:
\begin{itemize}
\item[(a)]For any $\epsilon>0$, 
there exists a unique minimizer $u_\epsilon\in W^{1,2}_{\rm loc}(M\times \mathbb R_+;Y)$ with $u_\epsilon(0)=u_0$ and $\int_{M\times\mathbb R_+}(|\partial_tu_\epsilon|^2+|\nabla u_\epsilon|^2)e^{-t/\epsilon}\dm{\rm d}t<+\infty,$ which satisfies the EVI.
\item[(b)] There exists a unique suitable weak solution of the heat flow of harmonic maps $u:M\times\mathbb R_+\to Y$ with initial value $u_0$, such that    
$u_\epsilon\overset{\epsilon\to0}{\longrightarrow}u$ in $L^2_{\rm loc}(M\times\mathbb R_+).$ (Recall  \cite[Definition 1.2]{LSSW25+} that a curve $u:[0,+\infty)\to L^2(M,Y)$ with  $u\in AC^2([0,+\infty);L^2(M,Y))\cap L^\infty(0,+\infty;W^{1,2}(M,Y))$ is called a suitable weak solution of the heat flow of harmonic map into $(Y,d_Y)$ with initial value $u_0$ if $u|_{t=0}=u_0$ and if $u$ is a solution of the EVI \eqref{equ-1.4} for the functional $G=E/2$. A curve $u(t):[0,T]\to L$ on a metric space $(L,D)$ is in $AC^2(0,T;L)$ if there exists a function $f\in L^2(0,T)$ such that $D(u(s),u(t))\ls \int_s^tf(r){\rm d}r$ for any $0<s<t<T$.)  
\item[(c)]Assume additionally that the space $Y$ is locally compact. Then $u\in C^\alpha(M\times(0,+\infty),Y)$ for some $\alpha\in (0,1)$.
\end{itemize}

Notice that the suitable weak solution $u(x,t)$ in the above (b) satisfies the EVI \eqref{equ-1.4}, and that the fact $u\in AC^2([0,+\infty);L^2(M,Y))$ implies $\lim_{t\to0}u(t)=u(0)$ in $L^2(M,Y)$. Therefore, the suitable weak solution $u(x,t)$ of the heat flow of harmonic maps is also an $EVI$-gradient flow of $E/2$. By using the above uniqueness of $EVI$-gradient flows  \cite{AGS08, GN21} again, we know that the notion of the suitable weak solution $u(x,t)$ of the heat flow of harmonic maps coincides with the notion of the semi-group weak solutions of the harmonic map heat flow in \cite{May98}, at least when $M=\Omega$ is a compact Riemannian manifold. Thus, in the following, we shall call this {\emph{a weak solution of harmonic map heat flow}}.

\subsection{Main result on the Lipschitz regularity}

Lipschitz continuity of harmonic maps plays a central role in establishing rigidity theorems of geometric group theory (see, for example, \cite{GS92} for the rigidity of lattices in groups of rank one, and \cite{DM08,DM10,DM12,DM21} for the rigidity of Teichem\"uller spaces, and the references therein for this topic). The first Lipschitz regularity for harmonic maps into singular space was established by Gromov and Schoen in their seminal work \cite{GS92}. Caffarelli and Lin \cite{CL08} proved the Lipschitz regularity of harmonic maps into an Euclidean splitting $\mathbb R^L=\oplus_k\mathbb R^{L_k}$, where the subspaces $\mathbb R^{L_k}$ are orthonormal subspaces.
In 1997, Jost \cite{Jost97} and Lin \cite{Lin97} independently proved that every harmonic map from a finite-dimensional Alexandrov space with curvature bounded from below to a $CAT(0)$ space is locally H\"older continuous. They further conjectured that this H\"older continuity could be improved to Lipschitz continuity (see \cite[page 119]{Lin97} and \cite[page 38]{Jost98}). This conjecture was resolved by the authors in \cite{ZZ18}.  Recently, Mondino-Semola \cite{MS22+} and Gigli \cite{Gig23} have independently proved that the Lipschitz regularity for a harmonic map from an $RCD$ space into a $CAT(0)$ space. Additionally, Assimos-Gui-Jost \cite{Ass-GJ24} proved the local Lipschitz continuity for sub-elliptic harmonic maps from $n$-dimensional Heisenberg groups into a $CAT(0)$ space.

It is a long-standing open question to ask if the weak solutions of harmonic map heat flow into a $CAT(0)$ space possess the Lipschitz regularity. Very recently, Lin, Segatti, Sire, and Wang \cite{LSSW25+} provided a partial answer to this question by establishing the Lipschitz continuity in $x$-variable for a wide class of $CAT(0)$ spaces that is described in the following theorem.
\begin{theorem}[Lin-Segatti-Sire-Wang \cite{LSSW25+}]\label{thm-LSSW} Let $(Y,d_Y)$ be a $CAT(0)$ space satisfying the following properties:
\begin{itemize}
	\item [(a)] $(Y,d)\hookrightarrow \mathbb R^L$ can be realized as a subset of an Euchlidean space $\mathbb R^L$ for some large $L$ and satisfies
	\item [(b)] for any $\delta>0$ and $P\in Y$, there exists a $r_{P,\delta}>0$ such that 
	$$\left|\frac{d_Y(P,Q)}{|P-Q|}-1\right|\ls \delta,\quad\ \forall Q\in Y \ \ {\rm with}\ \ d_Y(P,Q)<r_{P,\delta}.$$
\end{itemize}
If $u_0:\mathbb R^n\to Y$ satisfies $E(u_0)<\infty$ and $d(u_0,Q)\in L^\infty(\mathbb R^n)$ for some $Q\in Y$, then every weak solution $u$ of the harmonic map heat flow is Lipschitz continuous in $x$ and H\"older continuous with exponent $1\over 2$ in $t$ on $\mathbb R^n\times(0,+\infty)$.
\end{theorem}
They conjectured that their result in Theorem \ref{thm-LSSW} holds for \emph{any} $CAT(0)$ spaces (see \cite[Page 7]{LSSW25+}). Meanwhile, 
they also pointed out  \cite[Page 8]{LSSW25+} that: ``The Crandall-Liggett scheme only produces weak solutions and it is difficult to infer from the argument that the objects enjoy higher regularity. "

The main result of this paper is a complete answer to this problem. More precisely, we establish the Lipschitz regularity (in both space and time) of the weak solutions of harmonic map heat flow, along with an Eells-Sampson-type Bochner inequality.
\begin{theorem}\label{main-thm}
	Let $M$ be a complete Riemannian manifold with $Ric_M\gs K$, and let  $(Y,d_Y)$ be an arbitrary $CAT(0)$ space (not necessarily locally compact).  Let $\Omega\subset M$ be a bounded open domain,  and let $u(x,t)$ be a weak solution of harmonic map heat flow from $\Omega$ to $Y$ with initial data $u_0\in W^{1,2}(\Omega,Y)$ and boundary data $\psi\in W^{1,2}(\Omega,Y)$.  Suppose that the image of $u_0$ is bounded in $Y$.   Then
	\begin{itemize}
	\item[(i)]
	 $u(x,t)$ is locally Lipschitz continuous on $\Omega\times(0,+\infty);$ \\[-5pt]

	\item[(ii)] the  pointwise spatial Lipschitz constant
	 \begin{equation}\label{equ-1.6}
	 {\rm lip}_xu(x,t):=\limsup_{y\to x,\ y\not=x}\frac{d_Y\big(u(x,t),u(y,t)\big)}{d(y,x)},\quad \forall t\in(0,+\infty),
	 \end{equation}
 is in $V_{2,\rm loc}(\Omega\times(0,+\infty))\cap L^\infty_{\rm loc}(\Omega\times(0,+\infty))$ and satisfies the following Eells-Sampson-type Bochner inquality:
	\begin{equation}\label{equ-1.7}
		(\Delta-\partial_t)({\rm lip}_xu)^2\gs 2|\nabla {\rm lip}_xu|^2+2K({\rm lip}_xu)^2
	\end{equation}
	on $\Omega\times(0,+\infty)$ in the sense of distributions.
\end{itemize}	\end{theorem}
Recall that a function $f(x,t)\in V_2(\Omega\times (0,T))$ means
\begin{equation}\label{equ-1.8}
{\rm esssup}_{0<t<T}\|f(\cdot, t)\|_{L^2(\Omega)}+\left(\int_0^T\int_\Omega|\nabla f(\cdot, t)|^2\dm(x){\rm d}t\right)^{1/2}<+\infty;
\end{equation}
and a function $f(x,t)\in V_{2,{\rm loc}}(\Omega\times (0,T))$ means $f(x,t)\in V_{2}(\Omega'\times (a,b))$ for any $\Omega'\Subset \Omega$ and any $(a,b)$ with $0<a<b<T.$

We remark that our Theorem \ref{main-thm} allows the domain $\Omega = M$ when the manifold $M$ is compact without boundary.

\subsection{Outline of the proof and the organization of this paper}

In Sect. 2, we provide some necessary notations and preliminary results concerning Sobolev spaces of maps and geometry of $CAT(0)$ spaces. We also collect some properties of the weak solutions of the harmonic map heat flow obtained by Mayer in  \cite{May98}. In particular, we observe a simple but important fact   $$u(x,t)\in W^{1,2}_{\rm loc}(\Omega\times(0,+\infty),Y),$$ which can be derived from Mayer's regularity result (\ref{equ-1.5}) (see Proposition \ref{prop-2.5} for the details).

 The first step to the proof of the main result, Theorem \ref{main-thm}, is to establish the locally H\"older continuous of $u(x,t)$. This will be given in Sect. 4 (see Theorem \ref{thm-4.1}).  Recall that the classical De Giorgi-Nash-Moser theory implies the local H\"older regularity for weak solutions of elliptic and parabolic equations with divergence form.  But one can not apply the theory to the harmonic maps or their heat flows into $CAT(0)$ spaces directly, because the second fundamental forms on the right-hand side of the equations (\ref{equ-1.1})  are unbounded in general, even if the target spaces $Y$ are embedded in Euclidean spaces. Let $P$ be any fixed point in a $CAT(0)$ space $Y$ and let $u(x)$ be a (weak) harmonic map into $Y$. In  \cite{Jost97} and \cite{Lin97}, Jost and Lin independently observed  the following partial differential inequality:
$$
\Delta d^2_Y\big(P, u(x)\big) \gs 2|\nabla u|^2(x).
$$
 and developed a modified version of the De Giorgi-Nash-Moser  theory to get  H\"older continuity for any (weak) harmonic map $u: \Omega \to Y$.  The core of the De Giorgi-Nash-Moser theory consists of several iterations based on the (local) Poincar\'e inequality. The existence of the positive term   $2|\nabla u|^2(x)$ on the right-hand side of the above partial differential inequality is the key ingredient to perform the iteration arguments.

 In the present paper, we consider the weak solution of the harmonic map heat flow $u(x,t)$ into $CAT(0)$ spaces and try to adapt Jost and Lin's method.   We first prove
$$
(\Delta -\partial_t)d^2_Y\big(P, u(x,t)\big) \gs 2|\nabla u|^2(x,t), \quad \forall P \in Y.
$$
Noticing that the right-hand side of the above inequality involves only the spatial derivatives, when adapting  Jost and Lin's arguments from the elliptic case to the parabolic setting, one requires  the following form of (local) Poincar\'e inequality,
\begin{equation}
	\label{equ-1.9}
	\int_{t-r^2}^{t}\int_{B_r(x)}|f-f_{x,r}|^2dxdt\ls Cr^2\int_{t-2r^2}^{t}\int_{B_{2r}(x)}|\nabla f|^2dxdt,
\end{equation}
where $f_{x,r}:=\frac{1}{\mu(B_r(x))}\int_{B_r(x)}fdx.$ The term $|\nabla f|$ on the righ-hand side of (\ref{equ-1.9}) involves only the derivative of space variable $x$. Due to the loss of the time derivative on the right-hand side, it is clear that  (\ref{equ-1.9}) does not hold for a general function $f(x,t)\in W^{1,2}(\Omega\times(0,T))$. This is an essential difference from the elliptic case. Nevertheless, when $u(x,t)$ is a weak solution of harmonic map heat flow, based on Mayer's regularity result (\ref{equ-1.5}), we will prove that a modified version of such (local) Poincar\'e inequality is still expected (see Lemma \ref{lem-4.4} for a precise statement).

The second step to the proof of Theorem \ref{main-thm} is to show the Lipschitz continuity in the time variable $t$.  Estimating the derivatives in the time variable is absent in the elliptic case of harmonic maps. Our key observation is to show a subsolution property for $d^2_Y(u(x,t),v(x,t))$ of any two weak solutions of the harmonic map heat flow $u(x,t)$ and $v(x,t)$. Indeed, we will prove that the function
$$d^2_Y\big(u(x,t),v(x,t)  \big)\in W^{1,2}\cap L^\infty,$$ and satisfies
 \begin{equation}\label{equ-1.10}
 \big({\Delta}-\partial_t\big) d^2_Y(u(x,t),v(x,t))\gs 2R_{u^t,v^t}(x)  	
  \end{equation}
 	 in the sense of distributions, for almost all $t$, where $R_{u^t,v^t}(x)$ is defined as
  $$ R_{u^t,v^t}(x):= \liminf_{\epsilon\to 0^+} \frac{c_{n,2}}{\epsilon^n}\int_{B_\epsilon(x)\cap \Omega}\left(\frac{d_Y\big(u(x,t),u(y,t)\big)-d_Y\big(v(x,t),v(y,t)\big) }{\epsilon}\right)^2\dm(y).$$
This subsolution property will be established in Sect. 3 (Theorem \ref{thm-3.4}). In Sect. 5, we will apply the formula  (\ref{equ-1.10}) to any weak solution of harmonic map heat flow $u(x,t)$ and then choosing $v(x,t)= u(x, t + s)$  to get the desired the Lipschitz continuity of $u(x,t)$ in the time variable $t$.

Another immediate application is the $L^\infty$-stability of the harmonic map heat flows. We expect that the formula (\ref{equ-1.10}) should have additional applications in establishing rigidity theorems in geometric group theory.

 The third step to the proof of Theorem \ref{main-thm} is to show the Lipschitz continuity in the spatial variable. This is carried out in Sect. 7 (with some necessary tools given in Sect. 6). We will extend our method in \cite{ZZ18} to construct a new nonlinear ``Hamilton-Jacobi" flow to the harmonic map heat flow. The crucial point is to show the nonlinear ``Hamilton-Jacobi" flow is a supersolution of the heat equation in the sense of viscosity.  We now explain the idea as follows.

 To simplify the exposition, we assume here that the domain manifold $M$ has nonnegative Ricci curvature.   Recalling in our previous work \cite{ZZ18}, we developed a nonlinear ``Hamilton-Jacobi" flow method to prove the Lipschitz regularity of any (weak) harmonic map $u$ from $\Omega$ to $Y$. The nonlinear ``Hamilton-Jacobi" flow in \cite{ZZ18} was defined as follows: for any  $\varepsilon>0$ and any $\Omega'\Subset \Omega$,
 \begin{equation}\label{equ-1.11-add1}
 f_\varepsilon(x):=\inf_{y\in \Omega}\left\{\frac{ d^2(x,y)}{2\varepsilon}-d_Y\big(u(x),u(y)\big)\right\},\quad \forall x\in\Omega',
  \end{equation}
 It is a regularization of the function of $d_Y\big(u(x),u(y)\big)$ of two-group variables $x, y$.  We then showed that the nonlinear ``Hamilton-Jacobi" flow $f_\varepsilon(x)$ is superharmonic. The deep reason for the superharmonicity is that the function $d_Y\big(u(x),u(y)\big)$ satisfies some elliptic-type partial differential equation of variables $x, y$.

  For the parabolic case we are studying here, the most natural approach is to construct a nonlinear ``Hamilton-Jacobi" flow for the harmonic map heat flow as:
 \begin{equation}\label{equ-1.11-add2}
f_\varepsilon(x,t):=\inf_{(y,s)\in \Omega\times(0,T)}\left\{\frac{d^2(x,y)+|t-s|^2}{2\varepsilon}-d_Y\bigl(u(x,t),u(y,s)\bigr)\right\},\quad \forall (x,t)\in\Omega'\times(0,T),
\end{equation}
and then extend the elliptic method to the parabolic setting to establish the space-time local Lipschitz continuity of $u$.
But we encounter a new difficulty: because the function $d_Y(u(x,t),u(y,s))$ has two time variables $t, s$, we have not yet found an appropriate parabolic-type partial differential equation for the variables $t, s, x, y$.
Fortunately, as we have obtained the Lipschitz regularity in time in Sect. 5, it allows us to address the spatial Lipschitz regularity only. This leads us to consider a ``revised" nonlinear ``Hamilton-Jacobi" flow in the following form:
  \begin{equation} \label{equ-1.11-add3}
 	f_\varepsilon(x,t):=\inf_{y\in \Omega}\left\{\frac{ d^2(x,y)}{2\varepsilon}-d_Y\big(u(x,t),u(y,t)\big)\right\},\quad \forall (x,t)\in \Omega'\times(0,T).
 \end{equation}
The benefit of (\ref{equ-1.11-add3}) is that the function $d_Y(u(x,t),u(y,t))$  contains just one time variable, and satisfies a parabolic-type partial differential equation (see Lemma \ref{lem-7.3} for the details).

We will show that $f_\varepsilon(x,t)$ is a supersolution of the heat equation in the sense of viscosity (see Lemma \ref{lem-7.6} for a precise statement). We will also establish a connection between the nonlinear ``Hamilton-Jacobi" flow and the square norm of the spatial gradient (see Lemma \ref{lem-7.7} for the details). It is analogous to the classical Hamilton-Jacobi equation. This indicates that the mean value inequality for subsolutions of the heat equation will give an upper bound on the norm of the spatial gradient. So the local  Lipschitz regularity of $u(x,t)$ will be established.

 In the final section, we will prove the Eells-Sampson-type Bochner inequality (\ref{equ-1.7}), which is a parabolic version of our previous joint work with Xiao Zhong \cite[Theorem 1.9]{ZZZ19} for harmonic maps into $CAT(0)$ spaces. Our idea is to refine the arguments in Section 7 by lifting the power of the distance function in the definition of the nonlinear ``Hamilton-Jacobi" flow from the usual index 2 to a higher order index $p$ to extract quantitative information.

\begin{remark}
   Mayer's existence result for the weak solution of harmonic map heat flow has been extended to the setting of maps from $RCD$ spaces into $CAT(0)$ spaces (see \cite{Guo21}). In fact, the arguments presented in Sect. 3, 4, and  5 of this paper can be adapted to the more general setting of maps from $RCD$ spaces into $CAT(0)$ spaces. The only difficulty in the extension is the arguments in Sect. 7, where we need a theory of viscosity solutions for parabolic equations on $RCD$ spaces. We plan to address this problem in the future.
	\end{remark}

{\bf Acknowledgements.} We learned the work of Lin, Segatti, Sire and Wang through a lecture given by Professor Fang-Hua Lin at a conference at Jiangsu University in Zhenjiang in  April 2025.  Lin's lecture inspired us to consider the Lipschitz regularity of weak solutions of harmonic map heat flow constructed in \cite{May98}.  We also want to thank Professor Juergen Jost for the enlightening comments. 

The second author was partially supported by the National
Key R\&D Program of China (No. 2022YFA1005400), and the first author was partially supported by NSFC 12426202.

\section{Preliminaries and notations}

Let $(M,g)$ be an $n$-dimensional complete Riemannian manifold with metric $g$ and $Ric$ be
the Ricci curvature, $n\gs2$.  Let $B_r(x)$ be the geodesic ball with radius $r$ and center $x$.
We also denote $ \mu$  as the canonical measure on $M$. To simplify the notation,
the volume is denoted by $|E|:=\mu(E)$ for any measurable set $E\subset M$. For any $f\in L^1(E)$, we denote by $\fint_Ef{\rm d}\mu:=\frac{1}{|E|}\int_Ef{\rm d}\mu.$

\subsection{Sobolev spaces of maps}
Let $\Omega\subset M$ be a bounded open domain, and let $(Y,d_Y)$ be a complete metric space.
 A Borel measurable map $u:\Omega\to Y$ is said to be in the space $L^p(\Omega,Y)$, $1\ls p<+\infty$, if it has a separable range and, for some (hence, for all) $P\in Y$,
$$\int_\Omega d^p_Y\big(u(x),P\big){\rm d}\mu(x)<+\infty.$$
The space $L^p(\Omega,Y)$ is a complete metric space,  equipped with distance $D_{L^p}$ as
\begin{equation}\label{equ-2.1}
	 D^p_{L^p}(u,v):= \int_\Omega d_Y^p\big(u(x),v(x)\big){\rm d}\mu(x),\quad \forall \ u,v\in L^p(\Omega,Y).
\end{equation}

Given a map $u\in L^p(\Omega,Y),$ for each $\epsilon>0$, the \emph{approximating energy} $E^{u}_{p, \epsilon}$ is defined in \cite{KS93} as a functional on $C_c(\Omega)$, the space of continuous functions with compact supports in $\Omega$, as follows:
\begin{equation*}
	 E^{u}_{p,\epsilon}(\phi):=\int_\Omega\phi(x)\cdot e^u_{p,\epsilon}(x){\rm d}\mu(x),\quad \forall \phi\in C_c(\Omega),
\end{equation*}
 where
\begin{equation}\label{equ-2.2}
	e^u_{p,\epsilon}(x):=
c_{n,p}\int_{B_\epsilon(x)\cap \Omega} \frac{d_Y^p(u(x), u(y))}{\epsilon^{n+p}}{\rm d}\mu(y),
\end{equation}
where the constant $c_{n,p}=(n+p)(\int_{\mathbb S^{n-1}}|x^1|^pd\sigma(x))^{-1}$, and $\sigma$ is the cannonical Riemannian volume on $\mathbb S^{n-1}$. In particular, $c_{n,2}=\frac{n(n+2)}{\omega_{n-1}}$, where $\omega_{n-1}$ is the volume of $(n-1)$-dimensional sphere $\mathbb S^{n-1}$ with standard metric. We say that $u$ has finite $p$-energy (and write $u\in W^{1,p}(\Omega,Y)$ for $p>1$ and $u\in BV(\Omega,Y)$ for $p=1$), if
\begin{equation}\label{equ-2.3}
E_p[u]:=\sup_{0\ls \phi\ls 1,\ \phi\in C_c (\Omega)}E_p[u](\phi)<+\infty,	
\end{equation}
where $E_p[u](\phi):= \limsup_{\epsilon\to0^+}E^u_{p,\epsilon}(\phi)$. For $\phi\in C_c(\Omega)$ and $C>0$, define
$$\phi_\epsilon^C(x):=(1+C\epsilon)\big(\phi(x)+\max_{d(y,x)\ls 2\epsilon}|\phi(y)-\phi(x)|\big).$$
 It was proved \cite[Lemma 1.4.2]{KS93} that there exists a contstant $C_g>0$ (depending only on Ricci curvature control of the metric $g$) such that
 \begin{equation}\label{equ-2.4}
  	E^{u}_{p,\epsilon}(\phi)\ls E_p[u](\phi_\epsilon^{C_g}),\qquad\forall \phi\in C_c(\Omega)
 \end{equation}
for any sufficiently small $\epsilon>0$ (indeed, by the proof of \cite[Lemma 1.3.1, Page 576]{KS93}, $\epsilon$ can be chosen so that for any $x\in {\rm supp}(\phi)$ and any $y\in B_\epsilon(x)$, there exists a unique geodesic from $x$ to $y$; for example, $\epsilon <\min\{{\rm Inj}(\Omega)/2, d({\rm supp}(\phi),\partial\Omega)/4\}$ is enough).

For the case $p=2$, to simplify the notations, we denote by
$$D(u,v):=D_{L^2}(u,v),\quad \forall u,v\in L^2(\Omega,Y),$$
  and also $e^u_\epsilon:=e^u_{2,\epsilon}$, $E^u_\epsilon:=E^u_{2,\epsilon}$ and $E[u]:=E_2[u]$  for any $u\in W^{1,2}(\Omega,Y)$.

We collect some properties of $W^{1,2}(\Omega,Y)$, which  can be found in \cite{KS93,GT21b,MS22+}.
\begin{proposition}
\label{prop-2.1}
Let  $\Omega\subset M$ be a bounded open  domain, and let $u\in W^{1,2}(\Omega, Y)$.
\begin{enumerate}
	\item{\rm(Energy density)} There exists a nonnegative function (called energy density) $e_u\in L^1(\Omega)$ such that $E[u](\phi)=\int_\Omega \phi e_u\dm$ for all $\phi\in C_c(\Omega)$, and that
  $$\lim_{\epsilon\to0^+} e_{\epsilon}^u(x)= e_u(x)\quad  \mu{\rm-a.e.}\  x\in\Omega \ {\rm and\ also\ in}\
  L^1_{\rm loc}(\Omega).$$

 \item {\rm(Representation by heat kernel)} For almost all point $x\in \Omega$, one has
 $$\lim_{s\to0^+}\frac{1}{2s}\int_\Omega p_s(y,x)d_Y^2\big(u(x),u(y)\big)\dm(y)=e_u(x),$$
 where $p_s(x,y)$ is the heat kernel on $M$.  (This has been established even in the setting of general $RCD$ metric measure spaces, see \cite[Proposition 3.3]{MS22+} and \cite{GT21b,Gig23}.)
\item {\rm(Equivalence for $Y=\mathbb R$)}\indent If $Y=\mathbb R$,  the above  space $W^{1,2}(\Omega,\mathbb R)$
is equivalent to the usual Sobolev space $W^{1,2}(\Omega)$, and furthermore,
$$e_u(x)=|\nabla u|^2(x),\quad \mu{\rm-a.e.}\ x\in \Omega.$$

\item {\rm(Conctraction for maps)} If $\Phi: Y\to Z$ is a $1-$Lipschitz map, then $\Phi\circ u\in W^{1,2}(\Omega,Z)$ and $e_{\Phi\circ u}\ls e_u$ a.e. in $\Omega$; in particular, for any $P\in Y$ and letting $f_P:=d_Y(P,u(x))$, it holds $f_P\in W^{1,2}(\Omega)$ and $|\nabla f_P|^2(x)\ls e_u(x)$ almost all  $x\in \Omega.$

\item {\rm(Poincar\'e inequality)}\indent  There is a constant $c_\Omega>0$ depending only on $\Omega$, such that for any ball $B_r(x)$ with $B_{2r}(x)\subset \Omega$, and any $u\in W^{1,2}(\Omega,Y)$, it holds
\begin{equation}\label{equ-poincare}
\inf_{q\in Y}\int_{B_r(x)}  d^2_Y\big(q,u(x)\big){\rm d}\mu(x)  \ls c_\Omega\cdot r^2\int_{B_{2r}(x)}e_u(x){\rm d}\mu(x).
\end{equation}
(See also \cite{KST04}.)
\end{enumerate}
\end{proposition}

\subsection{$CAT(0)$ spaces}
Let $(Y,d_Y)$ be a complete metric space. A curve $\gamma:[0,1]\to Y$ is called a \emph{geodesic} if $d_Y(\gamma(t),\gamma(s))=|t-s|d_Y(\gamma(0),\gamma(1))$ for any $t,s\in[0,1].$  A metric space $(Y,d_Y)$ is called a \emph{geodesic space} if, for every pair of points $P_0,P_1\in Y$, there exists a geodesic $\gamma:[0,1]\to Y$ such that $\gamma(0)=P_0$ and $\gamma(1)=P_1$.

\begin{defn}\label{CAT0}
A geodesic space $(Y,d_Y)$ is called a $CAT(0)$ space, if it holds
\begin{equation}\label{equ-cat0}
	d^2_Y(P,\gamma(t))\ls (1-t)d^2_Y(P,\gamma(0))+td^2_Y(P,\gamma(1))-t(1-t)d_Y^2(\gamma(0),\gamma(1))
	\end{equation}
	for any $P\in Y$ and any geodesic $\gamma:[0,1]\to Y$.
 \end{defn}

If $(Y,d_Y)$ is a $CAT(0)$ space, then for any two points $P_0, P_1$, there exists \emph{uniquely} geodesic $\gamma:[0,1]\to Y$ such that $\gamma(0)=P_0$ and $\gamma(1)=P_1.$
We need the following lemma, essentially due to Reshetnyak \cite{Res68}.

\begin{lemma}\label{lem-2.3}
	 Let $(Y,d_Y)$ be a $CAT(0)$ space. Take any ordered sequence $\{P,Q,R,S\}\subset Y$.
	 We denote the distance $d_Y(A,B)$ abbreviatedly by $d_{AB}.$
	\begin{itemize}
		\item[(1)]We have
		$$ d_{PR}^2+d_{QS}^2\ls d_{PQ}^2+d_{QR}^2+d_{RS}^2+d_{SP}^2-\left(d_{RS}-d_{PQ}\right)^2,$$
		\item[(2)] Let point $Q_{m}$ be the mid-point of $Q$ and $R$ (i.e. $d_{QQ_m}=d_{RQ_m}=\frac{1}{2}d_{QR}$), then we have
\begin{equation*}
( d_{PS}-d_{QR})\cdot d_{QR}\gs (d^2_{PQ_{m}}-d^2_{PQ}-d^2_{Q_{m}Q})+(d^2_{SQ_{m}}-d^2_{SR}-d^2_{Q_{m}R}).
 \end{equation*}
 \item[(3)] Let $\lambda\in(0,1)$ and let $P_\lambda$ be the point in geodesic $PS$ such that $d_{PP_\lambda}=\lambda d_{PS}$, then  we have
		$$\lambda\left(d^2_{PS}+d_{PQ}^2-d^2_{SQ}\right)\ls d^2_{PP_\lambda}+d_{PQ}^2-d^2_{P_\lambda Q}.$$
	 \item[(4)] Let $\lambda,\mu\in(0,1)$, let $P_\lambda$ be the point in geodesic $PS$   such that $d_{PP_\lambda}=\lambda d_{PS}$,  and let  $Q_\mu$ be the point in geodesic $QR$  such that  $d_{QQ_\mu}=\mu d_{QR}$. Then  we have
	\begin{equation}\label{equ-pq-1}
		\begin{split}
			d^2_{P_\lambda Q_\mu}\ls &  \mu(1-\lambda)d^2_{PR}+(1-\mu)\lambda d^2_{QS}\\
			&+ \mu \lambda d^2_{SR}+(1-\lambda)(1-\mu) d^2_{PQ}-\lambda(1-\lambda)d^2_{PS}-\mu(1-\mu)d^2_{QR},
		\end{split}
	\end{equation}
	and
	 \begin{equation}\label{equ-pq-2}
			d^2_{P_\lambda Q_\mu}\ls  2(1-\lambda)d^2_{PQ}+2\lambda d^2_{RS}+2|\lambda-\mu|^2d^2_{QR}.\qquad\qquad\qquad\qquad
	\end{equation}	
	\end{itemize}
\end{lemma}

 \begin{proof}
	Both (1) and (2) are special cases of Corollary 2.1.3 in \cite{KS93}. Precisely, (1) is given by (2.1vi) in \cite{KS93} with $\alpha=0$, and (2) is given by (2.1v) in \cite{KS93} with $t = 1/2$ and $\alpha = 1$, see also \cite[Lemma 5.2]{ZZ18}.

For (3), by (\ref{equ-cat0}), we have
$$d^2_{QP_\lambda}\ls \lambda d^2_{QS}+(1-\lambda)d^2_{PQ}-\lambda(1-\lambda)d^2_{PS}.$$
Rewriting this, we get
\begin{equation*}
	\begin{split}
		\lambda\Big( d^2_{PQ}-d^2_{QS}+d^2_{PS}\Big)&\ls d^2_{PQ}-d^2_{QP_\lambda}-\lambda(1-\lambda)d^2_{PS}+ \lambda d^2_{PS}\\
		&= d^2_{PQ}- d^2_{QP_\lambda}  + \lambda^2 d^2_{PS}.
		\end{split}
\end{equation*}
This is the assertion (3), since $\lambda d_{PS}=d_{PP_\lambda}.$

For (4), considering the point $P_\lambda$ and geodesic $QR$, the (\ref{equ-cat0}) implies
$$d^2_{P_\lambda Q_\mu}\ls \mu d^2_{P_\lambda R}+(1-\mu) d^2_{P_\lambda Q}-\mu(1-\mu)d^2_{QR}.$$
By using (\ref{equ-cat0}) again to the geodesic $PS$ and the point $Q$ (or  $R$, resp.), we obtain
$$d^2_{P_\lambda Q}\ls (1-\lambda)d^2_{P  Q}+\lambda d^2_{SQ}-\lambda(1-\lambda)d^2_{PS},$$
$$d^2_{P_\lambda R}\ls (1-\lambda)d^2_{P  R}+\lambda d^2_{SR}-\lambda(1-\lambda)d^2_{PS}.$$
Combining the above three inequalities, we conclude (\ref{equ-pq-1}).

Let $Q_\lambda$ be the point in geodesic $QR$ such that $d_{QQ_\lambda}=\lambda d_{QR}$. By using \cite[(2.1iv) in Corollary 2.13]{KS93}, we have
$$d^2_{P_\lambda Q_\lambda}\ls (1-\lambda)d^2_{PQ}+\lambda d^2_{RS}.$$
Combining the triangle inequality, we get
\begin{equation*}
\begin{split}
	d^2_{P_\lambda Q_\mu}&\ls \Big(d_{P_\lambda Q_\lambda}+d_{Q_\lambda Q_\mu}\Big)^2\ls 2d^2_{P_\lambda Q_\lambda} +2d^2_{Q_\lambda Q_\mu}\\
	&\ls 2(1-\lambda)d^2_{PQ}+2\lambda d^2_{RS}+2|\lambda-\mu|^2d^2_{QR},
	\end{split}	
\end{equation*}
where we have used $d_{Q_\lambda Q_\mu}=|\lambda-\mu|d_{QR}$ in the last inequality. The proof is finished.
\end{proof}

\subsection{Harmonic map heat flows   into $CAT(0)$ spaces}
From now on, the target space $(Y,d_Y)$ is always assumed to be a $CAT(0)$ space.
 Given any $\psi\in W^{1,2}(\Omega,Y)$, we put
\begin{equation*}
	W^{1,2}_\psi(\Omega,Y):=\big\{u\in W^{1,2}(\Omega,Y)\ \big|\ d_Y(u,\psi)\in W^{1,2}_0(\Omega)\ {\rm and }\ E[u]\ls E[\psi]\big\}.
\end{equation*}
Since $(Y,d_Y)$ is of $CAT(0)$, the functional
$$u\mapsto \frac{1}{2}E[u]$$
  is convex on the metric space $L^2(\Omega,Y)$. According to \cite[Theorem 3.4]{May98}, for any $u_0\in W^{1,2}_\psi(\Omega,Y)$ there exists uniquely a gradient flow $t\mapsto F_t(u_0)\in W^{1,2}_\psi(\Omega,Y)$ of $E[u]/2$ with $F_0(u_0)=u_0$, defined on $(0,+\infty)$. We call such map $u(x,t):=F_t(u_0)$ a \emph{weak solution of harmonic map heat flow} with initial-boundary values $u_0, \psi\in W^{1,2}(\Omega, Y)$. We recall the construction of $F_t$ in \cite{May98} as follows. Let $h>0$, $J_h(u_0)$ is the minimizer of $$\frac{1}{2}E[u]+\frac{1}{2h}D^2(u,u_0),$$
which is unique in $W^{1,2}_\psi(\Omega,Y)$. The gradient flow $F_t$ is defined by
$$F_t(u_0):=\lim_{m\to+\infty}J^m_{t/m}(u_0)=\lim_{m\to+\infty}J_{t/m}\circ J_{t/m}\circ\cdots\circ J_{t/m}(u_0).$$
 The existence of this limit has been shown in  \cite[Theorem 1.13]{May98}.

We collect some properties of the weak solutions of harmonic map heat flow obtained in \cite{May98} as follows.
 \begin{lemma} \label{lem-2.4}
  Let $(Y,d_Y)$ be a $CAT(0)$ space and let $u(x,t)$ be a   weak solution of harmonic map heat flow on $\Omega\times(0,+\infty)$. Denote by $u^t(\cdot):=u(\cdot, t)$ for any $t\in (0,+\infty)$. We have the following:
 \begin{itemize}
 \item[(i)] {\rm (Semigroup property)} \ $F_{t+s}=F_t\circ F_s$ for any $t,s>0$.
 \item[(ii)] {\rm (Non-increasing of energy)} \  $E[u^t]\ls E[u^s]$ for any $t> s>0$.
  \item[(iii)]
For any $\bar u\in W^{1,2}_\psi(\Omega,Y)$ and $t>s>0,$ we have
\begin{equation}\label{equ-2.9}
	 D^2(\bar u, u^t)- D^2(\bar u, u^{t-s})\ls  -s\big(E[u^t]-E[\bar u]\big).
\end{equation}
\item[(iv)] {\rm (Regularity of flow)}\  The function  $t\mapsto u^t$ is in $C^{1/2}([0,T], L^2(\Omega,Y))\cap Lip([t_*,T], L^2(\Omega,Y))$ for any $0<t_*<T<+\infty$. Precisely,  given any $t_*, T\in (0,+\infty)$ with $t_*<T$, there exist two constant $L:=L_{t_*,T}>0$ and $B=B_{T}>0$  such that
\begin{equation}\label{equ-2.10}
	D(u^t,u^s)\ls L|t-s|,\quad \forall t,s\in[t_*,T],
\end{equation}
and
\begin{equation*}
	D(u^t,u^s)\ls B|t-s|^{1/2},\quad \forall t,s\in[0,T].
\end{equation*}
\item[(v)] {\rm (Reguarity of energy)}\ The function $t\mapsto E[u^t]$ is continuous on $[0,+\infty)$ and Lipschitz continuous on $[t_*,T]$ for any $T>t_*>0$.
\item[(vi)]  {\rm (Evolution Variational Inequality)}\  The curve $t\mapsto u^t$ satisfies the  Evolution Variational Inequality (EVI),
$$\frac{1}{2}\frac{\rm d}{\rm dt}D^2(u^t, \bar u)+\frac{E[u^t]}{2}\ls \frac{E[\bar u]}{2}\quad {\rm in}\ \ \mathscr D'(0,+\infty),\quad \forall \ \bar u\in W^{1,2}_\psi(\Omega,Y).$$
\end{itemize}
\end{lemma}
 \begin{proof}The (i) is \cite[Theorem 2.5]{May98}. 	The (ii) is \cite[Corollary 2.6]{May98}. The (iii) is  \cite[Lemma 2.37]{May98}. The (iv) comes from the combination of   \cite[Theorem 2.9]{May98} and \cite[Theorem 2.2]{May98}. The (v) is \cite[Corollary 2.10]{May98}.
 
 For the EVI in the assertion (vi), by dividing $s$ in (\ref{equ-2.9}) and using the fact that $t\mapsto D^2(\bar u, u^t)$ is a locally Lipschitz continuous in $(0,+\infty)$ (since (iv)), we have 
 $$\frac{\rm d}{\rm dt}D^2(u^t, \bar u)\ls -\big( E[u^t]-E[\bar u]\big) \quad {\rm in}\ \ \mathscr D'(0,+\infty).$$
 The proof is finished.
 \end{proof}

The above Lemma \ref{lem-2.4} (ii) implies $u(x,t)\in L^\infty\big((0,+\infty), W^{1,2}(\Omega,Y)\big).$
The following proposition shows that $u(x,t)\in W^{1.2}(\Omega\times(t_*,T),Y)$ for any $0<t_*<T<+\infty$.
\begin{proposition}\label{prop-2.5}
   Given any  $t_*,T\in(0,+\infty)$ with $t_*<T$, it holds
\begin{equation}\label{equ-2.11}
	\int_{t_*+\epsilon}^T\int_{\Omega_\epsilon}e_{2,\epsilon}^u(x,t){\rm d}\mu(x){\rm d}t\ls C_1\cdot  (T-t_*)\cdot\Big(L^2+ E[u^{t_*}]\Big)
	\end{equation}
			for any $\epsilon\in (0,r_0)$, where  $L$ is in (\ref{equ-2.10}), $r_0:=\min\{ t_*/2,(T-t_*)/2\}$, and $\Omega_\epsilon:=\{x\in \Omega: d(x,\partial\Omega)>\epsilon\}$. The constant $C_1$  depends only on $n$ and the Riemannian metric $g$ on $\Omega$.
			
			In particular,
$u(x,t)\in W^{1,2}(\Omega\times(t_*,T),Y)$.
\end{proposition}
\begin{proof}
	For any $(x,t)\in \Omega_\epsilon\times(t_*,T)$, the triangle inequality implies
	$$d^2_Y(u(x,t),u(y,s))\ls 2 d^2_Y(u(x,t),u(x,s))+2d^2_Y(u(x,s),u(y,s))$$
	for any $(y,s)\in B_{\epsilon}(x,t)$.
We have
	\begin{equation} \label{equ-2.12}
			\begin{split}
			e^u_{\epsilon}(x,t)&\ls c_{n+1,2}\int_{B_{\epsilon}(x,t)}\frac{d^2_Y(u(x,t),u(x,s))+d^2_Y(u(x,s),u(y,s))}{\epsilon^{n+1+2}}\dm(y){\rm d}s\\
			&= c_{n+1,2}\int_{t-\epsilon}^{t+\epsilon}\int_{B_{\epsilon}(x)}	\frac{d^2_Y(u^t(x),u^s(x))}{\epsilon^{n+3}}\dm(y){\rm d}s\\
			&\qquad\ +c_{n+1,2}\int_{t-\epsilon}^{t+\epsilon}\int_{B_{\epsilon}(x)}\frac{ d^2_Y(u^s(x),u^s(y))}{\epsilon^{n+3}}\dm(y){\rm d}s\\
				&	\ls c_{n+1,2}|B_{\epsilon}(x)|\int_{t-\epsilon}^{t+\epsilon}\frac{d^2_Y(u(x,t),u(x,s))}{\epsilon^{n+3}}{\rm d}s+ \frac{c_{n+1,2}}{c_{n,2}\cdot \epsilon} \cdot \int_{t-\epsilon}^{t+\epsilon}e^{u^s}_\epsilon(x){\rm d}s.
			\end{split}
			\end{equation}
	
By (\ref{equ-2.10}), we have	
	 			\begin{equation}\label{equ-2.13}
				\begin{split}
			\int_{t_*+\epsilon}^{T}\int_{t-\epsilon}^{t+\epsilon}\int_{\Omega_\epsilon} d^2_Y(u^t(x),u^s(x)) \dm(x){\rm d}s{\rm d}t& \ls \int_{t_*+\epsilon}^{T}\int_{t-\epsilon}^{t+\epsilon}  D^2(u^t,u^s) {\rm d}s{\rm d}t\\
			&\ls 2 \epsilon\cdot(L\epsilon)^2 (T-t_*).
					\end{split}
			\end{equation}	
By using (\ref{equ-2.4}) to a function $\phi\in C_c (\Omega)$ with $0\ls \phi\ls 1$, $\phi\equiv1$ on $\Omega_\epsilon$ and $d({\rm supp}(\phi),\partial\Omega)\gs  r_0/2$), we get
$$\int_{\Omega_\epsilon} e_\epsilon^{u^s}(x)\dm(x)\ls    E[u^s](\phi^{C_g}_\epsilon)\ls C_g'\cdot E[u^s]$$
for any $\epsilon>0$ sufficiently small, where we have used $\phi_\epsilon^{C_g}\ls 3(1+C_g\epsilon)\ls 3+3C_g:=C'_g$ whenever $\epsilon\ls1$. Therefore, we have
			\begin{equation}\label{equ-2.14}
				\begin{split}
				\int_{t_*+\epsilon}^{T}\int_{t-\epsilon}^{t+\epsilon}\int_{\Omega_\epsilon}e^{u^s}_\epsilon(x)\dm(x){\rm d}s{\rm d}t\
				&\ls C'_g \int_{t_*+\epsilon}^{T}\int_{t-\epsilon}^{t+\epsilon} E[u^s] {\rm d}s{\rm d}t\\
				&\ls C'_g\int_{t_*+\epsilon}^{T}2\epsilon\cdot \max_{s\in[{t-\epsilon},{t+\epsilon}]} E[u^s] {\rm d}t \\
				&\ls 2\epsilon\cdot C'_g\int_{t_0+\epsilon}^{T}E[u^{t_*}]{\rm d}s,
					\end{split}
			\end{equation}	
			for all $\epsilon>0$ sufficiently small, where the last inequality we have used Lemma \ref{lem-2.4} (ii). Integrating (\ref{equ-2.12}) over $\Omega_\epsilon\times(t_*+\epsilon,T)$, noticing $|B_\epsilon(x)|\ls c_{n,K}\epsilon^n$ (by Bishop inequality),  and  substituting (\ref{equ-2.13}) and (\ref{equ-2.14}) , we obtian the desired estimate (\ref{equ-2.11}).	  	
			
			Finally, noticing that $\chi_{\Omega_\epsilon\times(t_*+\epsilon,T)}\to 1$ almost all in $\Omega\times(t_*,T)$ as $\epsilon\to0^+$, by (\ref{equ-2.11})
	and the Fatou's lemma, letting $\epsilon\to0^+$, we conclude $u(x,t)\in W^{1,2}(\Omega\times(t_*,T),Y)$.  	
			The proof is finished.			
						\end{proof}

\section{The subsolution property of the  distance between harmonic map heat flows}

Let $\Omega\subset M$ be a bounded open  domain,   and let $(Y,d_Y)$ be a $CAT(0)$ space. Assume that $u^t(x)=u(x,t)$ is a   weak solution of harmonic map heat flow in $W^{1,2}_\psi(\Omega,Y)$ with the initial data $u_0\in W^{1,2}(\Omega,Y)$ and boundary data $\psi\in  W^{1,2}(\Omega,Y)$.
The main result in this section is Theorem \ref{thm-3.4}, a subsolution property of the distance between two weak solutions of harmonic map heat flow.

We first prove the following boundedness for $u(x,t)$.

\begin{lemma}\label{lem-3.1}
	For any fixed $M_0>0$, $P_0\in Y$, if  $u_0(x)\in  \overline{B_{M_0}(P_0)}$ for almost all $x \in \Omega$, then for any $t>0$, $u^t(x)\in \overline{B_{M_0}(P_0)}$ for almost all $x \in \Omega$.
\end{lemma}
\begin{proof}
	Fix any $h>0$ and let  $u:=J_{h}(u_0)$. We first show that $u(x)\in \overline{B_{M_0}(P_0)}$ for almost all $x\in \Omega$. \ Let $\pi_0: Y\to \overline{B_{M_0}(P_0)}$ be the projection from $Y$ to $\overline{B_{M_0}(P_0)}$, defined by
	$$\pi_0(P)=\begin{cases} P,& {\rm if }\ P\in \overline{B_{M_0}(P_0)},\\
	{\rm the\ nearest \ point\ from} \ P\ {\rm to}\ \overline{B_{M_0}(P_0)},& {\rm if }\ P\not\in \overline{B_{M_0}(P_0)}.
		\end{cases}
	 $$
Since the ball $\overline{B_{M_0}(P_0)}$ is convex, $\pi_0$ is well-defined, and is $1$-Lipschitz. Therefore, by Proposition \ref{prop-2.1} (4),  the map $\pi_0\circ u\in W^{1,2}_u(\Omega,Y)$ and   $E[\pi_0\circ u]\ls E[u].$

By $u_0(x)\in \overline{ B_{M_0}(P_0)}$ for almost all $x \in \Omega$ and the assumption that $Y$ is $CAT(0)$, we have
$$d_Y\big(u_0(x),\pi_0(u(x))\big)\ls d_Y(u_0(x),u(x)),$$ for almost all $x \in \Omega$.	
This implies $D(u_0,\pi_0\circ u)\ls D(u_0,u)$; and then $\pi_0\circ u$ is also a minimizer of $v\mapsto E[v]/2+\frac{1}{2h}D^2(u_0,v)$. From the uniqueness of $J_{h}(u_0)$, we conclude that $D(u,\pi_0\circ u)=0$. This implies  $d_Y(\pi_0\circ u(x),u(x))=0$ for almost all $x\in \Omega,$ and then  
$u(x)\in  \overline{B_{M_0}(P_0)}$ for almost all $x\in \Omega$.

For any $t>0$ and any $m\in\mathbb N$, from the above argument, we see that the image of $J_{t/m}(u_0)$ is contained in $\overline{B_{M_0}(P_0)}$. By repeating, we see that the image of $J^m_{t/m}(u_0)$ is contained in $\overline{B_{M_0}(P_0)}$. Letting $m\to+\infty$, we conclude that
$\lim_{m\to+\infty}J^m_{t/m}(u_0)(x)$ is in $\overline{B_{M_0}(P_0)}$ for almost all $x\in \Omega$. That is, $u^t(x)\in \overline{B_{M_0}(P_0)}$ for almost all $x \in \Omega$.	
\end{proof}

We shall consider the variations given by two Sobolev maps. Letting $u,v\in W^{1,2}(\Omega, Y)$, we define a function $R_{u,v}$ on $\Omega$  by
\begin{equation}\label{equ-3.1}
					R_{u,v}(x):=\liminf_{\epsilon\to0^+}R^{u,v}_\epsilon(x),
				\end{equation}	
					where $$ R^{u,v}_\epsilon(x):= \frac{c_{n,2}}{\epsilon^n}\int_{B_\epsilon(x)\cap \Omega}\left(\frac{d_Y\big(u(x),u(y)\big)-d_Y\big(v(x),v(y)\big) }{\epsilon}\right)^2\dm(y).$$
It is easy to check that $R_{u,v}\in L^1(\Omega)$. Indeed, by
  $$ \Big(d_Y\big(u(x),u(y)\big)-d_Y\big(v(x),v(y)\big)\Big)^2 \ls  2d^2_{Y}(v(x),v(y))+2d^2_{Y}(u(x),u(y)), $$
we have, for any $\epsilon>0$, that
\begin{equation}\label{equ-3.2}
	R^{u,v}_\epsilon(x)\ls 2e^u_\epsilon(x)+2e^v_\epsilon(x).
\end{equation}
 Therefore, by (\ref{equ-3.1}) and using the fact that $e^u_\epsilon(x)\to e_u(x)$ and $e^v_\epsilon(x)\to e_v(x)$ for almost all $x\in \Omega$, as $\epsilon\to0^+$, we have
 $$R_{u,v}(x)\ls 2e_u(x)+2e_v(x),\quad \mu{\rm-a.e.}\ x\in\Omega.$$
 This implies $R_{u,v}\in L^1(\Omega)$.

\begin{lemma}\label{lem-3.2}
	Let  $u,v\in W^{1,2}(\Omega,Y)\cap L^\infty(\Omega,Y)$. Then the function
	$$w(x):=d^2_Y\big(u(x),v(x)\big) \in W^{1,2}(\Omega)\cap L^\infty(\Omega).$$

	 Suppose $\phi\in  Lip_c (\Omega)$ with $0\ls \phi\ls1.$ For any $x\in \Omega$, $u_\phi(x)$ and $ v_\phi(x)$ denote the points in the (unique) geodesic connecting $u(x)$ and $v(x)$ such that
	\begin{equation*}
	\begin{split}
		d_Y\big(u_\phi(x),u(x)\big)&= \phi(x)\cdot d_Y\big(u(x),v(x)\big),\\
		 d_Y\big(v_\phi(x),v(x)\big)&= \phi(x)\cdot d_Y\big(u(x),v(x)\big).	
	\end{split}	
	\end{equation*}
	Then the maps $x\mapsto u_\phi(x)$ and $x\mapsto v_\phi(x)$ are in $W^{1,2}(\Omega,Y)\cap L^{\infty}(\Omega,Y)$. Moreover, it holds
	\begin{equation}\label{equ-3.3}
	\begin{split}
		&E[u_\phi]+ E[v_\phi]-E[u]- E[v]\\
		 \ls & -\int_\Omega\ip{\nabla \phi}{\nabla((1-2\phi) w)}
\dm -2\int_\Omega\big( \phi - \phi^2 \big) R_{u,v}  \dm.
	\end{split}
				\end{equation}
				 \end{lemma}
	
 \begin{proof} Thoughout of this proof, we always denote by $d_{PQ}:=d_Y(P,Q)$ for any $P,Q\in Y$, for simplifying the notations.

(i) We first show that   $w\in W^{1,2}(\Omega) \cap L^\infty(\Omega)$.
 Let $M_0>0$ and some $P_0\in Y$ such that $u(x), v(x)\in \overline{B_{M_0}(P_0)}$ for almost all $x \in \Omega$.
  The triangle inequality implies $w^{1/2}(x)=d_{u(x)v(x)}\ls 2M_0$ for almost all $x\in\Omega$, and
  \begin{equation*}
  	\begin{split}
  		|w^{1/2}(x)-w^{1/2}(y)|^2=|d_{u(x) v(x)}- d_{u(y) v(y)}|^2\ls 2d^2_{u(x) u(y)}+ 2d^2_{v(x) v(y)}
  	 		  	\end{split}
  \end{equation*}
  for almost all $x,y\in\Omega$. This implies,   for any $\epsilon>0$, that
$$e^{\sqrt w}_\epsilon(x)\ls 2e^u_\epsilon(x)+2e^v_\epsilon(x). $$
Thus, we have $w^{1/2}\in W^{1,2}(\Omega),$ by the definition of $e_u$ and $e_v$. Hence we conclude that $w^{1/2}\in W^{1,2}(\Omega) \cap L^\infty(\Omega)$. It follows  $w\in W^{1,2}(\Omega) \cap L^\infty(\Omega)$.

(ii) Now we want to show that $u_\phi$ and $v_\phi$ are in  $W^{1,2}(\Omega) \cap L^\infty(\Omega)$.
From $d_{u(x)v(x)}\ls 2M_0$ and  $u_\phi, v_\phi$ in geodesic $u(x)v(x)$, we have
  $ d_{u_\phi(x),P_0}\ls 3M_0$ and  $ d_{v_\phi(x),P_0}\ls 3M_0$  for almost all $x\in\Omega$. Therefore, $u_\phi$  and   $v_\phi$ are  in $L^{\infty}(\Omega,Y)\subset L^2(\Omega,Y)$, since $\Omega$ is bounded.

 We put $P=u(x)$, $Q=u(y)$, $R=v(y)$, $S=v(x)$, $\lambda=\phi(x)$ and $\mu=\phi(y)$. Then $P_\lambda=u_\phi(x)$ and  $Q_{\mu}=u_\phi(y)$. By using (\ref{equ-pq-2}) and the fact $d_{QR}\ls 2M_0$, we get
    \begin{equation*}
    \begin{split}
 d_{u_{\phi}(x)u_\phi(y)}^2=d^2_{P_\lambda Q_\mu}\ls & 2 d^2_{PQ}+2  d^2_{RS}+2|\lambda-\mu|^2\cdot 4M_0^2\\
 = & 2d^2_{u(x)u(y)}+2d^2_{v(x)v(y)}+8M_0^2\cdot |\phi(x)-\phi(y)|^2.
    	 \end{split}
 	 	\end{equation*}
 This implies for any $\epsilon>0$ that
 $$e^{u_\phi}_\epsilon(x)\ls 2e^u_\epsilon(x)+2e^v_\epsilon(x)+8M^2_0\cdot e^\phi_\epsilon(x).$$
 Therefore, letting $\epsilon\to0^+$, we conclude that $u_\phi\in W^{1,2}(\Omega,Y)$. The same argument states $u_\phi\in W^{1,2}(\Omega,Y)$ too.

(iii) In the last step, we want to show (\ref{equ-3.3}). Let $x,y\in \Omega$. We continue to put $P=u(x)$, $Q=u(y)$, $R=v(y)$, $S=v(x)$, $\lambda=\phi(x)$ and $\mu=\phi(y)$. Then $P_\lambda=u_\phi(x)$ and  $P_{1-\lambda}=v_\phi(x)$. By (\ref{equ-pq-1}) in Lemma \ref{lem-2.3}, we get
 \begin{equation*}
 \begin{split}
   	d_{P_\lambda Q_\mu}^2+d_{P_{1-\lambda} Q_{1-\mu}}^2 & \ls [\mu(1-\lambda)+\lambda(1-\mu)]\big(d^2_{PR}+d^2_{QS}\big)\\
 	&\quad+ [(1-\lambda)(1-\mu)
   +\lambda\mu]\big(d_{PQ}^2+d^2_{RS}\big)\\
   &\quad -2\lambda(1-\lambda)d^2_{PS}-2\mu(1-\mu)d^2_{QR}.
 \end{split}
 	 	\end{equation*}
 	 	
 	 	Combining with Lemma \ref{lem-2.3} (1), we get
 \begin{equation*}
 \begin{split}
   	d_{P_\lambda Q_\mu}^2+d_{P_{1-\lambda} Q_{1-\mu}}^2 & \ls d^2_{PQ}+d^2_{RS}+[\lambda+\mu-2\lambda\mu]\cdot\Big(d^2_{QR}+d^2_{PS}-(d_{RS}-d_{PQ})^2\Big)\\
   &\quad -2\lambda(1-\lambda)d^2_{PS}-2\mu(1-\mu)d^2_{QR}\\
   &   = d^2_{PQ}+d^2_{RS}-(\lambda-\mu)\Big[ (1-2\lambda) d^2_{PS}-(1-2\mu)d^2_{QR} \Big]\\
   &\quad -[(2\lambda-2\lambda^2)+(\mu-\lambda)(1-2\lambda)]\cdot (d_{RS}-d_{PQ})^2.
 \end{split}
 	 	\end{equation*} 	 	
Thus, we have
  \begin{equation*}
 \begin{split}
   	&d_{u_{\phi}(x)u_\phi(y)}^2+d_{v_{\phi}(x)v_\phi(y)}^2 -\big(d_{u (x)u (y)}^2+d_{v (x)v (y)}^2\big)\\
 \ls\ &  -\big(\phi(x)-\phi(y)\big)\Big[\big (1-2\phi(x)\big) w(x)-\big(1-2\phi(y)\big)w(y) \Big]\\
   &   +[-(2\phi(x)-2\phi^2(x))+|\phi(y)-\phi(x)|]\cdot (d_{v(x)v(y)}-d_{u(x)u(y)})^2.
 \end{split}
 	 	\end{equation*} 	
where we have used $w(x,y)=d^2_Y\big(u(x),v(x)\big)$ and $-(\mu-\lambda)(1-2\lambda)\ls |\mu-\lambda|$.
For any $\epsilon>0$, using the definition of $e^u_\epsilon(x)$ and noticing that $|\phi(y)-\phi(x)|\ls c_1\epsilon$ for any $y\in B_\epsilon(x)$, we have
\begin{equation*}
\begin{split}
	&e^{u_\phi}_\epsilon(x)+e^{v_\phi}_\epsilon(x)-\big(e^{u}_\epsilon (x)+e^v_\epsilon(x)\big) \\
	\ls \ &-c_{n,2}\int_{B_\epsilon(x)}\frac{\big(\phi(x)-\phi(y)\big)\big[\big (1-2\phi(x)\big) w(x)-\big(1-2\phi(y)\big)w(y) \big]}{\epsilon^{n+2}}\dm(y)\\
	& \ + \big[-2(\phi(x)-\phi^2(x))+c_1\cdot \epsilon\big]\cdot R^{u,v}_\epsilon(x).
			\end{split}	
\end{equation*}
Integrating  this inequality over ${\rm supp}(\phi)$, and letting $\epsilon\to0^+$, the estimate (\ref{equ-3.3}) follows from Fatou's lemma and the fact $\phi-\phi^2\gs0$.
 \end{proof}

 The following lemma deals with the time derivative for the weak solutions of the harmonic map heat flow.
\begin{lemma}\label{lem-3.3}
 Let $u^t(x), v^t (x)$ be two weak solutions of harmonic map heat flow from $\Omega$ to $Y$. Suppose that $\phi\in Lip_c(\Omega)$ such that  $0\ls \phi\ls 1$. 	
 Denoted by $u^{t}_\phi (x), v^{t}_\phi (x)$  are  the points in the (unique) geodesic connecting $u^t(x)$ and $v^t (x)$ such that
 	\begin{equation*}
 	\begin{split}
 		 d_Y\big(u^{t}_\phi (x), u^t(x)\big)&= \phi(x)\cdot d_Y\big(u^t(x),v^t (x)\big),\\
 		  	  d_Y\big(v^{t}_\phi (x), v^t (x)\big)&= \phi(x)\cdot d_Y\big(u^t(x),v^t (x)\big).
 	\end{split}	 			
 	\end{equation*}
 	  Then  we have for any $0<s<t$ that
 		\begin{equation}\label{equ-3.4}
 		\begin{split}
 			 		 \int_\Omega\phi(x)&\left(d^2_Y\big(u^t(x),v^t (x)\big)-d^2_Y\big(u^{t-s}(x),v^{t-s}(x)\big)\right){\rm d}\mu(x)\\
 		\ls\ &\Big(D^2(u^t,u^{t-s})+D^2(u^{t}_\phi ,u^t)-D^2(u^{t}_\phi ,u^{t-s})\Big)\\
 		& +\Big(D^2(v^t ,v^{t-s})+D^2(v^{t}_\phi ,v^t )-D^2(v^{t}_\phi ,v^{t-s})\Big).
 		\end{split}
 	\end{equation}
\end{lemma}
\begin{proof}
	For any $x\in\Omega$, by applying Lemma \ref{lem-2.3}(3) to $P=u^t(x),$ $S=v^t (x)$, $Q=u^{t-s}(x)$ and $\lambda=\phi(x)$, we get
		\begin{equation*}
	\begin{split}
		&\ \phi(x)\Big(d^2_Y(v^t (x),u^t(x))+d^2_Y(u^t(x),u^{t-s}(x))-d^2_Y(v^t (x),u^{t-s}(x))\Big)\\
		\ls& \ d^2_Y(u^{t}_\phi (x),u^t(x))+d^2_Y(u^t(x),u^{t-s}(x))-d^2_Y(u^{t}_\phi (x),u^{t-s}(x)),
	\end{split}
			\end{equation*}
and similarly
	\begin{equation*}
	\begin{split}
		&\ \phi(x)\Big(d^2_Y(v^t (x),u^t(x))+d^2_Y(v^t (x),v^{t-s}(x))-d^2_Y(u^t(x),v^{t-s}(x))\Big)\\
		\ls& \ d^2_Y(v^{t}_\phi (x),v^t (x))+d^2_Y(v^t (x),v^{t-s}(x))-d^2_Y(v^{t}_\phi (x),v^{t-s}(x)).
	\end{split}
			\end{equation*}
			By applying Lemma \ref{lem-2.3}(1) to $P=u^t(x),$ $Q=u^{t-s}(x)$, $R=v^{t-s}(x)$ and $S=v^{t}(x)$,  we have \begin{equation*}
\begin{split}
&d^2_Y(u^t(x),v^{t-s}(x))+d^2_Y(v^t (x),u^{t-s}(x))\\
\ls \ & d^2_Y(u^t(x),v_{t}(x))+d^2_Y(v^t (x),v^{t-s}(x))+ d^2_Y(v^{t-s}(x),u^{t-s}(x))+d^2_Y(u^{t-s}(x),u^{t}(x)).
\end{split}	
\end{equation*}					
Summing up the above three inequalities and integrating over $\Omega$, the desired estimate (\ref{equ-3.4}) follows.
\end{proof}

The main result in this section is the following.
\begin{theorem}\label{thm-3.4}
Let $\Omega\subset M$ be a bounded open domain   and let $(Y,d_Y)$ be a $CAT(0)$ space. Assume that  $u^t(x), v^t(x)$ are two bounded weak solutions of the harmonic map heat flow from $\Omega$ to $Y$.  Then, for any $0<t_*<T<+\infty$, the function
$$w(x,t):=d^2_Y\big(u(x,t),v(x,t)  \big)\in W^{1,2}(\mathcal Q)\cap L^\infty(\mathcal Q),$$  where $\mathcal Q:=\Omega\times(t_*,T)$, and satisfies
 \begin{equation}\label{equ-3.5}
 \big({\Delta}-\partial_t\big) w(x,t)\gs 2R_{u^t,v^t}(x)  	
  \end{equation}
 	 in the sense of distributions, for almost all $t\in(t_*,T)$, where $R_{u,v}(x)$ is defined in (\ref{equ-3.1}).
\end{theorem}

\begin{proof}
From Proposition \ref{prop-2.5} and  the assumption that $u$ and $v$ are bounded, we have that both $u(x,t)$ and $v(x,t)$ are in $W^{1,2}(\mathcal Q,Y)\cap L^\infty(\mathcal Q,Y)$. This implies $w^{1/2}\in W^{1,2}(\mathcal Q)\cap L^\infty(\mathcal Q)$, since
$$|w^{1/2}(x,t)-w^{1/2}(y,s)|\ls d_Y(u^t(x),u^s(y))+d_Y(v^t (x),v^s(y))$$
for any $(x,t),(y,s)\in \mathcal Q$. Hence, we have $w \in W^{1,2}(\mathcal Q)\cap L^\infty(\mathcal Q)$.

 Let $t\in(t_*,T)$ and let $\phi\in  Lip_c (\Omega)$ with $0\ls \phi\ls1$. For any $s\in (0,t)$, by using (\ref{equ-2.9}) to $u=u^t$ and $v^t$, (with taking $\bar u:=u^{t}_\phi$ and $\bar v:=v^t_\phi$),
 we have
 $$D^2(u^t,u^t_\phi)-D^2(u^{t-s},u^t_\phi)+D^2(v^t,v^t_\phi)-D^2(v^{t-s},v^t_\phi)\ls s\big(E[u^t_\phi]+E[v^t_\phi]-E[u^t]-E[v^t]\big).$$
Summing with (\ref{equ-3.4}), we have
\begin{equation*}
\begin{split}
	&\ \int_\Omega \phi(x)\left(w(x,t) -w(x,t-s) \right){\rm d}\mu(x) \\
	\ls & \ D^2(u^t,u^{t-s})+ D^2(v^t ,v^{t-s}) +  s\Big( E[u^{t}_\phi ]-E[u^t]+  E[v^{t}_\phi ]-E[v^t]\Big).
\end{split}
 	\end{equation*}
By combining this with (\ref{equ-3.3}),  we get
\begin{equation*}
\begin{split}
	&\ \int_\Omega \phi(x)\left(w(x,t) -w(x,t-s) \right){\rm d}\mu(x) \\
	\ls & \ D^2(u^t,u^{t-s})+ D^2(v^t ,v^{t-s}) \\
	&-s\int_\Omega\ip{\nabla \phi}{\nabla((1-2\phi) w)}
\dm -2s \int_\Omega\Big( \phi - \phi^2 \Big) R_{u^t,v^t} \dm.
\end{split}
 	\end{equation*}
Letting $L$ be a Lipschitz constant of $u^t$ and $v^t$ in $[t_*,T]$ (see (\ref{equ-2.10}), replacing $\phi$ by $\sqrt s\cdot \phi$ in this inequality, and dividing $s\sqrt s$, we obtain
 \begin{equation}\label{equ-3.6}
\begin{split}
 \int_\Omega    \phi(x)\frac{w(x,t) -w(x,t-s) }{s}{\rm d}\mu(x)
	 \ls &L\sqrt s- \int_\Omega   \left(\ip{\nabla \phi}{\nabla  w}+2R_{u^t,v^t}\cdot \phi \right)\dm\\
	 &\ +\sqrt s\left( \int_\Omega2\ip{\nabla \phi}{\nabla(\phi w)}+2  \phi^2 R_{u^t,v^t} \dm\right) .
\end{split}
 	\end{equation}
Noticing  $w(x,t)\in W^{1,2}(\Omega\times (t_*,T))$. This implies
	$$	\lim_{s\to0^+}  \frac{w(x,t) -w(x,t-s)}{s}=  \partial_t w(x,t) \quad {\rm in}\ \ 	L^2(\Omega),$$
	for almost all $t\in (t_*,T)$. Therefore, letting $s\to0^+$ in (\ref{equ-3.6}), we have
		$$\int_{\Omega}  \phi(x)   \partial_t w(x,t) {\rm d}\mu
	 \ls  - \int_{\Omega}  \left(\ip{\nabla \phi}{\nabla  w}+2\phi  R_{u^t,v^t } \right) \dm,$$
	for any $\phi\in Lip_c(\Omega)$ with $0\ls \phi\ls 1$, for almost all $t\in (t_*,T)$.	 That is,
	$$({\Delta}-\partial_t) w\gs 2R_{u^t,v^t }$$
	in the sense of distributions on $\Omega$, for almost all $t\in (t_*,T)$.	
\end{proof}

The following corollary will be used later.
\begin{corollary}\label{coro-3.5}
		Let $\Omega\subset M$ be a bounded open  domain, and let $(Y,d_Y)$ be a $CAT(0)$ space, and let $P\in Y$. Assume that $u^t(x) $ is a bounded weak solution of the harmonic map heat flow from $\Omega$ to $Y$. Let $0<t_*<T<+\infty$. Then
 \begin{equation}\label{equ-3.7}
 \big({\Delta}-\partial_t\big) d^2_Y(P,u^t(x))\gs 2  e_{u^t}(x)	
  \end{equation}
 	 in the sense of distributions, for almost all $t\in(t_*,T)$. And
 	 \begin{equation}\label{equ-3.8}
 \big({\Delta}-\partial_t\big) d_Y(P,u^t(x))\gs 0	
  \end{equation}
 	 in the sense of distributions, for almost all $t\in(t_*,T)$. 	
 \end{corollary}
\begin{proof}
	Taking $v^t (x)\equiv P$ in Theorem \ref{thm-3.4}, the desired assertion (\ref{equ-3.7}) follows, since $R^{u^t,P}_\epsilon=e^{u^t}_\epsilon$ and $R_{u^t,P}=e_{u^t}.$

	We put
		$f_P(x,t):= d_Y\big(u(x,t), P \big).$
		 By (\ref{equ-3.7}) and noticing that $|\nabla f_P|^2(x)\ls e_{u^t}(x)$ for almost all $x\in \Omega$, we have
		$$(\Delta-\partial_t) f^2_P\gs 2e_{e_{u^t}}\gs 2|\nabla f_P|^2$$
		in $\Omega$ in the sense of distributions. By the chain rule and $f_P\in W^{1,2}_{\rm loc}(\Omega)\cap L_{\rm loc}^\infty(\Omega)$, we conclude that, for any $\delta\in(0,1)$,  the function $g_\delta:=(f_P^2+\delta)^{1/2}$ satisfies
		$$ 2g_\delta(\Delta-\partial_t)g_\delta\gs 2|\nabla f_P|^2-|\nabla g_\delta|^2=2|\nabla f_P|^2\Big(1-\frac{f^2_P}{g^2_\delta}\Big)\gs 0$$
		 in $\Omega$ in the sense of distributions. Since $g_\delta\gs \delta>0$, it follows that $g_\delta$ is a supersolution of the heat equation in the sense of distributions. Letting $\delta\to0^+,$ we have $f_P$ is a supersolution of the heat equation in the sense of distributions. This proves (\ref{equ-3.8}).
\end{proof}

\section{Local H\"older Continuity}

Let $M$ be $n$-dimensional Riemannian manifold and $(Y, d_Y)$ be a $CAT(0)$ space.
We first consider the following general result:
\begin{theorem}\label{thm-4.1}
	Let $U_T:=U\times(0,T)$ be a cylinder, where $U\subset M$ is bounded open domain, and $T>0$.
	Suppose that a map $u(x,t)\in W^{1,2}(U_T,Y)\cap L^\infty(U_T,Y)$ satisfies:
	\begin{enumerate}
		\item there exists some $a_1>0$ such that for any $P\in Y$, it holds
	\begin{equation}\label{equ-4.1}
		(\Delta -\partial_t)d^2_Y\big(P,u(x,t)\big)\gs a_1\cdot e_{u^t}(x)
	\end{equation}
	 in the sense of distributions, and
	\item there exists a constant $L>0$ such that
\begin{equation}\label{equ-4.2}
	D\big(u^t,u^{t+s}\big)\ls L\cdot s,\quad \forall t, t+s\in(0,T).
\end{equation}
		\end{enumerate}	
		Then $u$ is locally H\"older continuous in $U_T$ (i.e., there exists a locally H\"older continuous function $\tilde u$ such that $\tilde u(x,t)=u(x,t)$ for almost all $(x,t)\in U_T.)$.
			\end{theorem}

 The argument is an extension of Jost's proof \cite{Jost97} from harmonic maps to the parabolic setting. For any $r>0$ and  $P\in Y$, we denote by
 $$v_{P,+,r}(x,t):=\sup_{Q_r(x,t)}d^2_Y\big(P,u(y,s)\big),\quad v_{P,-,r}(x,t):=\inf_{Q_r(x,t)}d^2_Y\big(P,u(y,s)\big),$$
where $Q_r(x,t):=B_r(x)\times(t-r^2,t+r^2)$, and also $$  v_{P,r}(x,t):= \fint_{Q_r(x,t)}d^2_Y\big(P,u(y,s)\big) \dm(y){\rm d}s.$$

Before giving the proof of Theorem 4.1, we need several lemmas. The first one is a parabolic version of \cite[Corollary 1]{Jost97}.
\begin{lemma}\label{lem-4.2}
	Let $u$ be as in the above Theorem 4.1. Assume that  $u(U_T)\subset B_{R_0}(P_0)$ for some $R_0>0$ and $P_0\in Y$.  Then there exist  constants $\delta_0\in(0,1/2)$ and $C_1>0$ (depending only on the lower bound of Ricci curvature on $U$, ${\rm diam}(U)$, $\mu(U)$, $R_0$ and $L$), such that for any cylinder $Q_r(x_0,t_0)$ with $Q_{8r}(x_0,t_0)\subset U_T$ and for any $P\in B_{2R_0}(P_0)$,
	we have
	\begin{equation}\label{equ-4.3}
		v_{P,+,r}(x_0,t_0)\ls \left(1-\delta_0\right)v_{P,+,4r}(x_0,t_0)+\delta_0 \cdot v_{P,r}(x_0,t_0)+C_1r^{2}.
	\end{equation}
	
	Let $m=m(\delta_0)\in\mathbb N$ such that $(1-\delta_0)^m\ls 4^{-2}$. Then   for any $\epsilon\in(0,1/4)$ and any $r>0$ with $Q_{8r}(x_0,t_0)\subset U_T$, there exists some $r'\in [\epsilon^m r,r/4]$ such that
\begin{equation}\label{equ-4.4}
	v_{P,+,\epsilon^m r}(x_0,t_0)\ls \epsilon\cdot  v_{P,+,r}(x_0,t_0)+ v_{P,r'}(x_0,t_0)+ \frac{C_1}{\delta_0}\cdot r^{2},
\end{equation}
where  the constant $C_1$ is in (\ref{equ-4.3}).	 Here this $r'$ may depend on $P, \epsilon$ and $r$.	\end{lemma}

\begin{proof}Fix any $P\in B_{2R_0}(P_0)$ and cylinder $Q_{r}(x_0,t_0)$ with $Q_{8r}(x_0,t_0)\subset U_T$. The function $$g(y,t):=v_{P,+,4r}(x_0,t_0)-d_Y^2\big(P,u(y,t)\big)$$
 is nonnegative in $Q_{4r}(x_0,t_0)$ and $(\Delta-\partial_t)g \ls0 $ in $Q_{4r}(x_0,t_0)$ in the sense of distributions. By the weak Harnack inequality, we obtain
\begin{equation}\label{equ-4.5}
	 \fint_{Q_r(x_0,t_0-3r^2)}g(y,t) \dm(y){\rm d}t\ls c_1\cdot \inf_{Q_r(x_0,t_0)}g=c_1\cdot\Big(v_{P,+,4r}(x_0,t_0)-v_{P,+,r}(x_0,t_0)\Big),
\end{equation}
for some positive constant $c_1$ (depending only on the lower bound of Ricci curvature on $U$, ${\rm diam}(U)$ and $\mu(U)$). We always take  $c_1>2$.

To simplify the notations, we denote $v_{+,r}:=v_{P,+,r}(x_0,t_0)$ and $v_r:=v_{P,r}(x_0,t_0)$ in the sequel of this proof.
By $u(U_T)\subset B_{3R_0}(P)$  and the triangle inequality, we have
$$|d^2_Y(P,u^t(y))-d^2_Y(P,u^{t-3r^2}(y))|\ls6R_0\cdot  d_Y(u^t(y),u^{t-3r^2}(y)) .$$
Hence, we have
\begin{equation*}
	\begin{split}
		\int_{B_r(x_0)} d^2_Y\big(P,u^t(y)\big) \dm(y)&\ls \int_{B_r(x_0)} d^2_Y\big(P,u^{t-3r^2}(y)\big) \dm(y)+6R_0\cdot \int_{B_r(x_0)} d_Y(u^t(y),u^{t-3r^2}(y))\dm(y)\\
		&\ls \int_{B_r(x_0)} d^2_Y\big(P,u^{t-3r^2}(y)\big) \dm(y)+6R_0\cdot D(u^t,u^{t-3r^2})\\
 \end{split}
\end{equation*}
From the assumption $D(u^t,u^{t+s})\ls Ls$, the left-hand side of  (\ref{equ-4.5}) satisfies
  \begin{equation*}
	\begin{split}
		&\frac{1}{2r^2|B_r(x_0)|}\int_{Q_r(x_0,t-3r^2)}g(y,t) \dm(y){\rm d}t \\
		=\ & v_{+,4r}-\frac{1}{2r^2|B_r(x_0)|}\int_{t_0-4r^2}^{t_0-2r^2}\int_{B_r(x_0)} d^2_Y\big(P,u(y,t)\big) \dm(y){\rm d}t	\\
		\gs\ & v_{+,4r}-\frac{1}{2r^2|B_r(x_0)|}\int_{t_0-4r^2}^{t_0-2r^2}\left(\int_{B_r(x_0)}  d_Y\big(P,u(y,t-3r^2)\big)\dm(y)+18R_0\cdot Lr^2 \right){\rm d}t		\\
			=\ & v_{+,4r}-v_{r}-18R_0\cdot Lr^2.
							\end{split}
\end{equation*}
 The combination with (\ref{equ-4.5}) implies
$$v_{+,r}\ls \left(1-\frac{1}{c_1}\right)v_{+,4r}+\frac{1}{c_1}v_{r}+\frac{c_2}{c_1}r^{2},$$
where $c_2:= 18R_0\cdot L$. This is the desired estimate (4.3) with $\delta_0=1/c_1\in(0,1/2)$ and $C_1:=c_2/c_1.$

 To show (\ref{equ-4.4}), by iterating (\ref{equ-4.3}), we get for any $\nu\in\mathbb N$ that
\begin{equation*}
\begin{split}
	v_{+,4^{-\nu}r}&\ls (1-\delta_0)^\nu\cdot v_{+,r}+ \sum_{j=1}^\nu(1-\delta_0)^{\nu-j}\left(\delta_0 v_{4^{-j}r}+C_1(4^{-j}r)^{2} \right) \\
	&\ls (1-\delta_0)^\nu\cdot v_{+,r}+ \delta_0\sum_{j=1}^\nu(1-\delta_0)^{\nu-j}  v_{4^{-j}r}+C_1\sum_{j=1}^\nu(1-\delta_0)^{\nu-j}(4^{-j}r)^{2}  \\	
	&\ls (1-\delta_0)^\nu\cdot v_{+,r}+ \max_{1\ls j\ls \nu}\{v_{4^{-j}r}\}+ \frac{C_1 }{\delta_0}\cdot r^{2} ,
\end{split}	
\end{equation*}
where we have used $\sum_{j=1}^\nu(1-\delta_0)^{\nu-j} \ls \frac{1}{\delta_0}$ and
$(4^{-j})^{2}\ls 1.$

For any $\epsilon\in(0,1/4)$, we choose $\nu$ so that
$$(1-\delta_0)^{\nu}< \epsilon\ls  (1-\delta_0)^{\nu-1}$$ and take $r':=4^{-j_0}r$ such that $v_{4^{-j_0}r}= \max_{1\ls j\ls \nu}\{v_{4^{-j}r}\}.$
Since
$(1-\delta_0)^m\ls 4^{-2}$, we have  $\epsilon^m\ls (1-\delta_0)^{(\nu-1)m}\ls 4^{-2(\nu-1)}\ls 4^{-\nu},$ by $\nu\gs2$ (because $\delta_0<1/2$ and $\epsilon<1/4.$).
Therefore, we have $\epsilon^m r\ls 4^{-\nu}r\ls 4^{-j_0}r=r'\ls r/4.$ The proof is finished.
\end{proof}

 The following lemma is a parabolic analogue to \cite[Lemma 8]{Jost97}.
\begin{lemma}\label{lem-4.3}
	Let $u$ be as in the above Theorem 4.1.
	Assume that  $u(U_T)\subset B_{R_0}(P_0)$ for some $R_0>0$ and $P_0\in Y$.  Then there exists a  constant   $C_2>0$ (depending only on the lower bound of Ricci curvature on $U$, ${\rm diam}(U)$, $\mu(U)$ and $\delta_0$), such that for any cylinder $Q_r(x_0,t_0)$ with $Q_{8r}(x_0,t_0)\subset U_T$ and for any $P\in B_{2R_0}(P_0)$,
	we have
 \begin{equation}\label{equ-4.6}
	  r^2\fint_{Q_{r/2}(x_0,t_0) } e_{u^t}(y)\dm(y){\rm d}t \ls \frac{C_2}{a_1}\big(v_{P,+,4r}(x_0,t_0)-v_{P,+,r}(x_0,t_0)+C_1r^{2}\big),
	 \end{equation}
	 where $C_1$ and $\delta_0$ are given in (\ref{equ-4.3}).
	  	 \end{lemma}

\begin{proof}
 Let $\phi(y)$ be a smooth cut-off function such that $\phi|_{B_{r/2}(x_0)}\equiv1$, ${\rm supp}(\phi)\subset B_{r}(x_0)$, $0\ls\phi\ls1$, $|\nabla \phi|\ls c_1/r$, and $|\Delta\phi|\ls c_1/r^2$. $\eta(t)$ is a smooth function on $(t_0-r^2,t_0+r^2)$ such that $\eta|_{(t_0-r^2/4,t_0+r^2/4)}=1$, ${\rm supp}(\eta)\subset(t-r^2,t+r^2)$,  $ 0\ls \eta(t)\ls 1$ and $|\eta'(t)|\ls c_1/r^2.$

Fix any $P\in Y$. By (\ref{equ-4.1}), we have
	\begin{equation*}
		\begin{split}
			&\quad \frac{a_1}{|B_{r/2}(x_0)|}\int_{Q_{r/2}(x_0,t_0) } e_{u^t}(y)\dm(y){\rm d}t \\
			&\ls \frac{ a_1}{ |B_{r/2}(x_0)|}	\int_{Q_{r}(x_0,t_0)} e_{u^t}(y)  \phi(y)\eta(t) \dm(y) {\rm d}t\\
			&\ls  \frac{1}{  |B_{r/2}(x_0)|}\int_{Q_{r}(x_0,t_0)} (\Delta-\partial_t)  \big(d^2_Y(P,u(y,t))-v_{P,+,4r}\big)   \phi(y)\eta(t) \dm{\rm d}t 	\\
			&\ls  \frac{1}{  |B_{r/2}(x_0)|}\int_{Q_{r}(x_0,t_0)}  |d^2_Y(P,u)-v_{P,+,4r}|\cdot |(\Delta-\partial_t) (\phi(y)\eta(t))| \dm{\rm d}t  	\\
			&\ls \frac{c_1}{r^2  \cdot|B_{r/2}(x_0)|}\int_{Q_{r}(x_0)}\Big(v_{P,+,4r}-d^2_Y(P,u) \Big) \dm {\rm d}t\\
			&\ls \frac{c_2}{2r^2  \cdot|B_{r}(x_0)|}\int_{Q_{r}(x_0)}\Big(v_{P,+,4r}-d^2_Y(P,u) \Big) \dm {\rm d}t=c_2(v_{P,+.4r}-v_{P,r}),			\end{split}
	\end{equation*}
	where we have used $d^2_Y(P,u)-v_{P,+,4r}\ls  0$ for the fourth inequality, and doubling property of $\mu$ for the last inequality. Recall that (\ref{equ-4.3}) implies
	$$v_{P,+,4r}-v_{P,r}\ls \frac{1}{\delta_0}\Big(v_{P,+,4r}-v_{P,+,r}+C_1r^{2}\Big).$$
	 By combining these two inequalities, we get
	\begin{equation*}
		\frac{a_1}{|B_{r/2}(x_0)|}\int_{Q_{r/2}(x_0,t_0) } e_{u^t}(y)\dm(y){\rm d}t \ls \frac{c_2}{\delta_0}\big(v_{P,+,4r}-v_{P,+,r}+C_1r^{2}\big)	.
	\end{equation*}	
	This is the desired estimate (\ref{equ-4.6}) with $C_2=4c_2/\delta_0$.
\end{proof}

Let $f\in L^2(\mathcal M,Y)$, where $\mathcal M$ is a smooth Riemannian manifold, and let $E\subset \mathcal M$ be a bounded subset. Then the function
$$Y\ni q\mapsto \int_Ed^2_Y\big(f(y),q\big)\dm(y)$$
has a unique minimizer, which is called the center of mass of $f$ on $E$, written by $\overline{f}_E$. Moreover, it lies in the convex hull of $f(E).$ In particular, if $f(E)\subset B_R(P)\subset Y$ then $\overline{f}_E\in B_R(P)$, because any geodesic ball $B_R(P)$ is convex in $Y$.

 The following lemma is a Poincar\'e-type inequality in the parabolic setting.
 \begin{lemma}\label{lem-4.4}
 Let $u$ be as in the above Theorem 4.1. Assume that  $u(U_T)\subset B_{R_0}(P_0)$ for some $R_0>0$ and $P_0\in Y$.  Then there exists a  constant $C_3>0$ (depending only on the lower bound of Ricci curvature on $U$, ${\rm diam}(U)$ and  $\mu(U)$), such that for any cylinder $Q_r(x_0,t_0)$ with $Q_{8r}(x_0,t_0)\subset U_T$,
	we have
	\begin{equation}\label{equ-4.7}
		 \fint_{Q_r(x_0,t_0)}d_Y^2\big(u(y,t), \overline{u}_{Q_r(x_0,t_0)}\big)\dm(y){\rm d}t\ls  C_3 r^2\left(  \fint_{Q_{2r}(x_0,t_0) } e_{u^t}(y)\dm(y){\rm d}t+2R_0L\right).
	\end{equation}
	(Remark that the right-hand side involves only spatial derivatives.)
 \end{lemma}
 \begin{proof}
 Fix any cylinder $Q_r(x_0,t_0)$	 with $Q_{8r}(x_0,t_0)\subset U_T$. To simplify the notations, we denote by $Q_r:=Q_r(x_0,t_0)$ and $B_r=B_r(x_0)$.
 	
 	For any $t\in (t_0-r^2,t_0+r^2)$, from \cite[Proposition 2.5.2]{KS93}, we have
 	$$d_Y^2(\overline{u^t}_{B_r},\overline{u^{t_0}}_{B_r})\ls \int_{B_r} d^2_Y\big(u^t(y),u^{t_0}(y)\big)\dm(y)\ls D^2(u^t,u^{t_0})\ls L^2(t-t_0)^2\ls L^2r^4.$$
 Therefore, by triangle inequality and $u(U_T)\subset B_{R_0}(P_0)$, we get
\begin{equation*}
\begin{split}
	\big|d^2_Y\big(u^t(y),\overline{u^t}_{B_r}\big)-d^2_Y\big(u^{t}(y),\overline{u^{t_0}_{B_r}}\big)\big|
 &\ls  d_Y (\overline{u^t}_{B_r},\overline{u^{t_0}}_{B_r})\cdot \big|d_Y\big(u^t(y),\overline{u^t}_{B_r}\big)+d_Y\big(u^{t_0}(y),\overline{u^{t_0}_{B_r}}\big)\big|\\
	&\ls  Lr^2 \cdot (2R_0).
\end{split}	
\end{equation*} 	
 	Integrating over $B_r$, we obtain
\begin{equation*}
\begin{split}
	\Big|\int_{B_r}d^2_Y\big(u^t(y),\overline{u^{t_0}}_{B_r}\big) 	\dm(y)-\int_{B_r}d^2_Y\big(u^t(y),\overline{u^t}_{B_r}\big) \dm(y)	\Big|\\
\ls \ & 2R_0L|B_r|\cdot r^2.
\end{split}
\end{equation*}
 	This   implies
 \begin{equation*}
 \begin{split}
 		 \int_{Q_r}d^2_Y\big(u(y,t),\overline{u}_{Q_r}\big)\dm(y){\rm d}t&  \ls  \int_{Q_r}d^2_Y\big(u(y,t),\overline{u^{t_0}}_{B_r}\big)\dm(y){\rm d}t
 		  = \int_{t_0-r^2}^{t_0+r^2}\int_{B_r}d^2_Y\big(u^t(y),\overline{u^{t_0}}_{B_r}\big)\dm(y){\rm d}t\\
 		  &\ls
 		  \int_{t_0-r^2}^{t_0+r^2}\left(\int_{B_r}d^2_Y\big(u^t(y),\overline{u^{t}}_{B_r}\big)\dm(y) +2R_0L|B_r|\cdot r^2\right)	{\rm d}t	 \\
 		  &\ls \int_{t_0-r^2}^{t_0+r^2}\left( c_U\cdot r^2\int_{B_{2r}} e_{u^t}(y)\dm(y)+ 2R_0L|B_r|\cdot r^2\right)	{\rm d}t	 \\
 		  & \ls c_Ur^2\int_{Q_{2r}}e_{u^t}(y)\dm(y){\rm d}t+2R_0L|Q_r|\cdot r^2,
 		  	\end{split}
 \end{equation*}
where we have used Poincar\'e inequality (\ref{equ-poincare}) for the third inequality. This implies the desired assertion, by $|Q_{2r}|\ls c_U|Q_r|$ for some $c_U>0$.
 \end{proof}

  We are now in a position to prove Theorem \ref{thm-4.1}.
\begin{proof}[Proof of Theorem \ref{thm-4.1}]

Take any cylinder $Q_r(x_0,t_0)$ with $Q_{8r}(x_0,t_0)\subset U_T$. In the proof, we denote $Q_r:=Q_r(x_0,t_0)$.

Let $m$  be given in Lemma  \ref{lem-4.2} and  let $\epsilon=\frac{1}{64}$. For each   $r'\in(\epsilon^m r,r/4]$, we have
\begin{equation*}
		v_{P_1,r'}(x_0,t_0)=\fint_{Q_{r'}}d^2_Y(u(y,t),P_1)\dm(y){\rm d}t\ls c_1(m)\cdot \fint_{Q_{r/4}}d^2_Y(u,P_1)\dm{\rm d}t,
\end{equation*}
where $P_1=\overline{u}_{Q_{r/4}}$. The above Poincar\'e type inequality (\ref{equ-4.7}) states
$$\fint_{Q_{r/4}}d^2_Y(u,P_1)\dm{\rm d}t\ls c_2r^2\fint_{Q_{r/2}}e_{u^t}\dm{\rm d}t+c_2 r^2.$$
Combining these two inequality and (\ref{equ-4.6}), we obtain
\begin{equation}\label{equ-4.8}
	v_{P_1,r'}(x_0,t_0)\ls \frac{c_3}{a_1}\left(v_{P,+,4r}(x_0,t_0)-v_{P,+,r}(x_0,t_0)+C_1r^2\right)+c_3r^2,
\end{equation}
for any $P\in B_{2R_0}(P_0)$.

Fix any $\rho\in(0,\epsilon^mr)$ and take $P=\overline{u}_{Q_{\rho}}$ and noticing $d_Y(P,P_1)\ls \sup_{(y,t)\in Q_{\epsilon^mr}}d_Y(u(y,t),P_1)$ (since $P$ is in the convex hull of  $u(Q_{\epsilon^mr})$), we have
\begin{equation*}
	\begin{split}
		v_{P,+,\epsilon^mr}&=\sup_{  Q_{\epsilon^mr}}d^2_Y(u(y,t),P)\ls 2\sup_{  Q_{\epsilon^mr}}d^2_Y(u(y,t),P_1)+2d^2_Y(P,P_1)\\
		&\ls 4\sup_{  Q_{\epsilon^mr}}d^2_Y(u(y,t),P_1)=4v_{P_1,+,\epsilon^mr}.\end{split}
\end{equation*}
 Similar, since $P_1$ is in the convex hull of $u(Q_r)$, it holds
 \begin{equation*}
		v_{P_1,+, r} \ls 4v_{P,+,r}.
\end{equation*}

Combining with (\ref{equ-4.4}) and (\ref{equ-4.8}), we have
\begin{equation*}
\begin{split}
	v_{P,+,\epsilon^mr}\ls 4v_{P_1,+,\epsilon^mr}&\ls 4\epsilon v_{P_1,+,r}+4v_{P_1,r'}+4C_1r^2/\delta_0\\
	&\ls 16\epsilon v_{P,+,r}+ \frac{4c_3}{a_1}\left(v_{P,+,4r} -v_{P,+,r} +C_1r^2\right)+4c_3r^2+4C_1r^2/\delta_0\\
	& \ls \frac{1}{4}v_{P,+,r}+c_4( v_{P,+,4r} -v_{P,+,r})+c_4r^2.
	\end{split}
		\end{equation*}
 Therefore,
 \begin{equation*}
\begin{split}
	(1+c_4)v_{P,+,\epsilon^mr} &\ls  v_{P,+,\epsilon^mr}+c_4 v_{P,+,r} \\
		& \ls \frac{1}{4}v_{P,+,r}+c_4  v_{P,+,4r}  +c_4r^2\ls \left(\frac{1}{4}+c_4\right)v_{P,+,4r}  +c_4r^2.
	\end{split}
		\end{equation*}
 Now we put $\omega_P(\rho'):=v_{P,+,\rho'}$ for any $\rho'\in [\rho,r]$, and we get
  $$\omega_P(\epsilon^mr)\ls  \delta_1\cdot \omega(4r)+c_5r^2,$$
  where $\delta_1=\frac{1/4+c_4}{1+c_4}<1.$ By iteration, we have for all $\rho'\in[\rho,r]$ that
  $$\omega_P(\rho')\ls c_6\left(\frac{\rho'}{r}\right)^\alpha\omega_P(r)+c_6\left(\frac{\rho'}{r}\right)$$
  for some $\alpha\in(0,1)$.
  In particular, $\omega_P(\rho)\ls c_7\rho^\alpha$ (where the constant $c_7$ depends on $r$), that is,
\begin{equation}\label{equ-4.9}
	\sup_{(y,t)\in Q_\rho}d_Y\big(u(y,t),\overline{u}_{Q_\rho}\big)=\sqrt{\omega_P(\rho)}\ls c_8\rho^{\alpha/2}.
\end{equation}
This yields
$$d_Y(\overline{u}_{Q_\rho}, \overline{u}_{Q_{\rho/2}})  \ls c_9 \rho^{\alpha/2}.$$
Hence, the sequence $\{\overline{u}_{Q_{2^{-j}r_1}}\}$, $r_1:=\epsilon^mr$, is a Cauchy sequence in $Y$. Hence, we conclude the limit
$$\tilde u(x_0,t_0):=\lim_{j\to+\infty}\overline{u}_{Q_{2^{-j}r_1}(x_0,t_0)}$$
exists. Moreover, $u=\tilde u$ almost all $(x_0,t_0)\in U_T$ and $\tilde u$ is $(\alpha/2)$-H\"older continuous. The proof is finished.
 \end{proof}

Now we use Theorem \ref{thm-4.1} on the weak solution of harmonic map heat flow to conclude the local H\"older continuity.

\begin{theorem}\label{thm-4.5}
	Let $u(x,t)$ be a weak solution of harmonic map heat flow on $\Omega\times(0,+\infty)$ with bounded initial data. Then for any $0<t_*<T<+\infty$, we have $u\in C_{\rm loc}^\alpha(\Omega\times(t_*,T),Y)$ for some $\alpha\in (0,1)$.
\end{theorem}

\begin{proof} By (\ref{equ-3.7}), we know that for any $P\in Y$, there holds
	$$({\Delta} - \partial_t)d^2_Y\big(P,u^t(x)\big)\gs 2e_{u^t}$$
	in the sense of distributions on $\Omega\times (t_*,T)$. Recall from (\ref{equ-2.10}) that $u\in Lip((t_*,T),L^2(\Omega,Y))$. Thus we can apply Theorem \ref{thm-4.1} to conclude that $u(x,t)\in C_{\rm loc}^\alpha(\Omega\times(t_*,T),Y)$ for some $\alpha\in (0,1)$.
	\end{proof}

\section{Lipschitz continuity in time}

We continue to denote by $u(x,t)$ a weak solution of harmonic map heat flow from $\Omega$ to $Y$ with bounded initial data.
 From the previous Theorem \ref{thm-4.5}, we always assume $u(x,t)\in C(\Omega\times(0,+\infty),Y)$.

We can now show that $t\mapsto u(x,\cdot)$ is locally  Lipschitz continuous in $(0,+\infty)$, for each fixed $x\in \Omega$.
\begin{theorem}\label{thm-5.1}
Let $\Omega\subset M$ be a bounded open domain, let $(Y,d_Y)$ be a $CAT(0)$ space, and let $u(x,t)$ be a bounded weak solution of harmonic map heat flow from $\Omega$ to $Y$.  Assume  $0<t_*<T<+\infty$ and a ball $B_R(\bar x)$ with $B_{2R}(\bar x)\subset \Omega$, $R^2<\min\{\frac{t_*}{2},\frac{T-t_*}{2}\}$, $R<1$. Then there exists a constant $c:=c_{n,K,R}>0$ such that
\begin{equation}\label{equ-5.1}
d_Y\big(u(x,t),u(x,t+s)\big)\ls c L\cdot s, \quad \forall t, t+s\in(t_*,T),\quad \ \forall x\in B_R(\bar x),
\end{equation}
where $L$ is given in (\ref{equ-2.10}), and $K$ is the lower bound of Ricci curvature of $M$.
\end{theorem}
 \begin{proof}
 Fix any $s>0$. The semi-group property ensures that $v(x,t):=u(x,t+s)$ is also a weak solution of the harmonic map heat flow. Using Theorem \ref{thm-3.4} to $u(x,t)$ and $v(x,t)$, we conclude that the function
 $w(x,t):=d_Y^2(u^{t+s}(x),u^t(x))$ satisfies
 $$({\Delta}-\partial_t) w(x,t)\gs 0$$
 in the sense of distributions on $\Omega\times (t_*,T-s)$. 	The maximum principle implies, for each cylinder $B_R(\bar x)\times (t_j-R^2,   t_j+R^2)$ with $B_{2R}(\bar x)\subset \Omega$ and $t_j=t_*+ jR^2$, $j=1,2,\cdots, \ell:=[(T-s-t_*)/R^2]$,  that
\begin{equation*}
	\begin{split}
	\|w\|_{L^\infty(B_R(\bar x)\times (t_j-R^2, t_j+R^2))}& \ls C_{n,K,R} \int_{B_{2R}(\bar x)\times(t_j-2R^2,  t_j+2R^2)}w\dm{\rm d}t\\
&\ls 	C_{n,K,R} \int_{t_j-2R^2}^{t_j+2R^2}\int_{\Omega}d^2_Y(u^t(x),u^{t+s}(x))\dm{\rm d}t	\\
&=	C_{n,K,R} \int_{t_j-2R^2}^{t_j+2R^2}D^2(u^t,u^{t+s}){\rm d}t\ls 4R^2 C_{ n,K,R}\cdot (L s)^2.
	\end{split}
\end{equation*}
By combining with
$$ \|w\|_{L^\infty(B_R(\bar x)\times (t_*, T-s))}\ls \sup_{j=1,2,\cdots, \ell}\|w\|_{L^\infty(B_R(\bar x)\times (t_j-R^2, t_j+R^2))},$$ we obtain
\begin{equation*}
	\|w\|_{L^\infty(B_R(\bar x)\times (t_0, T-s))}\ls 4R^2C_{n,K,R}L^2\cdot s^2.
\end{equation*}
Noticing that $u(x,t)$ is continuous in $\Omega\times(0,+\infty)$, the desired assertion (\ref{equ-5.1}) follows, with  $c:=2RC^{1/2}_{n,K,R}$.  \end{proof}

As an easy consequence, we have the following.
\begin{corollary}
	\label{coro-5.2}
	Let $\Omega,Y, t_*,T, B_R(\bar x)$ and $u(x,t)$ be as in Theorem \ref{thm-5.1}. Then for any $ t_0\in (t_*,T)$, it holds
	\begin{equation}
		\lim_{s\to0^+}\frac{1}{s}\int_\Omega p_s(x_0,y) d^2_Y\big(u(y , t_0-s),u(x_0,t_0)\big)\dm(y) = 2e_{u^{t_0}}(x_0)
	\end{equation}
	for almost all $x_0\in B_R(\bar x).$
 \end{corollary}
\begin{proof}
Fix any $t_0\in (t_*,T)$. Since  $u^{t_0}\in W^{1,2}(\Omega,Y)$, by Proposition \ref{prop-2.1}(2), we have that for almost all $x_0\in \Omega$,
\begin{equation}\label{equ-5.3}
	\lim_{s\to0^+}\frac{1}{s}\int_\Omega p_s(x_0,y)  d^2_Y\big(u^{t_0}(y),u^{t_0}(x_0)\big)\dm(y)=2 e_{u^{t_0}}(x_0).
\end{equation}
 Fix such an $x_0\in \Omega$.
  For any $s\in(0,t_0-t_*)$,
by using the elementary inequality
$$|a^2-b^2|=|(a-b)^2+2ab-2b^2|\ls (a-b)^2+2|b|\cdot|a-b|$$
 to $a:=d_Y\big(u(y,t_0-s),u(x_0,t_0)\big) $ and
$b:=d_Y\big(u(y,t_0),u(x_0,t_0)\big)$ and the triangle inequality, we have
\begin{equation*}
\begin{split}
	&\Big|d^2_Y\big(u(y,t_0-s),u(x_0,t_0)\big) - d^2_Y\big(u(y,t_0),u(x_0,t_0)\big)\Big|\\
	\ls\ &	\Big[d_Y\big(u(y,t_0-s),u(y,t_0)\big)  \Big]^2	 +2d_Y\big(u(y,t_0),u(x_0,t_0)\big)\cdot d_Y\big(u(y,t_0-s),u(y,t_0)\big).
		\end{split}	
\end{equation*}
Since the weak solution $u$ is bounded, we may assume $u(\Omega \times (t_*,T))\subset \overline{B_{M_0}(P_0)}$ for some $P_0 \in Y$ and $M_0 > 0$. By Theorem \ref{thm-5.1}, we have
\begin{equation*}
\begin{split}
	&\Big|d^2_Y\big(u(y,t_0-s),u(x_0,t_0)\big) - d^2_Y\big(u(y,t_0),u(x_0,t_0)\big)\Big|\\
	\ls\ &	
	\begin{cases} (cLs)^2+2 cLs \cdot d^2_Y\big(u(y,t_0),u(x_0,t_0)\big),& {\rm if}\  \ y\in B_R(\bar x),\\
	12M^2_0,&  {\rm if}\ \  y\in \Omega\setminus B_R(\bar x).
			\end{cases}
		\end{split}	
\end{equation*}
Integrating over $\Omega$ with respect to $p_s(x_0,\cdot)\dm$ and dividing by $s$, we have
\begin{equation*}
	\begin{split}
		I(s):=\ &\frac{1}{s}\int_\Omega p_s(x_0,y)\Big|d^2_Y\big(u(y,t_0-s),u(x_0,t_0)\big) - d^2_Y\big(u(y,t_0),u(x_0,t_0)\big)\Big|\dm(y)\\
		\ls\   &(cL)^2  s+cL\int_{B_R(\bar x)} p_s(x_0,y)  d_Y\big(u(y,t_0),u(x_0,t_0) \big) \dm(y)\\
		&\ \ + \frac{12M^2_0}{s}\int_{\Omega\setminus B_R(\bar x)} p_s(x_0,y)\dm(y).					\end{split}
\end{equation*}
Because $x_0\in B_R(\bar x)$ and $u(\cdot,t_0)$ is continuous at $x_0$, we have
\begin{equation}
	\label{equ-5.4}
\lim_{s\to0^+}\int_{B_R(\bar x)} p_s(x_0,y)  d_Y\big(u(y,t_0),u(x_0,t_0) \big) \dm(y)=0.
\end{equation}
The below Lemma \ref{lem-heat-ker} (by taking $\ell=0$ therein) implies
\begin{equation}
	\label{equ-5.5}
	\lim_{s\to0^+}\frac{1}{s}\int_{\Omega\setminus B_R(\bar x)}p_s(x_0,y)\dm(y)=0.
	\end{equation}
 Substituting (\ref{equ-5.4}) and (\ref{equ-5.5}) into $I(s)$, we conclude that $\lim_{s\to0^+}I(s)=0.$
 Therefore,
\begin{equation*}
		\lim_{s\to0^+}\frac{1}{s}\int_\Omega p_s(x_0,y) d^2_Y\big(u^{t_0-s}(y),u^{t_0}(x_0)\big)\dm(y)= \lim_{s\to0^+} \frac{1}{s}\int_\Omega p_s(x_0,y) d^2_Y\big(u^{t_0}(y),u^{t_0}(x_0)\big) \dm(y). \end{equation*}
Combining this and (\ref{equ-5.3}),  the desired estimate follows.
\end{proof}

\begin{lemma}
	\label{lem-heat-ker}Let $U$ be an open set and let $\ell\in\mathbb N\cup\{0\}$. Then for any $x_0\in U$, one has
	$$\lim_{s\to0^+}\frac{1}{s}\int_{M\setminus U}d^\ell(x_0,y)p_s(x_0,y)\dm(y)=0.$$
\end{lemma}
\begin{proof} It is well-known for experts (see, for example, \cite[Lemma 2.53]{MS25} for a similar statement). We give a proof here for completeness.
Recall the upper bound of heat kernels  (see, for example, \cite[Theorem 4.6 in Chapter IV]{SY94}), we have for any $y\in M\setminus U$ and any $s\ls1$ that
$$p_s(x_0,y)\ls \frac{c_n}{|B_{\sqrt s}(x_0)|}e^{-\frac{d^2(x_0,y)}{5s}+C_nKs}\ls c_{n,K}\cdot \frac{e^{-\frac{d^2(x_0,y)}{5s}}}{|B_{\sqrt s}(x_0)|}.$$
Multiplying by $d^\ell(x_0,y)$, integrating over $M\setminus U$ and dividing by $s$, we get
$$\frac{1}{s}\int_{M\setminus U}d^\ell(x_0,y)p_s(x_0,y)\dm(y)\ls    \frac{c_{n,K}}{s|B_{\sqrt s}(x_0)|} \int_{M\setminus B_{\epsilon}(x_0)}d^\ell(x_0,y)e^{-\frac{d^2(x_0,y)}{5s}} \dm(y),$$
where $\epsilon=d(x_0,\partial U).$
By using Bishop-Gromov inequality $|B_{\sqrt s}(x_0)|/|B_1(x_0)|\gs c_{n,K} \cdot s^{n/2}$ for all $s<1$, and the growth of area of geodesic spheres (without loss the generality, we  consider only the case $K<0$)
$$|\partial B_r(x_0)|\ls |\partial B_r(0)\subset \mathbb H^n_{K/(n-1)}|=c_{n,K}\sinh^{n-1}\big(\sqrt{(n-1)K}r\big)\ls c_{n,K}\cdot e^{\sqrt{(n-1)K}r},$$
we get, for all $s<1$, that
$$ \frac{1}{s|B_{\sqrt s}(x_0)|} \int_{M\setminus B_{\epsilon}(x_0)}d^\ell(x_0,y)e^{-\frac{d^2(x_0,y)}{5s}} \dm(y)\ls \frac{c_{n,K}}{s^{1+n/2}} \int_{\epsilon}^\infty r^le^{-\frac{r^2}{5s}} \cdot e^{\sqrt{(n-1)K}r}{\rm d}r.$$
By a direct computation, we have
$$\lim_{s\to0^+}\frac{1}{s^{1+n/2}} \int_{\epsilon}^\infty r^le^{-\frac{r^2}{5s}} \cdot e^{\sqrt{(n-1)K}r}{\rm d}r=0.$$
The proof is finished. \end{proof}

Given a map $v\in   C(\Omega,Y)$ and $r>0$, we define a function $x\mapsto {\rm lip}_r v(x)$ on $\Omega$ by
\begin{equation}
\label{equ-5.6}	
{\rm lip}_rv(x):=\sup_{s\in(0,r)}\sup_{y\in B_s(x)}\frac{d_Y\big(v(x),v(y)\big)}{s}\in [0,+\infty]. \end{equation}
It is clear that ${\rm lip}v(x)=\lim_{r\to0^+}{\rm lip}_rv(x)$ by their definitions.

The next result will play an important role in Section 7 for the proof of Lipschitz regularity in space variables.
\begin{proposition}
	Let $\Omega,Y, t_*,T, B_R(\bar x)$ and $u(x,t)$ be as in Theorem \ref{thm-5.1}. Then there exists a constant $C:=C_{n,K,R}>0$ such that for almost all $t\in (t_*,T)$, it holds
		\begin{equation}\label{equ-5.7}
		\int_{B_R(\bar x)}[{\rm lip}_r u^t(x)]^2\dm(x)\ls C\int_\Omega e_{u^t}(x)\dm(x)+cL \cdot R |B_R(\bar x)|,\qquad \forall r\in(0,R/4),
	\end{equation}
	where the constant $cL$ is given in Theorem \ref{thm-5.1}.
	
		In particular, for almost all $(x,t)\in B_R(\bar x)\times(t_*,T)$ we have ${\rm lip}_ru^t(x)<+\infty$, for any $r\in(0,R/4)$.
\end{proposition}
 \begin{proof}

 We first recall a fact that  for any $h\in W^{1,2}(B_{2R}(\bar x))$, it holds
\begin{equation}
	\label{equ-5.8}
\fint_{B_s(x)}|h(x)-h(y)|\dm(y)\ls  \Big(M(|\nabla h|)( x)+M(M(|\nabla h|))( x) \Big)\cdot s
\end{equation}	
 	  for any ball $B_s(x)$ with $x\in B_R(\bar x)$ and $s<R/2$, where   $Mw $ is the Hardy-Littlewood maximal function for a function $w\in L^1_{\rm loc}(\Omega)$:
 	  $$Mw(x)=\sup_{r>0}\frac{1}{|B_r(x)|}\int_{B_r(x)\cap \Omega}|w|\dm.$$
 	  Indeed, according to \cite[Theorem 3.2]{HK00}, there exists a constant $c=c_{n,K,R}\in(0,1)$ such that
 	  for almost all $x,y\in B_R(\bar x)$ with $d(x,y)<cR$, we have
 	  $$|h(x)-h(y)|\ls d(x,y)\cdot \Big(M(|\nabla h|(x)+M(|\nabla h|)(y)\Big).$$
 	Integrating over $B_s(x)$, we have
 	\begin{equation*}
 	\begin{split}
 	\fint_{B_s(x)}|h(x)-h(y)| \dm(y)&\ls s\fint_{B_s(x)}\Big(M(|\nabla h|(x)+M(|\nabla h|)(y)\Big)\dm(y) \\
 	& \ls s \Big(M(|\nabla h|)(x)+	M[M(|\nabla h|)](x)\Big),
\end{split}
\end{equation*}
which is (\ref{equ-5.8}).

 	Fix any $ x\in B_R(\bar x)$ and $t_0 \in (t_*,T)$, we denote $$f(y,t):=d_Y\big(u^{t_0}(x), u^t(y)\big).$$
 	By  using Theorem \ref{thm-5.1} and the triangle inequality, we have that 	$$|f(y,t+s)-f(y,t)| \ls d_Y(u^{t+s}(y),u^t(y))\ls   cL|s|$$
 for all $t, t+s\in (t_*,T)$. This implies
 	$$ |\partial_tf(y,t)|\ls \limsup_{s\to0} \frac{|f(y,t+s)-f(y,t)|}{|s|}\ls cL $$
 	for any $t\in(t_*,T)$.
 	By combining with the fact that
 	$\Delta f\gs  \partial_t f  $
 	on $\Omega\times(0,+\infty)$ in the sense of distributions and noting the continuity of the function $f$, we have
 	$$\Delta f(\cdot, t_0)\gs -cL$$
 	in the sense of distributions on $\Omega$. By the local boundedness of subsolution of $\Delta f\gs -cL$ (see, for example, \cite[Theorem 4.1]{HL11}), we have
\begin{equation}\label{equ-5.9}
	\sup_{y\in B_s(x)}f(y,t_0)\ls C_{n,K,R} \cdot\Big(\fint_{B_{2s}(x)}f(y,t_0)\dm(y)+cL\cdot s^2\Big) \end{equation}
 	for any ball $B_s(x)$ with $x\in B_{R}(\bar x)$ and $s<R/2$.
 	
By using (\ref{equ-5.8}) to $f(y,t_0)$ and noticing $f(x,t_0)=0$ and the fact that $|\nabla f(y,t_0)|\ls \sqrt{e_{u^{t_0}}}(y)$ for almost all $y\in \Omega$, we have
$$\fint_{B_{s}( x)}f(y,t_0)\dm(y) \ls ( M(\sqrt{e_{u^{t_0}}})(x)+M[M(\sqrt{e_{u^{t_0}}})](x))\cdot s$$
for all ball $B_s(x)$ with $x\in B_{R}(\bar x)$ and $s<R/2.$ Combining with (\ref{equ-5.9}), we get
$$\sup_{y\in B_s(x)}\frac{d_Y\big (u^{t_0}(x),u^{t_0}(y)\big)}{s}\ls C_{n,K,R}\Big(M(	\sqrt{e_{u^{t_0}})}(x)+M[M(\sqrt{e_{u^{t_0}}})](x)+cLs\Big),$$
for all $x\in B_{R}(\bar x)$ and $s<R/4.$ This implies
  \begin{equation}\label{equ-5.10}
	{\rm lip}_r u^{t_0}(x)\ls C_{n,K,R}\Big( M (\sqrt{e_{u^{t_0}}})(x)+M[M( \sqrt{e_{u^{t_0}}})](x)\Big)+cL\cdot R/4,
	\end{equation}
for all   $x\in B_R(\bar x)$ and all $r<R/4$.

Finally, since $\sqrt{e_{u^{t_0}}}\in L^2(B_R(\bar x))$, the assertion (\ref{equ-5.7}) comes from the combination of (\ref{equ-5.10}) and the $L^2$-boundedness of the Hardy-Littlewood operator.
 \end{proof}

 \section{Asymptotic mean value inequality for heat equations}

In this section, we give an asymptotic mean value inequality for heat equations. We first consider a global version as follows (c.f. \cite[Lemma 2.1]{CN12}, see also Lemma 4.2 in \cite{Gig23}).

\begin{lemma}\label{lem-6.1} Let $M$ be an $n$-dimensional complete Riemannian manifold with $Ric\gs K$ for some $K\ls0$.
Let $g(x,t)\in W^{1,2}(M\times (0,T))\cap C(M\times (0,T))$ for some $T>0$. Suppose that
  $$({\Delta}-\partial_t)g(x,t)\ls f(x,t)$$
   for some $f\in L^1(M\times (0,T))$ in the sense of distributions, then for any $(x_0, t_0 )\in  M\times (0,T)$, we have
\begin{equation}\label{equ-6.1}
	H_s[g(\cdot, t_0 -s)](x_0)-g(x_0, t_0 )\ls \int_0^sH_\tau [f(\cdot,t_0-\tau)](x_0){\rm d}\tau,\qquad \forall s\in(0,t_0),
\end{equation}
where $H[f]$ is the heat flow given by
$$H[f(\cdot)](x):=\int_Mp_s(x,y)f(y)\dm(y).$$
\end{lemma}

\begin{proof} Fix any $(x_0,t_0)\in M\times (0,T)$.
By
$H_{\tau} [g(\cdot, t_0  - \tau)](x_0)=\int_M p_{ \tau}(x_0,x)g(x, t_0 -\tau)\dm(x)$,
 we have
\begin{equation*}
\begin{split}
	\partial_\tau H_{\tau} [g(\cdot, t_0  -\tau)](x_0) =&\int_M\Big(\Delta p_{\tau}(x_0,x) \cdot g(x, t_0 -\tau)-p_{\tau}(x_0,x)\cdot \partial_t g(x, t_0 -\tau)\Big)\dm(x)\\
	 \ls&  \int_M p_{ \tau}(x_0,x)\cdot f(x, t_0 - \tau )\dm(x)= H_{\tau} [f(\cdot, t_0  -\tau)](x_0).
\end{split}
	\end{equation*}
For any $s\in(0,t_0)$, integrating over $(0,s)$ and using $\lim_{\tau\to 0}H_{\tau}[g(\cdot, t_0 -\tau)](x_0)=g(x_0, t_0 )$ (since $g$ is continuous), we obtain
\begin{equation*}
H_s[g(\cdot, t_0 -s)](x_0)-g(x_0, t_0 )\ls  \int_0^sH_{ \tau}[f(\cdot, t_0 - \tau)](x_0){\rm d}\tau.
\end{equation*}
The proof is finished.
\end{proof}

For our purpose, we need a local version as follows.

\begin{lemma}\label{lem-6.2}
 Let $\Omega\subset M$ be a bounded open domain.
Let $ \mathcal Q=\Omega\times (0,T)$ be a bounded cylinder and $f\in L^1(\mathcal Q)$. Suppose that  $g\in W^{1,2}(\mathcal Q)\cap C(\mathcal Q)$ and  $g(x,t)\ls M_0$  for some constant $M_0>0$. If
   $$({\Delta}-\partial_t)g(x,t)\ls f(x,t)\ \ {\rm on}\ \ \mathcal Q$$  in the sense of distributions, then  for any $(x_0,t_0)\in\mathcal Q$
\begin{equation}\label{equ-6.2}
	\limsup_{s\to0^+} \frac{\int_\Omega p_s(x_0,y) g(y, t_0-s)\dm(y)-g(x_0, t_0)}{s}\ls \limsup_{s\to0^+}   \frac{1}{s}\int_0^s\int_\Omega p_\tau(x_0,y) f(y,t_0-\tau)\dm(y){\rm d}\tau.
	\end{equation}

\end{lemma}

\begin{proof}
Let $\Omega'\subset\subset\Omega$ and denote $R:=d(\Omega',\partial\Omega)/4$. For any $x_0\in \Omega'$, we take a smooth cut-off function $\phi:M\to [0,1]$ such that   $\phi|_{B_R(x_0)}\equiv 1$, ${\rm supp}\phi\subset B_{2R}(x_0)$.	We have $g(x,t)\phi(x)\in W^{1,2}(M\times (0,T))\cap C(M\times (0,T))$, where it is understood that $g\phi=0$ outside of $\mathcal Q$.
We have
	$${ \Delta}(g\phi)-\partial_t(g\phi)\ls \tilde f:= f\phi+2\ip{\nabla g}{\nabla \phi}+g\Delta\phi\in L^1( M\times (0,T))$$
in the sense of distributions.  By apply the above Lemma \ref{lem-6.1} to $g\phi$, we get for any $(x_0,t_0)\in \mathcal Q$ and any $s\in(0,t_0)$ that
	\begin{equation}\label{equ-6.3}
		\begin{split}
	& H_s[(g\phi)(\cdot,t_0-s)](x_0)-(g\phi)(x_0,t_0) \ls  \int_0^sH_\tau [\tilde f(\cdot,t_0-\tau)](x_0){\rm d}\tau	\\
		\ls \ &\int_0^s \int_\Omega p_\tau(x_0,y)  f(y,t_0-\tau)\dm(y){\rm d}\tau+\int_0^s\int_\Omega p_\tau(x_0,y) |f(1-\phi)|(y,t_0-\tau)\dm(y){\rm d}\tau\\
		&\ \ +\int_0^s \int_\Omega p_\tau(x_0,y)  \Big|2\ip{\nabla g}{\nabla \phi}+g\Delta\phi\Big| (y,t_0-\tau)\dm(y){\rm d}\tau.
					\end{split}
	\end{equation}
Since $(1-\phi)=0$ on $B_R(x_0)$, and by using the upper bound of the heat kernel (see, for example, \cite[Theorem 4.6 in Chapter IV]{SY94}, we have for any $\tau<1$ that
\begin{equation*}
	\begin{split}
		\int_\Omega p_\tau(x_0,y)|f(1-\phi)|(y,t_0-\tau)\dm(y) &=\int_{\Omega\setminus B_R(x_0)}p_\tau(x_0,y) |f|(y,t_0-\tau)\dm(y)\\
		&\ls \frac{C_{n,K}}{|B_{\sqrt \tau}(x_0)|}e^{-\frac{R^2}{5\tau}}\|f(\cdot, t_0-\tau)\|_{L^1(\Omega)}\\
		&\ls \frac{C_{n,K,R}}{|B_1(x_0)|}e^{-\frac{R^2}{10\tau}}\|f(\cdot, t_0-\tau)\|_{L^1(\Omega)}
			\end{split}
\end{equation*}
for all $\tau$ sufficiently small,
where $K$ is a lower bound of Ricci curvature on $M$, and, in the last inequality, we have used that $|B_{\sqrt\tau}(x_0)|/|B_1(x_0)|\gs c_{n,K}\tau^{n/2}$ for all $\tau<1$ and that $\tau^{n/2}\gs c_{n,R} e^{-\frac{R^2}{10\tau}}$ for all sufficiently small $\tau$.
Therefore, we obtain
\begin{equation}\label{equ-6.4}
	\begin{split}
	&\limsup_{s\to0^+} \frac{1}{s} \int_0^s\int_\Omega p_\tau(x_0,y)|f(1-\phi)|(y,t_0-\tau)\dm(y){\rm d}\tau  \\
	\ls & \limsup_{s\to0^+}C_{1} s^{-1} \int_0^s
 e^{-\frac{R^2}{10\tau}}\|f(\cdot, t_0-\tau)\|_{L^1(\Omega)}{\rm d}\tau\\
 \ls& \limsup_{s\to0^+}  C_1 s^{-1}e^{-\frac{R^2}{10s}}\int_0^s
  \|f(\cdot, t_0-\tau)\|_{L^1(\Omega)}{\rm d}\tau	\\
  \ls &\limsup_{s\to0^+} C_1 s^{-1} e^{-\frac{R^2}{10s}}\|f\|_{L^1(\mathcal Q)}=0,
  	\end{split}
\end{equation}
where we have used $e^{-\frac{R^2}{10\tau}}\ls e^{-\frac{R^2}{10s}}$ for all $\tau\ls s$. Similarly, since $ |\nabla \phi|=|\Delta\phi|=0$ on $B_R(x_0)$, we have
\begin{equation}
	\label{equ-6.5}
		\limsup_{s\to0^+} \frac{1}{s} \int_0^s \int_\Omega p_\tau(x_0,y)  \Big|2\ip{\nabla g}{\nabla \phi}+g\Delta\phi\Big| (y,t_0-\tau)\mu(y){\rm d}\tau =0.
\end{equation}

 Noticing that $(1-\phi)=0$ on $B_R(x_0)$ again and $g\ls M_0$, we have
\begin{equation} \label{equ-6.6}
\begin{split}
	&\frac{1}{s} \int_\Omega p_s(x_0,y)g(y,t_0-s)d\mu(y)-H_s[(g\phi)(\cdot,t_0-s)](x_0) \\
	= \ &\frac{1}{s} \int_\Omega    p_s(x_0,y)g(y,t_0-s)(1-\phi)(y,t_0-s)d\mu(y) \\
	\ls \ &\frac{1}{s} \int_{\Omega\setminus B_R(x_0)}    p_s(x_0,y)g^+(y,t_0-s)d\mu(y) \\
	\ls\ & M_0 s^{-1} \int_{\Omega\setminus B_R(x_0)}p_s(x_0,y)\dm(y)\to 0,	\qquad {\rm as}\ s\to0^+.
\end{split}
	\end{equation}
Finally, by combining the above four inequalities (\ref{equ-6.3})--(\ref{equ-6.6}), we conclude (\ref{equ-6.2}).
\end{proof}

The following proposition is an analog of Lebesgue's differential theorem.
\begin{proposition}\label{prop-6.3}
	Let $ \mathcal Q=\Omega\times (0,T)$ be a bounded cylinder and $f(x,t)\in L^1(\mathcal Q)$. Then there exists $\mathcal N\subset \mathcal Q$ with $(\mu \times \mathscr L^1)(\mathcal N)=0$  such that
	\begin{equation}\label{equ-6.7}
		\limsup_{s\to0^+}   \frac{1}{s}\int_0^sH_\tau [f(\cdot,t-\tau)](x){\rm d}\tau= f(x,t),	\quad \forall(x,t)\in\mathcal Q\setminus\mathcal N.
		\end{equation}
	\end{proposition}
	(Here $f$ and $g$ are understood as their zero extensions on  $M\times(0,T)$.)

	\begin{remark}
		\label{rem-6.4}
In the case when $f(x,t)=f(x)$, it is well-known that $H_sf(x)\to f(x)$, as $s\to0^+$, at any Lebesgue point of $f$.
		\end{remark}
	
	\begin{proof}[Proof of Proposition \ref{prop-6.3}]
		Note that, for any $\tilde f\in L^1(\mathcal Q)$ and any $(x,t)\in\mathcal Q$,
\begin{equation*}
	\begin{split}
		 &\left|\frac{1}{s}\int_0^sH_\tau [f(\cdot,t-\tau)](x){\rm d}\tau- f(x,t)\right|\\
		\ls \ &  \left|\frac{1}{s}\int_0^sH_\tau [(f-\tilde f)(\cdot,t-\tau)](x){\rm d}\tau  \right|  +  \left|\frac{1}{s}\int_0^sH_\tau [\tilde f(\cdot,t-\tau)](x){\rm d}\tau- f(x,t)\right|.
		 	\end{split}
\end{equation*}	
If $\tilde f$ is continuous, the last term converges to $|\tilde f(x,t)-f(x,t)|$	 as $s\to 0^+$. Letting $L(x,t)$ denote the upper limit of the term on the left-hand side, we obtain
\begin{equation}\label{equ-6.8}
	L(x,t)\ls \limsup_{s\to0}\left|\frac{1}{s}\int_0^sH_\tau [(f-\tilde f)(\cdot,t-\tau)](x){\rm d}\tau  \right|  +  |\tilde f -f|(x,t). \end{equation}
It suffices to prove that for each $\varepsilon>0$ and $\alpha>0$, we have
$$(\mu \times \mathscr L^1)(E_{\varepsilon,\alpha})=0,\quad {\rm where}\quad E_{\varepsilon,\alpha}:=\left\{(x,t)\in\Omega\times(\varepsilon,T)\big|\ L(x,t)>\alpha\right\}.$$
The equation (\ref{equ-6.8}) implies
\begin{equation*}
\begin{split}
	E_{\varepsilon,\alpha}\subset &\left\{(x,t)\in\Omega\times(\varepsilon,T)\big|\ \limsup_{s\to0}\left|\frac{1}{s}\int_0^sH_\tau [(f-\tilde f)(\cdot,t-\tau)](x){\rm d}\tau  \right|  >\alpha/2\right\}\\
	&\ \bigcup \left\{(x,t)\in\Omega\times(\varepsilon,T)\big|\ |f-\tilde f |>\alpha/2\right\},
\end{split}	
\end{equation*}
and hence, by Markov inequality, we get
\begin{equation*}
\begin{split}
	(\mu \times \mathscr L^1)(E_{\varepsilon,\alpha})\ls & \frac{2}{\alpha} \int_\varepsilon^T\int_\Omega \limsup_{s\to0}\left|\frac{1}{s}\int_0^sH_\tau [(f-\tilde f )(\cdot,t-\tau)](x){\rm d}\tau	\right|\dm(x){\rm d}t\\
	& + \frac{2}{\alpha} \int_\varepsilon^T\int_\Omega  |f-\tilde f |(x,t) \dm(x){\rm d}t.\\
\end{split}	
\end{equation*}
By Fatou's lemma, the first term on the right-hand side can be bounded by
\begin{equation*}
\begin{split}
 &\int_\varepsilon^T\int_\Omega \limsup_{s\to0}\left|\frac{1}{s}\int_0^sH_\tau [(f-\tilde f )(\cdot,t-\tau)](x){\rm d}\tau	\right|\dm(x){\rm d}t \\
 \ls & \limsup_{s\to0}\frac{1}{s}\int_0^s\left(\int_\varepsilon^T\int_M\left|H_\tau [(f-\tilde f )(\cdot,t-\tau)](x)	\right|\dm(x){\rm d}t\right){\rm d}\tau\\
	 \ls & \limsup_{s\to0}\frac{1}{s}\int_0^s\left(\int_\varepsilon^T\int_M\left|f-\tilde f 	\right|(x,t-\tau)\dm(x){\rm d}t\right){\rm d}\tau\\
	 \ls & \limsup_{s\to0}\frac{1}{s}\int_0^s\left(\int_0^T\int_M\left|f-\tilde f 	\right|(x,t)\dm(x){\rm d}t\right){\rm d}\tau\\
	 \ls & \|f-\tilde f \|_{L^1(M\times (0,T))},
\end{split}	
\end{equation*}
where we have used the $L^1(M)$-contraction of the semigroup $H_\tau$ in the second inequality. Therefore, we obtain
\begin{equation*}
	(\mu \times \mathscr L^1)(E_{\varepsilon,\alpha})\ls \frac{2}{\alpha}\|f-\tilde f \|_{L^1(M\times (0,T))}	 + \frac{2}{\alpha} \|f-\tilde f \|_{L^1(\mathcal Q)}.
\end{equation*}
Since $\|f-\tilde f \|_{L^1(M\times (0,T))}$ can be made arbitrarily small whth appropriate choice of $\tilde f$, it follows that $(\mu \times \mathscr L^1)(E_{\varepsilon,\alpha})=0$ for each $\varepsilon>0$ and each $\alpha>0.$
\end{proof}

Finally, we apply it to the weak solutions of harmonic map heat flow.

\begin{corollary}\label{coro-6.5}
  Let $\Omega,Y, t_*,T, B_R(\bar x)$ and $u(x,t)$ be as in Theorem \ref{thm-5.1} .     Then for each $T\in(0,\infty)$ there exists $\mathcal N_T\subset B_R(\bar x)\times(t_*,T)$ with
 $(\mu\times\mathscr L^1)(\mathcal N_T)=0$ such that for any $P\in Y$  the following asymptotic mean value inequality holds
   \begin{equation} \label{equ-6.9}
	\limsup_{s\to0^+}\frac{1}{s}\int_\Omega p_s(x_0,y)  w_{P,s}(x_0,y,t_0) \dm(y)\ls 0,
	\end{equation}
for all $(x_0,t_0)\in(B_R(\bar x)\times(t_*,T))\setminus \mathcal N_T,$ where
\begin{equation}
\label{equ-6.10}	
w_{P,s}(x,y,t):=-d^2_Y\big(P,u(y, t-s)\big)+d^2_Y\big(P,u(x,t)\big)+d^2_Y\big(u(y,t-s),u(x,t)\big).\end{equation}
(The set $\mathcal N_T$ is independent  of $P$.)
	\end{corollary}

\begin{proof}
By applying Proposition \ref{prop-6.3} to $f(x,t):=-2e_{u^t}(x)\in L^1(\Omega\times(0,T))$ there exists $\mathcal N_1\subset \Omega\times(0,T)$ with zere $(\mu \times \mathscr L^1)$-measure such that
\begin{equation}
	\label{equ-6.11}
\limsup_{s\to0^+}   \frac{1}{s}\int_0^sH_\tau [-2e_{u^{t_0-\tau}}(\cdot)](x_0){\rm d}\tau= -2e_{u^{t_0}}(x_0),	\quad \forall(x_0,t_0)\in (\Omega\times(0,T))\setminus\mathcal N_1.	\end{equation}
	For any $P\in Y$ and $(x_0,t_0)\in (\Omega\times(0,T))\setminus \mathcal N_1$, we have
	$$(\Delta-\partial_t)\Big(-d^2_Y\big(P, u(y,t)\big)+d^2_Y\big(P,u(x_0,t_0)\big)\Big)\ls -2e_{u^{t}}(y)$$
	in the sense of distributions. Therefore, by using Lemma \ref{lem-6.2}  and combining with (\ref{equ-6.11}), we get
	 \begin{equation}\label{equ-6.12}
	\limsup_{t\to0}\frac{1}{s}H_s \left[-d^2_Y\big(P,u(\cdot, t_0-s)\big)+d^2_Y\big(P,u(x_0,t_0)\big) \right](x_0)\ls -2e_{u^{t_0}}(x_0).
		\end{equation}
		Here $-d^2_Y\big(P,u(\cdot, t_0-s)\big)+d^2_Y\big(P,u(x_0,t_0)\big)$ is understood as a function on $M \times (0,T)$ with the zero extension.

Finally, by combining Corollary \ref{coro-5.2} and (\ref{equ-6.12}), the desired assertion (\ref{equ-6.9}) holds, and then the proof is finished.
\end{proof}

\section{Lipschitz continuity in space-time}

Let $\Omega\subset M$ be a bounded open domain in an $n$-dimensional complete Riemannian manifold with $Ric_M\gs K$ for some $K\in \mathbb R$, and let $ Y $ be a $CAT(0)$ space. The main result of this section is Theorem \ref{thm-7.8}, the locally Lipschitz regularity in spatial variables for the weak solution of the harmonic map heat flow.

 \subsection{Nonlinear Hamilton-Jacobi flows}

  Since  $u_0\in L^\infty(\Omega,Y)$, we assume that  its image is contained in a ball $ \overline{B_{M_0}(P_0)}$ for some $M_0>0$ and $P_0\in Y$. According to Lemma \ref{lem-3.1} , we know that  $u^t(\Omega)\subset  \overline{B_{M_0}(P_0)}$ for all $t>0$.

Let $T>0$ and let  $B_R(\bar x )$ be a ball with $B_{2R}(\bar x )\subset\Omega$.
We first introduce a family of functions on $B_R(\bar x )\times (0,T)$ as follows: for any $\varepsilon>0$,
 \begin{equation} \label{equ-7.1}
 	f_\varepsilon(x,t):=\inf_{y\in \Omega}\left\{\frac{e^{-2Kt}\cdot d^2(x,y)}{2\varepsilon}-F(x,y,t)\right\},\quad \forall (x,t)\in B_R(\bar x )\times(0,T),
 \end{equation}
where
\begin{equation*}
	F(x,y,t):=	d_Y\big(u(x,t),u(y,t)\big).
	\end{equation*}
We put
\begin{equation}
	\label{equ-varepsilon}
 \varepsilon_0:=\frac{e^{-2|K|T}R^2}{8M_0},
 \end{equation}
and collect some basic properties of $f_\varepsilon$ in the following lemma.

\begin{lemma}\label{lem-7.1}
	Under the above notations, we have for each $\varepsilon\in(0,\varepsilon_0)$ that
	\begin{enumerate}
	\item  $-2M_0\ls f_\varepsilon(x,t)\ls0$,  for any  $(x,t)\in B_R(\bar x )\times(0,T)$;
	\item there holds
\begin{equation}
\label{equ-7.2}	
f_{\varepsilon}(x,t)=\min_{y\in B_{C_1\sqrt{\varepsilon}}(x)}\left\{\frac{e^{-2Kt}\cdot d^2(x,y)}{2\varepsilon}-F(x,y,t)\right\}
\end{equation}
for any $(x,t)\in B_{R}(\bar x )\times(0,T)$, where $C_1:=\sqrt{6M_0\cdot e^{2|K|T}}$;
\item  $f_\varepsilon\in C(B_{R}(\bar x )\times(t_*,T))\cap W^{1,2}(B_{R}(\bar x )\times(t_*,T))$ for any $t_*>0$.	
\end{enumerate}
\end{lemma}
\begin{proof}
	For (1),  we first see that $f_\varepsilon(x,t)\ls 0$, by taking $y=x$ in the definition. The lower bound is given by the fact $f_\varepsilon(x,t)\gs -F(x,y,t)$ and  $F(x,y,t)\ls 2M_0$, since $u^t(\Omega)\subset \overline{B_{M_0}(P_0)}$.
	
	For (2),  if $d(x,y)\gs C_1\sqrt\varepsilon$ then we have
	$$\frac{e^{-2Kt}\cdot d^2(x,y)}{2\varepsilon}-F(x,y,t)\gs \frac{e^{-2Kt}\cdot 6M_0\cdot  e^{2|K|T}\varepsilon}{2\varepsilon}-2M_0\gs M_0>0.$$
	Since $f_\varepsilon(x,t)\ls0$, it follows (\ref{equ-7.2}).

	For (3), fixed any $t_*>0$, given any $(x,t),(x',t')\in B_{R}(\bar x )\times(t_*,T)$, we choose one $y'\in \Omega $ such that
	$$f_\varepsilon(x',t')=\frac{e^{-2Kt'}d^2(x',y')}{2\varepsilon}-F(x',y',t').$$
	Therefore, by (\ref{equ-7.1}), we get
	\begin{equation*}
		\begin{split}
			f_\varepsilon(x,t)-f_\varepsilon(x',t')&\ls \frac{e^{-2Kt}d^2(x,y')}{2\varepsilon}-F(x,y',t)-\left(\frac{e^{-2Kt'}d^2(x',y')}{2\varepsilon}-F(x',y',t')	\right)\\
			&= \frac{\big(e^{-2Kt} -e^{-2Kt'}\big)d^2(x,y')+e^{-2Kt'}\big(d^2(x,y')-d^2(x',y')\big)}{2\varepsilon}\\
			&\quad 	+d_Y\big(u(x,t),u(y',t)\big)- d_Y\big(u(x',t'),u(y',t')\big)\\
		 &	\ls \frac{{\rm diam}(\Omega)\cdot e^{2|K|T}}{\varepsilon}\big(|t-t'|+d(x,x')\big) +d_Y\big(u(x,t),u(x',t')\big)+ cL |t-t'|,
		 		\end{split}
	\end{equation*}
	where, for the last inequality, we have used the fact that $|d_Y(u(y',t)u(y',t')|\ls cL|t-t'|$ for any $t,t'>t_0$ (since Theorem \ref{thm-5.1}). By the symmetry of $(x,t)$ and $(x',t')$, we obtain (noticing that $K\ls0$)
		\begin{equation*}
					|f_\varepsilon(x,t)-f_\varepsilon(x',t')| \ls C_{\varepsilon,L}\big(|t-t'|+d(x,x')\big)  +d_Y\big(u(x,t),u(x',t')\big)
	\end{equation*} 	
	for any $t,t'>t_0$, where
	 $$C_{\varepsilon,L}:=\frac{{\rm diam}(\Omega)\cdot e^{2|K|T}}{\varepsilon}+cL.$$
	This implies the following: \\
	(i) $f_\varepsilon$ is continuous at $(x,t)\ $ (since $u$ is continuous at $(x,t)$), and\\
	(ii) $e^{f_\varepsilon}_r(x,t)\ls  4C^2_{\varepsilon,L}+2e^u_r(x,t)$ for any $r>0$ sufficiently small. Therefore, we have
	$$e^{f_\varepsilon}(x,t)\ls  4C^2_{\varepsilon,L}+2e^u(x,t).$$
	 This yields $f_\varepsilon\in W^{1,2}\big(B_R(\bar x )\times(t_*,T)\big)$.
	 		\end{proof}

It is convenient for us to consider another notion of weak solutions of the parabolic equations, the viscosity supersolutions (or subsolutions).
\begin{defn}
		\label{def-7.2}
		Let $\mathcal M$ be a complete Riemannian manifold and $U_T:=U\times(0,T)\subset \mathcal M \times \mathbb R$ be an open set, and let $f\in C(U_T)$. A function $g\in C(U_T)$ is a \emph{viscosity supersolution (resp. subsolution)} of
		$$(\Delta-\partial_t)u= f$$
	 in $U_T$ if,   for any $\phi\in C^2(U_T)$ and any $(\hat x,\hat t)\in U_T$ such that $u-\phi$ attains a local minimum (resp. maximum) at $(\hat x,\hat t)$, one has
		$$(\Delta-\partial_t)\phi(\hat x,\hat t)\ls f(\hat x,\hat t)\qquad (resp. \ \gs).$$
		\end{defn}

\begin{lemma}
	\label{lem-7.3} For any fixed $t_*>0$, the function $F(x,y,t)$ is a viscosity subsolution of the   equation
	$$(\Delta^{(2)}-\partial_t)F(x,y,t)=-4cL,$$
	on $B_{R}(\bar x )\times B_{R}(\bar x ) \times(t_*,T)$, where $\Delta^{(2)}$ is the Laplace-Beltrami operator on $M\times M$, and the constant $cL$ is given in Theorem \ref{thm-5.1}.
	\end{lemma}
	\begin{proof}
		For any $P\in Y$, from (\ref{equ-3.8}), the function $f_P(x,t):= d_Y\big(u(x,t), P \big)$ satisfies
		$$\Delta f_P\gs \partial_t f_P$$
		on $\Omega\times(0,+\infty)$ in the sense of distributions. On the other hand,	 by triangle inequality and Theorem \ref{thm-5.1}, we have
		 $$|f_P(x,t)-f_P(x,t')|\ls d_Y\big(u(x,t),u(x,t')\big)\ls c L|t-t'|$$
		for any $t,t'\gs t_*$ and $x\in B_R(\bar x)$. Hence, we have  $|\partial_t f_P|(x,t)\ls  cL,$ and then
		 $$\Delta f_P\gs \partial_t f_P	\gs -cL$$
		 on $B_R(\bar x)\times (t_*,T)$ in the sense of distributions.

		  It is well-known that every continuous distributional subsolution is also a viscosity subsolution (see, for example, \cite{Ish95}). Hence, we see that $\Delta f_P\gs-cL$ in the sense of viscosity.
		
		  Now we will show that for any $\hat t\in(t_*,T)$,
		  \begin{equation}\label{equ-7.4}
		  	\Delta^{(2)}F(x,y,\hat t)\gs -2cL
		  \end{equation}
		  on $B_R(\bar x )\times B_{R}(\bar x )$ in the sense of viscosity.  Given any $\phi(x,y)\in C^2( B_R(\bar x )\times B_{R}(\bar x ))$ and any point $(\hat x,\hat y)\in B_R(\bar x )\times B_{R}(\bar x )$ such that the function $F(x,y,\hat t)-\phi(x,y)$ attains a local maximum at $(\hat x,\hat y)$. Then the function $\phi(\cdot ,\hat y)$ is in $C^2(B_R(\bar x ))$ and the function $F(\cdot,\hat y,\hat t)-\phi(\cdot,\hat y)$ attains a local maximum at $\hat x$. From the above fact that $\Delta d_Y\big(u(\cdot, \hat t),u(\hat y,\hat t)\big)\gs -cL$ in the sense of viscosity, we conclude that
		  $$\Delta_x\phi(\hat x, \hat y)\gs -cL.$$
		  Similarly, we have $\Delta_y\phi(\hat x, \hat y)\gs -cL$ too. Since $\phi\in C^2(B_R(\bar x )\times B_R(\bar x ))$, we obtain		
		    $$ \Delta^{(2)}\phi(\hat x, \hat y)=\Delta_x\phi(\hat x, \hat y)+ \Delta_y\phi(\hat x, \hat y)\gs -2cL.$$	
		Therefore, we conclude that $\Delta^{(2)}F(x,y,\hat t)\gs -2cL$
		  on $B_R(\bar x )\times B_R(\bar x )$ in the sense of viscosity.
		
		  Finally, by Theorem \ref{thm-5.1} and the triangle inequality, we have
		  \begin{equation*}
		  	\begin{split}
		  		|F(x,y,t)-F(x,y,t')|\ls d_Y\big(u(x,t),u(x,t')\big)+d_Y\big(u(y,t),u(y,t')\big)\ls 2cL|t-t'|.
		  			  	\end{split}
		  \end{equation*}		
	This implies $|\partial_t F(x,y,t)|\ls 2cL$.	 By combining with (\ref{equ-7.4}), it follows that
	$$(\Delta^{(2)}-\partial_t)F(x,y,t)\gs -4cL$$
	in the sense of viscosity.
	   \end{proof}

\begin{remark}\label{rem-7.4}
Compared to our previous work \cite{ZZ18} for showing the Lipschitz continuity of harmonic maps,    it seems more natural to replace (\ref{equ-7.1}) with
$$
f_\varepsilon(x,t):=\inf_{(y,s)\in \Omega\times(0,T)}\left\{\frac{e^{-2Kt}d^2(x,y)+|t-s|^2}{2\varepsilon}-d_Y\big(u(y,s),u(x,t)\big)\right\}
$$
for all $(x,t)\in \Omega'\times(0,T)$. However, the function $d_Y\big(u(x,t),u(y,s)\big)$ appearing in the right-hand side depends on two time variables, which makes it difficult to satisfy a parabolic-type partial differential equation, unlike to the function $(d_Y\big(u(x,t),u(y,t)\big)$ in   Lemma \ref{lem-7.3}.
\end{remark}

 We also need the following parabolic perturbation lemma.
\begin{lemma}
	[pertubation]\label{lem-7.5} Let $(U\times V)_T:=U\times V\times(0,T)$ be a cylinder, where $U,V\subset M$ are  bounded open domains with smooth boundaries, and let $h(x,y,t)\in C((U \times V)_T)$ be a viscosity supersolution of
	$$(\Delta^{(2)}-\partial_t)h\ls C\quad {\rm on}\quad(U\times V)_T,$$
	 for some $C>0$, where $\Delta^{(2)}$ is the Laplace-Beltrami operator on $M\times M$.  Assume that $h$ attains a local minimum at $(\hat x,\hat y,\hat t)\in (U\times V)_T$. Assume that a Borel set $E\subset (U\times V)_T$ has full $(\mu\times\mu\times \mathscr L^1)$-measure.
	
	Then there exists a constant $\delta_0>0$ (depending on the Riemannian $g$ and the injective radius at  $\hat x$ and $\hat y$)  such that the following statement holds: for any $\delta\in (0,\delta_0)$, there exist three functions $\eta_0(t) \in C^\infty (\hat t-\delta_0,\hat t+\delta_0),  \eta_1(x)\in C^\infty(B_{\delta_0}(\hat x))$ and $\eta_2(y)\in C^\infty(B_{\delta_0}(\hat y))$ and a point  $(x_0,y_0,t_0)\in E\cap \big(B_{\delta_0}(\hat x)\times B_{\delta_0}(\hat y)\times(\hat t-\delta_0,\hat t+\delta_0)\big)$ such that $h(x,y,t)+\eta_0(t) + \eta_1(x)+\eta_2(y)$ attains its minimum  in $B_{\delta_0}(\hat x)\times B_{\delta_0}(\hat y)\times(\hat t-\delta_0,t_0]$  at $(x_0,y_0,t_0)$ and that
	$$ |\partial_t \eta_0|(t) + |\Delta\eta_1|(x)+|\Delta\eta_2|(y)\ls c\cdot\delta,$$ for all $(x,y,t)\in B_{\delta_0}(\hat x)\times B_{\delta_0}(\hat y) \times(\hat t-\delta_0,\hat t+\delta_0),$
	where the constant $c>0$ depends only on $n$.
	
\end{lemma}
\begin{proof}
	This comes essentially from the Alexandrov-Bakelman-Pucci estimate for the viscosity solutions to the parabolic equations (see \cite{ACP11}). We will give the details in Appendix A.
\end{proof}

We can now give the key lemma of this section.

\begin{lemma}\label{lem-7.6}
Under the above notations, for each $\varepsilon\in(0,\varepsilon_0)$, the function $f_\varepsilon$ is a viscosity supersolution of the heat equation  in $B_{R}(\bar x )\times(t_*,T)$.
\end{lemma}

\begin{proof}
Fix any $\varepsilon\in(0,\varepsilon_0)$. We shall prove the assertion by a contradiction argument. Suppose that it fails. Then there exists a function $v\in C^2(B_R(\bar x )\times (0,T))$ and a point $ (\hat x,\hat t)\in B_R(\bar x )\times (0,T)$ such that the function $f_\varepsilon-v$ attains a local minimum at $(\hat x,\hat t)$, however
\begin{equation*}
	\theta_0:=(\Delta-\partial_t)v(\hat x,\hat t)>0.
\end{equation*}
  Let $\hat y$ be a point where the minimum in (\ref{equ-7.2}) is achieved. Then  the function
$$H(x,y,t):=\frac{e^{-2Kt}  d^2(x,y)}{2\varepsilon}-F(x,y,t)-v(x,t)$$
attains a local minimum at $(\hat x,\hat y,\hat t)$.

(i) We now want to perturb the minimum point $(\hat x,\hat y, \hat t)$ to a point where the asymptotic mean value inequality in Corollary \ref{coro-6.5} is available.

From Corollary \ref{coro-6.5}, there exists a $(\mu\times\mu\times\mathscr L^1)$-zero measurable set $\mathscr N$ such that for any $P\in Y$,
\begin{equation} \label{equ-7.5}
	\limsup_{s\to0^+}\frac{1}{s}\Big(\int_\Omega p_s(x_0,x)w_{P,s}(x_0,x,t_0) \dm(x)+  \int_\Omega p_s(y_0,y)w_{P,s}(y_0,y,t_0) \dm(y)\Big) \ls 0,
	\end{equation}
for all $(x_0,y_0,t_0)\in \big(B_{R}(\bar x )\times B_R(\bar y)\times(t_*,T)\big)\setminus\mathscr N$.

 By applying  Lemma \ref{lem-7.5}, there exists $\delta>0$ (arbitrarily small), three  functions $\eta_0(t) \in C^\infty (\hat t-\delta,\hat t+\delta), \eta_1(x) \in C^\infty(B_\delta(\hat x))$ and $\eta_2(y) \in C^\infty(B_\delta(\hat y))$, and some point $(x_0,y_0,t_0)\in \big(B_{\delta}(\hat x)\times B_\delta(\hat y)\times(\hat t-\delta,\hat t+\delta)\big)\setminus\mathscr N$,
 so that the function
$$H_1(x,y,t):=H(x,y,t)+\eta_0(t) + \eta_1(x)+\eta_2(y)$$
attains a  minimum at $(x_0,y_0,t_0)$ in $ B_{\delta}(\hat x)\times B_\delta(\hat y)\times(\hat t-\delta,t_0]$, and
\begin{equation}\label{thm-7.6+}
	 |\partial_t \eta_0|(t) + |\Delta \eta_1|(x)+|\Delta \eta_2|(y)\ls \frac{\theta_0}{8},
\end{equation}
for all $(x,y,t)\in \big(B_{\delta}(\hat x)\times B_\delta(\hat y)\times(\hat t-\delta,\hat t+\delta)\big).$ We can also assume that   $\delta>0$ is so small that
\begin{equation}\label{thm-7.7+}
 	(\Delta-\partial_t)v(x,t)\gs \frac{\theta_0}{2},\qquad \forall (x,t)\in B_\delta(\hat x)\times(\hat t-\delta,\hat t+\delta),
 \end{equation}
 since $v\in C^2(B_R(\bar x )\times (0,T)).$

(ii) We shall get a contradiction by the minimum of $H_1$ at $(x_0,y_0,t_0)$.

For any $s>0$, we denote by  $\mu_{x_0}^s:=p_s(\cdot,x_0)\cdot \mu$ and $\mu_{y_0}^s:=p_s(\cdot,y_0)\cdot\mu$ the heat kernel measures at time $s>0$ with center $x_0$ and $y_0$ respectively. Moreover, we denote by  $\Pi^s:=\Pi^s_{(x_0,y_0)}$ be an optimal coupling of $\mu_{x_0}^s$ and $\mu_{y_0}^s$ with respected to the $L^2$-Wasserstein distance, that is a probability measure on $M\times M$ whose first and second marginals are $\mu_{x_0}^s$ and $\mu_{y_0}^s$ repectively and
$$\int_{M\times M}d^2(x,y){\rm d}\Pi^s(x,y)\ls \int_{M\times M}d^2(x,y){\rm d}\gamma(x,y)$$
for any probability measre $\gamma$ on $M\times M$ with the same marginals.

Since $(x_0,y_0,t_0)$ is a minimum of $H_1(x,y,t)$ in $B_\delta(\hat x)\times B_\delta(\hat y)\times(\hat t-\delta,t_0]$, we have
\begin{equation}\label{equ-7.8}
\begin{split}
	0\ls &\liminf_{s\to0^+}\frac{1}{s}\int_{B_\delta(\hat x)\times B_\delta(\hat y)}\Big( H_1(x,y,t_0-s)-H_1(x_0,y_0,t_0)\Big){\rm d}\Pi^s(x,y)\\
	& := \liminf_{s\to0^+}\frac{1}{s}\Big(I_1(s)+I_2(s)+I_3(s)+I_4(s)\Big),
	\end{split}
	\end{equation}
where
\begin{equation}\label{equ-7.9}
\begin{split}
	 I_1(s):&= \frac{1}{2\varepsilon}\int_{B_\delta(\hat x)\times B_\delta(\hat y)}\Big(e^{-2K(t_0-s)}d^2(x,y)-e^{-2Kt_0}d^2(x_0,y_0)\Big){\rm d}\Pi^s(x,y), \\
	 I_2(s):&=-\int_{B_\delta(\hat x)\times B_\delta(\hat y)}\Big(F(x,y,t_0-s)-F(x_0,y_0,t_0)\Big){\rm d}\Pi^s(x,y),\\
	  I_3(s):&=-\int_{B_\delta(\hat x)\times B_\delta(\hat y)}\Big(v(x,t_0-s)-v(x_0,t_0)\Big){\rm d}\Pi^s(x,y),\\
	   I_4(s):&=\int_{B_\delta(\hat x)\times B_\delta(\hat y)}\Big(\big(\eta_0(t_0 - s)-\eta_0(t)\big)+ \big(\eta_1(x)-\eta_1(x_0)\big)+\big( \eta_2(y)-\eta_2(y_0) \big)\Big){\rm d}\Pi^s(x,y).	  	
	   	\end{split}
	\end{equation}

The estimates to the integrals $I_1, I_2, I_3,I_4$ will be given in the following. For simplifying the notations, we denote by $B_1:=B_\delta(\hat x)$ and $B_2:=B_\delta(\hat y)$.

 For estimating $I_1(s)$, we recall the contraction property of $L^2$-Wasserstein distance (\cite[Corollary 1.4]{RS05}, namely
\begin{equation} \label{equ-7.10}
	\int_{M\times M} d^2(x,y){\rm d}\Pi^s(x,y)\ls  e^{-2Ks}d^2(x_0,y_0),\qquad \forall s>0.
\end{equation}
We claim that
\begin{equation}\label{equ-7.11}
	\limsup_{s\to 0^+}\frac{1}{s}\int_{(M\times M)\setminus (B_1\times B_2)}\Big|d^2(x,y)-d^2(x_0,y_0)\Big|{\rm d}\Pi^s(x,y)=0.
\end{equation}
Indeed, by the triangle inequality, we have
\begin{equation}\label{equ-7.12}
	\begin{split}
		|d^2(x,y)-d^2(x_0,y_0)|&=\big|\big(d(x,y)-d(x_0,y_0)\big)^2+2d(x_0,y_0)\cdot\big(d(x,y)-d(x_0,y_0)\big)\big|\\
		&\ls \big(d(x,x_0)+d(y,y_0)\big)^2+2d(x_0,y_0)\cdot\big(d(x,x_0)+d(y,y_0)\big)\\
		&\ls 2d^2(x,x_0)+2d^2(y,y_0)+2d(x_0,y_0)\cdot\big(d(x,x_0)+d(y,y_0)\big).\\
			\end{split}
\end{equation}
Noticing that the first and second marginals of $\Pi^s$ are $\mu_{x_0}^s$ and $\mu_{y_0}^s$ repectively, we obtain
\begin{equation*}
	\begin{split}
		\int_{(M\times M)\setminus (B_1\times B_2)}d^2(x,x_0){\rm d}\Pi^s(x,y)&=\int_{[(M\setminus B_1)\times M]\cup [B_1\times(M\setminus B_2) ] }d^2(x,x_0){\rm d}\Pi^s(x,y)\\
		&=\int_{M\setminus B_1}d^2(x,x_0)\dm^s_{x_0}(x)+ \int_{   B_1\times(M\setminus B_2)  }d^2(x,x_0){\rm d}\Pi^s(x,y)\\
		&\ls \int_{M\setminus B_1}d^2(x,x_0)\dm^s_{x_0}(x)+\int_{ M\setminus B_2}(2\delta)^2\dm^s_{y_0}(y),
			\end{split}
\end{equation*}
where we have used that $d(x,x_0)\ls 2\delta$ on $B_1$, since $x_0\in B_1$. Similarly, we get
\begin{equation*}
	\begin{split}
		\int_{(M\times M)\setminus (B_1\times B_2)}d^2(y,y_0){\rm d}\Pi^s(x,y) &\ls \int_{M\setminus B_2}d^2(y,y_0)\dm^s_{y_0}(y)+\int_{ M\setminus B_1}(2\delta)^2\dm^s_{x_0}(x),\\
	\int_{(M\times M)\setminus (B_1\times B_2)} d(x,x_0) {\rm d}\Pi^s(x,y) &\ls \int_{M\setminus B_1}d(x,x_0)\dm^s_{x_0}(x)+\int_{ M\setminus B_2}(2\delta)\dm^s_{y_0}(y),\\
	\int_{(M\times M)\setminus (B_1\times B_2)} d(y,y_0) {\rm d}\Pi^s(x,y) &\ls \int_{M\setminus B_2}d(y,y_0)\dm^s_{y_0}(y)+\int_{ M\setminus B_1}(2\delta)\dm^s_{x_0}(x).
				\end{split}
\end{equation*}
Integrating (\ref{equ-7.12}) over $(M\times M)\setminus (B_1\times B_2)$ and substituting above four inequalities, we obtain
\begin{equation*}
	\begin{split}
	&	\int_{(M\times M)\setminus (B_1\times B_2)}\Big|d^2(x,y)-d^2(x_0,y_0)\Big|{\rm d}\Pi^s(x,y) \\
	\ls &\ 2\int_{M\setminus B_1}d^2(x,x_0)\dm^s_{x_0}(x)+
	2\int_{M\setminus B_2}d^2(y,y_0)\dm^s_{y_0}(y)\\
	&\quad +2d(x_0,y_0) \cdot\left(\int_{M\setminus B_1}d(x,x_0)\dm^s_{x_0}(x)+
	\int_{M\setminus B_2}d(y,y_0)\dm^s_{y_0}(y)\right)\\
	&\quad +(8\delta^2+4\delta d(x_0,y_0))\left(\int_{ M\setminus B_1}1\dm^s_{x_0}(x)+\int_{ M\setminus B_2}1\dm^s_{y_0}(y)	\right).
			\end{split}
\end{equation*}
Now the claim (\ref{equ-7.11}) comes from Lemma \ref{lem-heat-ker} (by taking $\ell=0,1,2$ therein).

It is clear that
\begin{equation}\label{equ-7.13}
	\begin{split}
		\int_{(M\times M)\setminus (B_1\times B_2)}1{\rm d}\Pi^s(x,y)&=\int_{M\setminus B_1}1\dm^s_{x_0}(x)+ \int_{   B_1\times(M\setminus B_2)  }1{\rm d}\Pi^s(x,y)\\
		&\ls \int_{M\setminus B_1}1\dm^s_{x_0}(x)+ \int_{ M\setminus B_2  }1{\rm d}\mu^s_{y_0}(y)=o(s)
	\end{split}
\end{equation}
as $s\to 0^+.$

By combining with (\ref{equ-7.10}) and (\ref{equ-7.11}), we obtain
\begin{equation*}
	\begin{split}
		&\frac{1}{s}\int_{B_1\times B_2}\Big(d^2(x,y)-e^{-2Ks}d^2(x_0,y_0)\Big){\rm d}\Pi^s(x,y)  \\
		=\ &\frac{1}{s}\int_{B_1\times B_2}\Big(d^2(x,y)- d^2(x_0,y_0)\Big) {\rm d}\Pi^s(x,y) 	+\frac{1-e^{-2Ks}}{s}d^2(x_0,y_0) \int_{B_1\times B_2}1{\rm d}\Pi^s(x,y)\\
		=\ & \frac{1}{s}\int_{M\times M}\Big(d^2(x,y)- d^2(x_0,y_0)\Big) {\rm d}\Pi^s(x,y) +o(1)	+\frac{1-e^{-2Ks}}{s}d^2(x_0,y_0) \int_{B_1\times B_2}1{\rm d}\Pi^s(x,y)\\
		\ls \ &\frac{e^{-2Ks}-1}{s}d^2(x_0,y_0)+o(1)	+\frac{1-e^{-2Ks}}{s}d^2(x_0,y_0) \int_{B_1\times B_2}1{\rm d}\Pi^s(x,y)\\
		=\ & \frac{e^{-2Ks}-1}{s}d^2(x_0,y_0)\int_{(M\times M)\setminus (B_1\times B_2)}1{\rm d}\Pi^s(x,y)+o(1)\\
		=\ &o(1).		
			\end{split}
\end{equation*}
That implies
\begin{equation}\label{equ-7.14}
	\limsup_{s\to0^+} \frac{1}{s}I_1(s)\ls 0.
\end{equation}

For estimating  $I_2(s)$, we claim
\begin{equation}\label{equ-7.15}
	\limsup_{s\to0^+} \frac{1}{s}I_2(s)\ls 0.
\end{equation}

In the case of $F(x_0,y_0,t_0)=0$, by noticing that  $-F(x,y,t_0-s)\ls0$,  (\ref{equ-7.15}) holds trivially. In the following, we assume that $F(x_0,y_0,t_0)\not=0$.
Let us put
$$P=u(x, t_0-s), \quad Q=u(x_0,t_0), \quad R=u(y_0,t_0) \quad {\rm and}\quad S=u(y,t_0-s).$$
From Lemma \ref{lem-2.3} (2) and the notation (\ref{equ-6.10}), we have
$$\Big(F(x,y,t_0-s)-F(x_0,y_0,t_0)\Big)\cdot F(x_0,y_0,t_0)\gs -w_{Q_m,s}(x_0,x,t_0)-w_{Q_m,s}(y_0,y,t_0),$$
where $Q_m$ is the midder point of $Q$ and $R$ (i.e., $d_Y(Q_m,Q)=d_Y(Q_m,R)=d_Y(Q,R)/2$).
Integrating this over $B_1\times B_2$, we have
\begin{equation*}
	\begin{split}
		&\int_{B_1\times B_2}\Big(-F(x,y,t_0-s)+F(x_0,y_0,t_0)\Big){\rm d}\Pi^s(x,y) \cdot F(x_0,y_0,t_0)\\
		\ls&  \int_{B_1}w_{Q_m,s}(x_0,x,t_0)\dm_{x_0}^s(x)+\int_{B_2}w_{Q_m,s}(y_0,y,t_0)\dm_{y_0}^s(y).
	\end{split}
\end{equation*}
By using $F(x_0,y_0,t_0)>0$, (\ref{equ-7.5}) and
$$	\lim_{s\to0^+}\frac{1}{s}\int_{\Omega\setminus B_1}  w_{P,s}(x_0,x,t_0) \dm_{x_0}^s(x)=\lim_{s\to0^+}\frac{1}{s}\int_{\Omega\setminus B_2} w_{P,s}(y_0,y,t_0) \dm^s_{y_0}(y)= 0$$
(due to the fact that $w_{P,s}$ is bounded  and Lemma \ref{lem-heat-ker}),
 we conclude (\ref{equ-7.15}).

Notice that
\begin{equation*}
	\begin{split}
		I_3(s)&=  \int_{B_1\times B_2}\Big(-v(x,t_0-s) + v(x_0,t_0)\Big){\rm d}\Pi^s(x,y)	\\
		&=\int_{B_1}\Big(-v(x,t_0-s) + v(x_0,t_0)\Big)\dm_{x_0}^s(x)\cdot \mu_{y_0}^s(B_2).
\end{split}
\end{equation*}
 By using Lemma \ref{lem-6.2} to $-v(x,t)+v(x_0,t_0)$ and (\ref{thm-7.7+}), we have
\begin{equation}
	\label{equ-7.16}
	\limsup_{s\to0^+} \frac{I_3(s)}{s}\ls -\frac{\theta_0}{2}\cdot \liminf_{s\to0^+}  \mu^s_{y_0}(B_2)= -\frac{\theta_0}{2}.
	\end{equation}

By using (\ref{thm-7.6+}) and Lemma \ref{lem-6.2}, the same argument shows that
\begin{equation}
	\label{equ-7.17}
\limsup_{s\to0^+} \frac{I_4(s)}{s}\ls  \frac{\theta_0}{8}.
\end{equation}

Finally, the combination of (\ref{equ-7.8}), (\ref{equ-7.9}) and (\ref{equ-7.14})---(\ref{equ-7.17}) would imply
$0\ls -\theta_0/4$, which is impossible. Therefore, we have completed the proof.
 \end{proof}

\subsection{Lipschitz continuity in space variables}

We continue to assume that $u(x,t)\in W^{1,2}(\Omega\times(0,+\infty),Y)$ be a weak solution of harmonic map heat flow with a bounded initial data $u_0$, i.e., $u_0(\Omega)\subset \overline{ B_{M_0}(P_0)}$ for some $M_0>0$ and $P_0\in Y$.
We also assume that $u$ is continous in $\Omega\times(0,+\infty)$.
For some ball  $B_R(\bar x )$  with $B_{2R}(\bar x )\subset \Omega$ and some $T>0$, let $f_\varepsilon(x,t)$ be given in (\ref{equ-7.1}).

  For convenience, we denote by
\begin{equation}\label{equ-7.18}
	v_\varepsilon(x,t):=-f_\varepsilon(x,t),\quad \forall \varepsilon\in(0,\varepsilon_0),\ \  \forall (x,t)\in B_R(\bar x )\times (0,T).
\end{equation}
Then $0\ls v_\varepsilon(x,t)\ls 2M_0$ and
\begin{equation}\label{equ-7.19}
	 (\Delta-\partial_t)v_\varepsilon(x,t)\gs0\quad {\rm  on}\ \  B_R(\bar x )\times(0,T),
\end{equation}
in the sense of viscosity, and hence also in the sense of distributions (see \cite{Ish95}).

The classical Hamilton-Jacobi equation says that the temporal derivative of the Hamilton-Jacobi flow is equal to minus the square norm of the spatial gradient. The next lemma reminds the Hamilton-Jacobi equation.
\begin{lemma}
	\label{lem-7.7} For any $\varepsilon\in(0,\varepsilon_0)$, we have
	\begin{equation}
		\label{equ-7.20}
		\frac{v_\varepsilon(x,t)}{\varepsilon}\ls 2  e^{2|K|T} \cdot \Big({\rm lip}_{C_1\sqrt{\varepsilon}}u^t(x)\Big)^2
	\end{equation}
for almost all $(x,t)\in B_R(\bar x )\times(0,T)$, where the constant $C_1$ is given in (\ref{equ-7.2}).
	\end{lemma}

\begin{proof}
Let $(x,t)  \in B_R(\bar x )\times(0,T)$. If ${\rm lip}_{C_1\sqrt \varepsilon}u^t(x)=+\infty$, we have done. In the following, we assume that  ${\rm lip}_{C_1\sqrt \varepsilon}u^t(x)<+\infty$.

According to (\ref{equ-7.2}), there exists $y:=y_x\in B_{C_1\sqrt \varepsilon}(x)$ such that
\begin{equation}\label{equ-7.21}
	v_\varepsilon(x,t)=F(x,y,t)-\frac{e^{-2Kt}d^2(x,y)}{2\varepsilon}.
\end{equation}
If $y=x$, we have $v_\varepsilon(x,t)=0$, then (\ref{equ-7.20}) holds trivially. We can assume that $d(x,y):=s>0$.
By using (\ref{equ-5.6}) and $y\in B_{C_1\sqrt\varepsilon}(x)$, we have
\begin{equation}\label{equ-7.22}
	F(x,y,t)=\frac{F(x,y,t)}{s}s\ls {\rm lip}_{C_1\sqrt{\varepsilon}}u^t(x)\cdot s= {\rm lip}_{C_1\sqrt{\varepsilon}}u^t(x)\cdot d(x,y).
\end{equation}
By combining this with (\ref{equ-7.21}) and the fact that $v_\varepsilon\gs 0$, we obtain
$$\frac{e^{-2Kt}d^2(x,y)}{2\varepsilon}=F(x,y,t)-v_\varepsilon(x,t)\ls F(x,y,t)\ls {\rm lip}_{C_1\sqrt{\varepsilon}}u^t(x)\cdot d(x,y).$$
This is
$$d(x,y)\ls 2\varepsilon\cdot e^{2Kt} \cdot {\rm lip}_{C_1\sqrt{\varepsilon}}u^t(x)\ls 2\varepsilon\cdot e^{2|K|T} \cdot {\rm lip}_{C_1\sqrt{\varepsilon}}u^t(x).$$
Substituting into (\ref{equ-7.22}) and using (\ref{equ-7.21}) again, we conclude that
$$v_\varepsilon(x,t)\ls F(x,y,t)\ls {\rm lip}_{C_1\sqrt{\varepsilon}}u^t(x)\cdot d(x,y)\ls 2\varepsilon\cdot e^{2|K|T} \cdot \Big({\rm lip}_{C_1\sqrt{\varepsilon}}u^t(x)\Big)^2.$$
This is (\ref{equ-7.20}). The proof is finished.
\end{proof}

The following is the main result of this section.

\begin{theorem}\label{thm-7.8}
	Let $\Omega, Y$ be  as above and   let $u_0\in L^\infty(\Omega,Y)$ with the image $u_0(\Omega) \subset \overline{B_{M_0}(P_0)}$  for some $M_0>0$ and $P_0\in Y$.
	Suppose that $u(x,t)$ is a weak solution of harmonic map heat flow from $\Omega$ to $Y$ with initial data $u_0$. Let $B_{2R}(\bar x)\subset \Omega$ and $0<t_*<T<+\infty$ with $R<1$ and $R^2<t_*/2$. Then we have
	$$d_Y\big(u(x,t),u(y,t)\big)\ls C\cdot d(x,y),\quad \forall x,y\in B_{R/4}(\bar x),\ \ \forall t\in(t_*,T),$$
	for some  constant $C$ depending only on $n,K,R,M_0,t_*,T,L$ and $\int_{\Omega\times(t_*,T)}e_u\dm{\rm d}t$, where $L$ is given in (\ref{equ-2.10}).	
\end{theorem}

\begin{proof}
	 We put $\varepsilon_1:=\min\{\varepsilon_0/2, R/4\}$, where $\varepsilon_0$ is given in (\ref{equ-varepsilon}).
	
	Integrating (\ref{equ-7.20}) over $Q_{R/2}(\bar x,t_0):=B_{R/2}(\bar x)\times (t_0-R^2/4,t_0+R^2/4)$ and using (\ref{equ-5.7}), we have
\begin{equation*}
	\begin{split}
		\int_{Q_{R/2}(\bar x,t_0)}\frac{v_{\varepsilon}(x,t)}{\varepsilon}\dm(x){\rm d}t&\ls  2e^{2|K|T} \int_{Q_{R/2}(\bar x,t_0)}\Big({\rm lip}_{C_1\sqrt{\varepsilon}}u^t(x)\Big)^2\dm(x){\rm d}t\\
		&\ls   C_{n,K,R,T}\int_{t_0-R^2/4}^{t_0+R^2/4}\int_\Omega e_{u^t}(x)\dm(x){\rm d}t+C_{n,K,R,L}\\
		&\ls C_{n,K,R,T}\int_{\Omega\times(t_0/2,T)}e_u(x,t)\dm(x){\rm d}t+C_{n,K,R,L}:=\mathscr A.
	\end{split}
\end{equation*}
for any $\varepsilon\in(0,\varepsilon_1)$.
Noticing that $(\Delta-\partial_t) \frac{v_\varepsilon(x,t)}{\varepsilon} \gs0$ in the sense of distributions, we get
$$\sup_{Q_{R/4}(\bar x,t_0)}\frac{v_\varepsilon(x,t)}{\varepsilon}\ls C_{n,K,R}\int_{Q_{R/2}(\bar x,t_0)} \frac{v_\varepsilon(x,t)}{\varepsilon}\dm{\rm d}t\ls C\cdot \mathscr A$$
for all $\varepsilon\in(0,\varepsilon_1).$  From the definition of $v_\varepsilon$, we get
\begin{equation}
	\label{equ-7.23}
	\frac{F(x,y,t)}{\varepsilon}-\frac{e^{-2Kt}d^2(x,y)}{2\varepsilon^2}\ls C\mathscr A
\end{equation}for any $\varepsilon\in(0,\varepsilon_1)$ and any $(x,y,t)\in B_{R/4}(\bar x)\times\Omega\times(t_*,T)$.

Given any two point $x,y\in B_{R/4}(\bar x)$, if $d(x,y) < \varepsilon_1$, we take $\varepsilon=d(x,y)$ in (\ref{equ-7.23}) and get
$$\frac{F(x,y,t)}{d(x,y)}\ls C\mathscr A+e^{2|K|T}:=\mathscr A_1.$$
If $d(x,y) \geq \varepsilon_1$, we can choose a finite number of points $x_1=x, x_2,\cdots, x_\ell=y$ such that $d(x_i,x_{i+1})<\varepsilon$ for all $i=1,\cdots, \ell-1,$ and
$\sum_{j=1}^{\ell-1}d(x_j,x_{j+1})\ls d(x,y).$ By the triangle inequality, we get
$$F(x,y,t)\ls \sum_{j=1}^{\ell-1} F(x_j,x_{j+1},t)\ls \mathscr A_1\sum_{j=1}^{\ell-1} d(x_j,x_{j+1})\ls  \mathscr A_1\cdot d(x,y).$$
The proof is finished.
\end{proof}

By combining Theorem \ref{thm-5.1} and Theorem \ref{thm-7.8}, we conclude that the following local Lipschitz continuity holds for a weak solution of the harmonic map heat flow in space-time.
\begin{corollary}
	\label{coro-7.9}
	Let $\Omega\subset M$ be a bounded open domain, let $(Y,d_Y)$ be a $CAT(0)$ space, and let $u(x,t)$ be a weak solution of harmonic map heat flow from $\Omega$ to $Y$ with a bounded initial data $u_0$.    Then  $u(x,t)$ is locally Lipschitz continuous on $\Omega\times(0,+\infty).$
\end{corollary}
\begin{proof}
	This is a direct corollary of Theorem \ref{thm-5.1} and Theorem \ref{thm-7.8}.
\end{proof}

\section{Eell-Sampson-type Bochner Inequality}

In this section, we shall prove the Eell-Sampson-type inequality.
\begin{theorem}
	\label{thm-8.1}
	Let $\Omega\subset M$ be a bounded open domain, let $(Y,d_Y)$ be a $CAT(0)$ space, and let $u(x,t):\Omega\times(0,+\infty)\to Y$ be a weak solution of harmonic map heat flow with a bounded initial data $u_0$.    Then the function $ {\rm lip}_xu\in V_{2,\rm loc}(\Omega\times(0,+\infty))\cap L^\infty_{\rm loc}(\Omega\times(0,+\infty))$ and satisfies
	\begin{equation}\label{equ-8.1}
		(\Delta-\partial_t)({\rm lip}_xu)^2\gs 2|\nabla {\rm lip}_xu|^2+2K({\rm lip}_xu)^2
	\end{equation}
	on $\Omega\times(0,+\infty)$ in the sense of distributions, where ${\rm lip}_xu(x,t):={\rm lip}u^t(x)$.
	\end{theorem}

 Theorem \ref{thm-8.1} is a parabolic version of our previous joint work with Xiao Zhong \cite[Theorem 1.9]{ZZZ19} for harmonic maps into $CAT(0)$ spaces.  We will refine the arguments in the previous section by lifting the nonlinear ``Hamilton-Jacobi" flow from the usual index 2 to a higher order index $p$ to extract quantitative information.

We use the same notations as in the previous section. Let $M_0>0$ be a constant such that $u_0(\Omega)\subset \overline{B_{M_0}(P_0)}\subset Y$. Fix arbitrarily  $0<t_*<T<+\infty$, and any ball $B_R(\bar x)$ with $B_{2R}(\bar x)\subset\Omega$. Without loss of generality, we can assume that $R\ls1$ and $M_0\gs1$.  According to Corollary \ref{coro-7.9}, $u\in Lip(B_{2R}(\bar x)\times(t_*,T))$. We denote by $\tilde L$ a Lipschitz constant of $u$ in $B_{2R}(\bar x)\times(t_*,T)$.

 Let $p\gs 2$ and $\varepsilon>0$, we define
\begin{equation}
	\label{equ-8.2}
	f_{\varepsilon,p}(x,t):=\inf_{y\in \Omega}\left\{\frac{e^{-pKt}\cdot d^p(x,y)}{p\varepsilon^{p-1}}-F(x,y,t)\right\},\quad \forall (x,t)\in B_R(\bar x )\times(0,T),
 \end{equation}
where $	F(x,y,t)=d_Y\big(u(x,t),u(y,t)\big)$.  When $p=2$,  $f_{\varepsilon,2}$ is the same $f_\varepsilon$ in the previous section. It is clear that
$$-2M_0\ls f_{\varepsilon,p}(x,t)\ls0,\qquad \forall(x,t)\in B_R(\bar x )\times(0,T).$$
Since $p^{\frac{1}{p-1}}\ls 2 $ for all $p\gs2$ (and assumption $R\ls1$ and $M_0\gs1$), we have
$$\frac{e^{-\frac{p}{p-1}|K|T}\cdot R^{\frac{p}{p-1}}}{(4M_0)^{\frac{1}{p-1}}}\cdot \frac{1}{p^{\frac{1}{p-1}}}\gs \frac{e^{-2|K|T}R^2}{8M_0}:=\varepsilon_0,\qquad \forall p\gs2. $$
Therefore, for any  $\varepsilon\in(0,\varepsilon_0)$ and any $(x,t)\in B_R(\bar x)\times (0,T)$, the infimum in the (\ref{equ-8.2}) can be attained at a point $y\in B_{2R}(\bar x)$. 
Similar to the key Lemma \ref{lem-7.6} in the previous section, we have the following.
\begin{lemma}\label{lem-8.2}
Under the above notations, for each $\varepsilon \in (0, \varepsilon_0)$ and any $p\gs 2$, $p\in\mathbb N$, the function $f_{\varepsilon,p}$ is a viscosity supersolution of the heat equation  in $B_{R}(\bar x )\times(t_*,T)$.
\end{lemma}

\begin{proof}
	The proof is closely parallel to that of Lemma \ref{lem-7.6}, with the contraction property of the $L^2$-Wasserstein distance replaced by that of the $L^p$-Wasserstein distance. Hence, we follow the same argument and highlight the necessary modifications.
	
	The first step (i) is the same, with the function $H_1(x,y,t)$ replaced by
	$$\tilde H_1(x,y,t):=\frac{e^{-pKt d^p(x,y)}}{p\varepsilon^{p-1}}-F(x,y,t)-v(x,t) +\eta_0(t) +\eta_1(x)+\eta_2(y).$$
	It attains a minimum at $(x_0,y_0,t_0)$ in $B_{\delta}(\hat x)\times B_\delta(\hat y)\times(\hat t-\delta,t_0]$, and the asmyptotic mean value inequality (\ref{equ-7.5}) is available at $x_0$ and $y_0$.
		
	 In step (ii), by using the same notations of $\mu_{x_0}^s$ and $\mu_{y_0}^s$, we take the optimal coupling of them with respect to the $L^p$-Wasserstein distance, denoted by $\tilde\Pi^s:=\tilde\Pi^s_{(x_0,y_0)}$. The contraction property of $L^p$-Wasserstein distance( \cite[Corollary 1.4]{RS05}) states
\begin{equation} \label{equ-8.3}
	\int_{M\times M} d^p(x,y){\rm d}\tilde\Pi^s(x,y)\ls  e^{-pKs}d^p(x_0,y_0),\qquad \forall s>0.
\end{equation}
Similar to (\ref{equ-7.8}), the minimum property of $(x_0,y_0,t_0)$ implies
\begin{equation}
\label{equ-8.4}
\begin{split}
	0\ls &\liminf_{s\to0^+}\frac{1}{s}\int_{B_\delta(\hat x)\times B_\delta(\hat y)}\Big( \tilde H_1(x,y,t_0-s)-\tilde H_1(x_0,y_0,t_0)\Big){\rm d}\tilde\Pi^s(x,y)\\
	& := \liminf_{s\to0^+}\frac{1}{s}\Big(\tilde I_1(s)+I_2(s)+I_3(s)+I_4(s)\Big),
	\end{split}
	\end{equation}
where $I_2, I_3, I_4$ are exactly the same as in (\ref{equ-7.9}), and $I_1$ replaced by $\tilde I_1$:
\begin{equation*}
	 \tilde I_1(s):= \frac{1}{p\varepsilon^{p-1}}\int_{B_\delta(\hat x)\times B_\delta(\hat y)}\Big(e^{-pK(t_0-s)}d^p(x,y)-e^{-pKt_0}d^p(x_0,y_0)\Big){\rm d}\tilde\Pi^s(x,y).	  	
	 	\end{equation*}	
To conclude $\limsup_{s\to0^+} \tilde I_1(s)/s\ls0$, we need only replaced (\ref{equ-7.10}) by (\ref{equ-8.3}), and  (\ref{equ-7.12}) by
\begin{equation*}
\begin{split}
	|d^p(x,y)-d^p(x_0,y_0)|&\ls |d(x,y)-d(x_0,y_0)|^p+\sum_{k=1}^p\binom{p}{k}d^k(x_0,y_0)|(d(x,y)-d(x_0,y_0)|^{p-k}\\
	&\ls |d(x,x_0)+d(y,y_0)|^p+\sum_{k=1}^p\binom{p}{k}d^k(x_0,y_0)|(d(x,x_0)+d(y,y_0)|^{p-k},
	\end{split}	
\end{equation*}
 where we used the elementary equality $(a+b)^p-a^p= \sum_{k=1}^{p} \binom{p}{k} a^{p-k} b^k, \forall p\in\mathbb N,$ to $b=d(x_0,y_0)$ and $a=d(x,y)-d(x_0,y_0)$. (This is the reason why we assume that $p$ is an integer.)
 	
The rest of the arguments are exactly the same as in Lemm \ref{lem-7.6}.  Thus, we have finished the proof.
\end{proof}

For any fixed $t>0$, we define a new metric on $M$ by
$$d_t(x,y): =e^{-Kt}d(x,y),\quad \forall x,y\in M.$$
Under this metric, we can rewriten the function $ f_{\varepsilon,p}(\cdot,t)$ as
$$f_{\varepsilon,p}(x,t):=\inf_{y\in \Omega}\left\{\frac{  d_t^p(x,y)}{p\varepsilon^{p-1}}-d_Y(u^t(x),u^t(y))\right\},\quad \forall  x \in B_R(\bar x).$$

\begin{lemma}\label{lem-8.3}
	There exists a positive constant, still denoted by  ${\varepsilon}_0>0$, depending only on $K, T, R, M_0$,  such that  for each $\varepsilon\in(0,\varepsilon_0)$, there hold
	\begin{enumerate}
	
	\item  $f_{\varepsilon,p}\in Lip(B_{R}(\bar x )\times(t_*,T))$.	
\item  $ -f_{\varepsilon,p}\ls C_{K,T, \tilde L} \cdot \varepsilon$ on $B_{R}(\bar x)\times(t_*,T)$ for some constant $C_{K,T, \tilde L}>0$, depending only on $K, T, \tilde L$ (independent of  $p$ and $\varepsilon$).
			\item Let $q>1$ with $1/q+1/p=1$.
	Fix any $t\in (t_*,T)$. Then we have that, for almost all $x\in B_R(\bar x)$,
	\begin{equation}\label{equ-8.5}
	\lim_{\varepsilon\to0^+}\frac{f_{\varepsilon,p}(x,t)}{\varepsilon}=-\frac{[\widetilde{\rm lip}u^t(x)]^q}{q},
		\end{equation}
		where for a map $v:\Omega\to Y$, $\widetilde{\rm lip}v(x)$  is the point-wise Lipschitz constant of $v$ at $x$ with respected to $d_t(x,y)$, i.e.
		$$\widetilde{\rm lip}v(x)=\limsup_{d_t(y,x)\to0}\frac{d_Y\big(v(x),v(y)\big)}{d_t(x,y)}.$$		
	\end{enumerate}
\end{lemma}
\begin{proof}
		For (1),   given any $(x,t),(x',t')\in B_{R}(\bar x )\times(t_*,T)$, by choosing $\tilde{\varepsilon}_0>0$ small enough, we pick one $y'\in B_{2R}(\bar x )$ such that
	$$f_{\varepsilon,p}(x',t')=\frac{e^{-pKt'}d^p(x',y')}{p\varepsilon^{p-1}}-F(x',y',t').$$
	Therefore, by (\ref{equ-8.2}), we get
	\begin{equation*}
		\begin{split}
			f_{\varepsilon,p}(x,t)-f_{\varepsilon,p}(x',t')&\ls \frac{e^{-pKt}d^p(x,y')}{p\varepsilon^{p-1}}-F(x,y',t)-\left(\frac{e^{-pKt'}d^p(x',y')}{p\varepsilon^{p-1}}-F(x',y',t')	\right)\\
			&\ls \frac{e^{-pKt} d^p(x,y')-e^{-pKt'}d^p(x',y')}{p\varepsilon^{p-1}}+2\tilde L(d(x,x')+|t-t'|)\\
			&\ls C_{p,\varepsilon,K, \tilde L,R} \cdot
			(d(x,x')+|t-t'|)		 		\end{split}
	\end{equation*}
	where we have used the fact that $u$ is $\tilde L$-Lipschitz on $B_{2R}(\bar x )\times(t_*,T)$ . By the symmetry of $(x,t)$ and $(x',t')$, we conclude (1).

The assertion (2) is contained in the proof of \cite[Lemma 4.1 (3)]{ZZZ19}. For completeness, we give the proof here.
Given any $(x,t) \in B_{R}(\bar x )\times(t_*,T)$, we choose one $y=y_{x,t} \in B_{2R}(\bar x ) $ such that
\begin{equation}\label{equ-8.6}
	f_{\varepsilon,p}(x ,t )=\frac{e^{-pKt }d^p(x ,y)}{p\varepsilon^{p-1}}-F(x,y,t).
\end{equation}
	Since $f_{\varepsilon,p}\ls0$ and $F(x,y,t)\ls \tilde L d(x,y),$ we get
	$$\frac{e^{-pKt}d^{p}(x,y)}{p\varepsilon^{p-1}}\ls  \tilde L d(x,y)$$
	and then $$d(x,y)\ls (e^{p|K|T}\tilde L)^{\frac{1}{p-1}}p^{\frac{1}{p-1}}\cdot \varepsilon\ls c_{K,T,\tilde L}\cdot \varepsilon$$ for some positive constant  $c_{K,T,\tilde L}$ depending only on $K, T, \tilde L$ (independent of  $p$ and $\varepsilon$), where we have used $p^{\frac{1}{p-1} }\ls 2$ for all $p\gs2 $.	Thus $$-f_{\varepsilon,p}(x,t)\ls F(x,y,t)\ls \tilde Ld(x,y)\ls (c_{K,T,\tilde L} \cdot \tilde L)\varepsilon.$$
	 This proves (2).

Fix any $t\in (t_*,T)$.	Since $u^t$ is a Lipschitz map on $\Omega$ under the metric $d$, it is also Lipschitz continuous under the new metric $d_t$. Now the result (3) is \cite[Lemma 4.4]{ZZZ19}.
	\end{proof}

Notice that
\begin{equation}
	\label{equ-8.7}
\widetilde{\rm lip}u^t(x)=\limsup_{d_t(y,x)\to0}\frac{d_Y\big(v(x),v(y)\big)}{e^{-Kt} d(x,y)}=e^{Kt}\cdot{\rm lip}u^t(x).
\end{equation}

Now we are in the position to prove Theorem \ref{thm-8.1}.
\begin{proof}[Proof of Theorem  \ref{thm-8.1}] According to Corollary \ref{coro-7.9}, we have  $ {\rm lip}_xu\in   L_{\rm loc}^\infty(\Omega\times (0,\infty)).$

As the statement is local, it suffices to prove that the function $$  {\rm lip}_xu(x)\in V_2(B_{R/2}(\bar x)\times (t_*+R^2/4,T-R^2/4)) $$
and satisfies (\ref{equ-8.1}) on $B_{R/2}(\bar x)\times (t_*+R^2/4,T-R^2/4)$ in the sense of distributions. Without loss of generality, we also assume that $T-t_*\gs  R^2$.

Fix any $p\in\mathbb N, p\gs2$ . Denote by $$g_\varepsilon(x,t):= \frac{-f_{\varepsilon,p}(x,t)}{\varepsilon}.$$
From Lemma \ref{lem-8.2}, we get for each $\varepsilon\in(0,\varepsilon_0)$ that
$$(\Delta-\partial_t)g_\varepsilon(x,t)\gs0,$$
on $B_{R(\bar x)}\times(t_*,T)$ in the sense of viscosity, and hence also in the sense of distributions. By the Caccioppoli inequality for  positive subsolutions of the heat equations, we have
\begin{equation}
	\label{equ-8.8}
\int_{t_*+R^2/4}^{T-R^2/4}\int_{B_{R/2}(\bar x)}|\nabla g_{\varepsilon}|^2\dm{\rm d}t\ls c_{n,K,R} \int_{t_*}^{T}\int_{B_{R}(\bar x)}g_\varepsilon^2\dm{\rm d}t.
\end{equation}
Noticing that $\|g_\varepsilon(x,t)\|_{L^\infty(B_R(\bar x)\times(t_*,T))}\ls C_{K,T, \tilde L}$ (by Lemma \ref{lem-8.3} (2)), the family of functions $\{g_\varepsilon\}_{0<\varepsilon<\varepsilon_0}$ is bouunded in $V_2(B_{R/2}(\bar x)\times (t_*+R^2/4,T-R^2/4))$ and
\begin{equation}
	\label{equ-8.9}\|g_\varepsilon(x,t)\|^2_{V_2}\ls C_{K,T, \tilde L} (1+c_{n,K,R}|B_R(\bar x)| (T-t_*)).
\end{equation}
By letting $\varepsilon\to0^+$ and  using Lemma \ref{lem-8.3} (3) and (\ref{equ-8.7}), we obtain that for each $q  \in(1,2]$ with $1/q+1/p=1$, the function $(x,t)\mapsto e^{qKt}\frac{[{\rm lip}u^t]^q(x)}{q} $ is also in $ V_2(B_{R/2}(\bar x)\times (t_*+R^2/4,T-R^2/4))$,  and
\begin{equation}
	\label{equ-8.10}
\|e^{qKt}\frac{[{\rm lip}u^t]^q}{q}\|^2_{V_2}\ls C_{K,T, \tilde L} (1+c_{n,K,R}|B_R(\bar x)| (T-t_*)),
\end{equation}
and it is a subsolution of the heat equation on $B_{R/2}(\bar x)\times (t_*+R^2/4,T-R^2/4)$ in the sense of distributions.

Finally, letting $p\to+\infty$, and hence $q\to1^+$ in (\ref{equ-8.10}),  we conclude that $(x,t)\mapsto e^{Kt} {\rm lip}u^t(x)$ is in $V_2(B_{R/2}(\bar x)\times (t_*+R^2/4,T-R^2/4))$ and
satifies
$$(\Delta-\partial_t)(e^{Kt}{\rm lip}u^t)\gs0$$
  on $B_{R/2}(\bar x)\times (t_*+R^2/4,T-R^2/4)$ in the sense of distributions. This is
  $$(\Delta -\partial_t){\rm lip}u^t-K\cdot {\rm lip}u^t\gs 0$$
  on $B_{R/2}(\bar x)\times (t_*+R^2/4,T-R^2/4)$ in the sense of distributions, since $e^{Kt}\gs e^{-|K|T}>0$.
 This is (\ref{equ-8.1}), and then the proof is finished.
 \end{proof}

\begin{proof}[Proof of Theorem \ref{main-thm}]
 This is the combination of Corollary \ref{coro-7.9} and Theorem \ref{thm-8.1}.
\end{proof}

 \appendix
   \renewcommand{\appendixname}{Appendix~\Alph{section}}

       \section{The Parabolic Perturbation Lemma via ABP Estimates.}
 In this appendix, we give a proof of the parabolic perturbation lemma, Lemma \ref{lem-7.5}.
 Let us recall the Alexandrov-Bakelman-Pucci (ABP) estimates for parabolic equations in \cite{ACP11}.
\begin{defn}
	\label{def-a1}
	Let $\Omega\subset \mathbb R^n$ and $\Omega_T=\Omega\times (0,T)$, and let $r>0$. Given $f:\Omega_T\to\mathbb R$, the \emph{parabolic upper contact set } of $f$ of scale $r$ is defined as
	\begin{equation}
		\label{equ-a1}
		\Gamma^+_r(f):=\left\{(x,t)\in \Omega_T\Big| \ \begin{aligned}
			& \exists \xi\in B_r(0)  \ {\rm such\ that}\\
			&  f(z,s)-\ip{\xi}{z} \ls f(x,t)-\ip{\xi}{x},\ \
			  \forall (z,s)\in \Omega\times(0,t]
		\end{aligned}\right\}.
	\end{equation}
	\end{defn}
Equivalent, a point $(x,t)\in \Gamma_r^+(f)$ if and only if there exists $\xi\in B_r(0)$ such that $(x,t)$ is one of the maximum point of function $f(z,s)-\ip{\xi}{z} $ in $\Omega\times(0,t]$.

 In \cite{ACP11},    Argiolas, Charro, and   Peral have proven an ABP estimate for some nonlinear parabolic equations in divergence form. For our purpose, we consider only the following linear operator. Let
 $$Lf:= \partial_j( a_{ij} \partial_if)$$
  be a uniformly elliptic operator on $\Omega$ with elliptic constants $0<\lambda<\Lambda<\infty$, and  the coefficients $a_{ij}\in C^1(\Omega)$.
 \begin{theorem}[\cite{ACP11}]
 	\label{thm-a2}Let $\Omega_T=\Omega\times(0,T)\subset \mathbb R^{n+1}$ be a bounded open domain and $g\in L^{n+1}(\Omega_T)\cap C(\Omega_T)$. Consider $f\in C(\overline \Omega_T)$ wich satisfies
 	$$(L-\partial_t)f\gs- g(x,t)$$
 	in $\Omega_T$ in the viscosity sense. Assume that  $\sup _{\Omega} f > 0$ and $\sup_{\partial_{\mathcal P}\Omega_T}f^+ =0$. Then  the following ABP estimate holds:  For any $r\in (0,\sup_\Omega f)$,
 	 	\begin{equation}
 		\label{equ-a2}r^{n+1}\ls  C_{n,d,\lambda,\Lambda}\cdot \|g^+\|_{L^{n+1}(\Gamma^+_{r/d}(f))},
 	\end{equation}
 	where $\partial_{\mathcal P}{\Omega_T}$ is the parabolic boundary of $\Omega_T$, $d:={\rm diam}(\Omega)$ and $f^+:=\max\{0,f\}$.
 \end{theorem}
\begin{proof}
	This is contained in the proof of \cite[Theorem 2]{ACP11}. The first step is to consider the case when $f\in C^{2,1}(\Omega_T)\cap C(\overline{\Omega_T})$, we refer to \cite[Eq. (17) and Eq. (18)]{ACP11} to get the estimate  (\ref{equ-a2}). The second step  deals with the general case when $f\in C(\overline{\Omega_T})$, see \cite[the last three lines on page 887]{ACP11} for estimate (\ref{equ-a1}).
		\end{proof}

\begin{corollary}
	\label{coro-a3} Let $\Omega_T=\Omega\times(0,T)\subset \mathbb R^{n+1}$ be a bounded open domain, $0\in \Omega$, and $g\in L^{n+1}(\Omega_T)\cap C(\Omega_T)$. Consider $f\in C(\overline \Omega_T)$ wich satisfies
 	$$(L-\partial_t)f\ls  g(x,t)$$
 	in $\Omega_T$ in the viscosity sense. Suppose that $f$ attains its minimum at unique point $(0,\hat t)\in \Omega_T $. Assume that $E$  is a measurable set with full measure in $\Omega_T$. Then for any $\delta\in(0,1)$, there exist a vector $\xi\in B_\delta(0)$ and a point $(v_0,t_0)\in E\cap \Omega_T$ such that the function  $f(v)+\ip{\xi}{v}$ attains one of its minimum in  $\Omega\times(0,t_0]$ at $(x_0,t_0).$
 	\end{corollary}

\begin{proof}Let  $\delta>0$. Since $E$ has full measure, it suffices to show   $|\Gamma_\delta^+(-f)|>0.$

 We choose a small open set $V$ such that $(0,\hat t)\in V\Subset \Omega_{T}$. As $f\in C(\overline{\Omega_T}\setminus V)$ and that $\overline{\Omega_T}\setminus V$ is a bounded closed set, $f$ has a minimum in $\overline{\Omega_T}\setminus V$. The assumption that $(0,\hat t)$ is the unique minimum point of $f$ on $\Omega_T$ implies
$$\min_{\overline{\Omega_T}\setminus V} f> f(0,\hat t).$$

We define a function
$$\tilde f(x,t):=-f(x,t)+\frac{\min_{\overline{\Omega_T}\setminus V} f+f(0,\hat t)}{2}\quad {\rm on} \ \Omega_T.$$
Then we have the following:
\begin{enumerate}
	\item $\tilde f\in C(\overline \Omega_T)$,
	\item $(L-\partial_t)\tilde f\gs -g$ in $\Omega_T$ in the viscosity sense,
	\item $\sup_{\partial_{\mathcal P}U}\tilde f^+ =0,\quad $ (since  $\max_{\overline{\Omega_T}\setminus V}\tilde f=\frac{-\min_{\overline{\Omega_T}\setminus V} f+f(0,\hat t)}{2}<0$ and $\partial_{\mathcal P}U\subset\partial U\subset \overline{\Omega_T}\setminus V$).
	\item $\sup_{U}\tilde f>0,\quad $ (since $\tilde f(0,\hat t)=\frac{\min_{\overline{\Omega_T}\setminus V} f-f(0,\hat t)}{2}>0$).
	
 \end{enumerate}	
Therefore, by using Theorem \ref{thm-a2} to $\tilde f$, we get that $\|(-g)^+\|_{L^{n+1}(\Gamma_\delta^+(\tilde f))}>0$. In particular, $|\Gamma_\delta^+(\tilde f)|>0.$	 By the definition (\ref{equ-a1}),  it is clear that $\Gamma_\delta^+(-f)=\Gamma_\delta^+(\tilde f)$.  Hence $|\Gamma_\delta^+(-  f)|>0$ and the proof is finished.
  \end{proof}

\begin{remark}\label{rem-a4}
The ABP estimates for elliptic equations have been extended to the setting of Riemannian manifolds \cite{WZ13}, and very recently to the setting of $RCD$ metric measure spaces \cite{MS22+,Han26}.  It would be very interesting to generalize the ABP estimate for parabolic equations to the setting of Riemannian manifolds and $RCD$ metric measure spaces. 	
\end{remark}
Now we will prove the parabolic perturbation lemma as follows.
\begin{proof}
	[Proof of Lemma \ref{lem-7.6}]

Let $\{x^1,x^2,\cdots,x^n\}$ and $\{y^1,y^2,\cdots, y^n\}$ be tow local coordinate systems near $\hat x$ and $\hat y$ respectively. Let $g_{ij}(x)$ and $g_{ij}(y)$ be the Riemannian metrics under $\{x^i\}$ and $\{y^i\}$ respectiively.
Let $\delta_0>0$ with $\delta_0\ls \min \{inj(\hat x),inj(\hat y),1\}$, where $inj(\hat x)$ is the injective radius of $\hat x$, such that
\begin{equation}\label{equ-a3}
 	|g_{ij}(x)-\delta_{ij}|+|g_{ij}(y)-\delta_{ij}|\ls \frac{1}{10n^2},\qquad |\partial_kg_{ij}(x)|+ |\partial_kg_{ij}(y)|\ls 1,
 	 \end{equation}
 for all  $(x,y)\in B_{\delta_0}(\hat x)\times B_{\delta_0}(\hat y)$.

As $(\hat x,\hat y,\hat t)$ is a local minimum point of $h(x,y,t)$ on $(U\times V)_T$, we can also assume that it is a minimum point of $h$ on $B_{\delta_0}(\hat x)\times B_{\delta_0}(\hat y)\times (\hat t-\delta_0,\hat t+\delta_0)$.

Let $\exp_{\hat x}:  B_{\delta_0}(0) \subset \mathbb R^n\to B_{\delta_0}(\hat x)$ and $\exp_{\hat y}:B_{\delta_0}(0)\subset \mathbb R^n\to B_{\delta_0}(\hat y)$ be the exponential maps centered at $\hat x$ and $\hat y$ respectively.

For any $\delta\in(0,\delta_0)$, the function
\begin{equation}
	\label{equ-a4}\tilde h(v,w,t):=h(\exp_{\hat x}(v), \exp_{\hat y}(w),t)+ \delta(t-\hat t)^2 + \delta|v|^2+\delta|w|^2
\end{equation}
defined on $B_{\delta_0}(0) \times  B_{\delta_0}(0)  \times (\hat t-\delta_0,\hat t+\delta_0)$,  satisfies the following:
\begin{enumerate}
	\item $\tilde h$ has a \emph{unique} minimum at point $(0,0,\hat t)$, on  $B_{\delta_0}(0) \times B_{\delta_0}(0)\times (\hat t-\delta_0,\hat t+\delta_0)$.
	\item $\tilde h\in C(\overline{B_{\delta_0}(0) \times B_{\delta_0}(0) \times (\hat t-\delta_0,\hat t+\delta_0)})$
	\item $$(L-\partial_t)\tilde h\ls C+c_{n,K}\delta$$
	on $ B_{\delta_0}(0) \times B_{\delta_0}(0) \times (\hat t-\delta_0,\hat t+\delta_0)$ in the viscosity sense for some $c_{n,K}>0$, where $K$ is a lower bound of Ricci curvature, and  $L$ is the elliptic operator given by $\Delta^{(2)}$:
	$$L=\frac{1}{\sqrt g}\frac{\partial}{\partial x^i}\left(\sqrt g g^{ij}\frac{\partial}{\partial x^j}\right)(x)+ \frac{1}{\sqrt g}\frac{\partial}{\partial y^i}\left(\sqrt g g^{ij}\frac{\partial}{\partial y^j}\right)(y),$$
	where $g={\rm det}(g_{ij})$, and $(g^{ij})=(g_{ij})^{-1}.$
		\end{enumerate}
		
Let $E$ be a set with full measure. Then the set
$$\tilde E:=\{(v,w,t): (\exp_{\hat x}(v),\exp_{\hat y}(w),t)\in E\}$$
 has full measure in $B_{\delta_0}(0) \times B_{\delta_0}(0)  \times (\hat t-\delta_0,\hat t+\delta_0)$.
		By applying Corollary \ref{coro-a3},  there exist $(\xi_x,\xi_y)\in \mathbb R^n\times \mathbb R^n$ with $|\xi_x|^2+|\xi_y|^2<\delta^2$ and  a point $(v_0,w_0,t_0)\in\tilde E\cap( B_{\delta_0}(0) \times  B_{\delta_0}(0) \times (\hat t-\delta_0,\hat t+\delta_0))$ such that
		$$\tilde h(v,w,s)+\ip{\xi_x}{v}+\ip{\xi_y}{w}\gs \tilde h(v_0,w_0,t_0)+\ip{\xi_x}{v_0}+\ip{\xi_y}{w_0}$$
		for all $(v,w,s)\in  B_{\delta_0}(0)  \times B_{\delta_0}(0)  \times (\hat t-\delta_0,t_0].$
		
We put
\begin{equation}
	\label{equ-a5}
 \eta_0 := \delta (t-\hat t)^2, \qquad    \tilde \eta_1(v):=\delta|v|^2	+\ip{\xi_x}{v},  \qquad   \tilde\eta_2(w):=\delta|w|^2	+\ip{\xi_y}{w},
\end{equation}	
on  $B_{\delta_0}(0)\times B_{\delta_0}(0) \times (\hat t-\delta_0,\hat t+\delta_0)$. Then, by noticing that $|\xi_x|^2+ |\xi_y|^2\ls  \delta^2$, we have
\begin{equation}
	\label{equ-a6} |\partial_j\tilde \eta_1(v)|+|\partial_i\partial_j\tilde \eta_1(v)|\ls c_1\delta, \qquad |\partial_j\tilde \eta_2(w)|+|\partial_i\partial_j\tilde \eta_2(w)|\ls c_1\delta
	\end{equation}
	on  $B_{\delta_0}(0) \times B_{\delta_0}(0)$,	where the constant $c_1>0$ depends only on $n$.

Finally, we pull back these functions by
\begin{equation}
	\eta_1(x):=\tilde\eta_1\big(\exp_{\hat x}^{-1}(x)\big),\qquad \eta_2(y):=\tilde\eta_2\big(\exp_{\hat y}^{-1}(y)\big).
	\end{equation}	
By combining  (\ref{equ-a3}), (\ref{equ-a6}) and $\Delta=\frac{1}{\sqrt{g}}	\partial_i(\sqrt gg^{ij}\partial_j)$,
we have
 \begin{equation}
|\partial_t \eta_0(t)| \ls c_2\delta, \qquad 	|\Delta\eta_1(x)|\ls c_2\delta, \qquad 	|\Delta\eta_2(y)|\ls c_2\delta, 	\end{equation}	
on $B_{\delta_0}(\hat x)\times B_{\delta_0}(\hat y)\times (\hat t-\delta_0,\hat t+\delta_0)$,	where $c_2$ depends only on $n$.
		Now the point
	$$(x_0,y_0,t_0)=\big(\exp_{\hat x}(v_0), \exp_{\hat y}(w_0),	t_0)$$
	and the functions $\eta_0(t), \eta_1(x),\eta_2(y)$ meet all conclusions of this lemma. The proof is finished.	
 \end{proof}

\end{document}